\documentclass[10pt]{amsart}

\usepackage{amssymb,amsmath,amsthm}
\usepackage[all]{xy}
\usepackage{eucal}

\usepackage{color}

\theoremstyle{plain}
\newtheorem{thm}[subsection]{Theorem}

\newtheorem{prop}[subsection]{Proposition}
\newtheorem{assumption}[subsection]{Basic Assumption}

\newtheorem{cofibrancy}[subsection]{Cofibrancy Condition}

\theoremstyle{definition}
\newtheorem{defn}[subsection]{Definition}

\theoremstyle{remark}
\newtheorem{rem}[subsection]{Remark}

\makeatletter
\let\c@equation\c@subsection

\makeatother

\newcommand{\ZZ}{{ \mathbb{Z} }}

\newcommand{\DD}{{ \mathsf{D} }}

\newcommand{\Ho}{{ \mathsf{Ho} }}

\newcommand{\sSet}{{ \mathsf{sSet} }}

\newcommand{\Mod}{{ \mathsf{Mod} }}
\newcommand{\ModR}{{ \mathsf{Mod}_\capR }}

\newcommand{\Spectra}{{ \mathsf{Sp}^\Sigma }}

\newcommand{\Cat}{{ \mathsf{Cat} }}

\newcommand{\M}{{ \mathsf{M} }}

\newcommand{\SymSeq}{{ \mathsf{SymSeq} }}

\newcommand{\Set}{{ \mathsf{Set} }}

\newcommand{\Alg}{{ \mathsf{Alg} }}

\newcommand{\LL}{{ \mathsf{L} }}
\newcommand{\RR}{{ \mathsf{R} }}

\newcommand{\Cube}{{ \mathsf{Cube} }}

\newcommand{\TQ}{{ \mathsf{TQ} }}
\newcommand{\TAQ}{{ \mathsf{TAQ} }}
\newcommand{\K}{{ \mathsf{K} }}

\newcommand{\coAlg}{{ \mathsf{coAlg} }}
\newcommand{\res}{{ \mathsf{res} }}
\newcommand{\BK}{{ \mathsf{BK} }}
\newcommand{\CGHaus}{{ \mathsf{CGHaus} }}

\newcommand{\AlgJ}{{ \Alg_J }}
\newcommand{\coAlgK}{{ \coAlg_\K }}
\newcommand{\AlgO}{{ \Alg_\capO }}

\newcommand{\capX}{{ \mathcal{X} }}
\newcommand{\capY}{{ \mathcal{Y} }}

\newcommand{\capO}{{ \mathcal{O} }}
\newcommand{\capR}{{ \mathcal{R} }}

\newcommand{\capP}{{ \mathcal{P} }}

\newcommand{\powerset} {{ \capP }}

\newcommand{\Sk}{{ \mathrm{sk} }}

\newcommand{\ev}{{ \mathrm{ev} }}

\newcommand{\id}{{ \mathrm{id} }}
\newcommand{\op}{{ \mathrm{op} }}

\newcommand{\Smash}{{ \,\wedge\, }}

\newcommand{\tensor}{{ \otimes }}

\newcommand{\tensorcheck}{{ \check{\tensor} }}

\newcommand{\tensordot}{{ \dot{\tensor} }}

\newcommand{\wequiv}{{ \ \simeq \ }}

\newcommand{\Iso}{{  \ \cong \ }}
\newcommand{\Equal}{{ \ = \ }}
\newcommand{\rarrow}{{ \rightarrow }}
\newcommand{\larrow}{{ \leftarrow }}

\newcommand{\function}[3]{{ {#1}\colon\thinspace{#2}\rarrow{#3} }}
\newcommand{\functionlong}[3]{{ {#1}\colon\thinspace{#2}\longrightarrow{#3} }}

\DeclareMathOperator*{\colim}{colim}
\DeclareMathOperator*{\holim}{holim}

\DeclareMathOperator{\Map}{Map}

\DeclareMathOperator{\BAR}{Bar}
\DeclareMathOperator{\Cobar}{Cobar}

\DeclareMathOperator{\Tot}{Tot}
\DeclareMathOperator{\Sing}{Sing}

\DeclareMathOperator{\Hombold}{\mathbf{Hom}}
\DeclareMathOperator{\hombold}{\mathbf{hom}}
\DeclareMathOperator{\hofib}{hofib}
\DeclareMathOperator{\iter}{iterated}

\DeclareMathOperator{\im}{im}

\title[Derived {K}oszul duality and $\TQ$-homology completion]{Derived {K}oszul duality and $\TQ$-homology completion of structured ring spectra}

\author{Michael Ching}
\author{John E. Harper}

\address{Department of Mathematics and Statistics, Amherst College, Amherst, MA, 01002, USA}

\email{mching@amherst.edu}

\address{Department of Mathematics, The Ohio State University, Newark, 1179 University Dr, Newark, OH 43055, USA}
\email{harper.903@math.osu.edu}

\begin{document}

\begin{abstract}
Working in the context of symmetric spectra, we consider any higher algebraic structures that can be described as algebras over an operad $\capO$. We prove that the fundamental adjunction comparing $\capO$-algebra spectra with coalgebra spectra over the associated comonad $\K$, via topological Quillen homology (or $\TQ$-homology), can be turned into an equivalence of homotopy theories by replacing $\capO$-algebras with the full subcategory of $0$-connected $\capO$-algebras. This resolves in the affirmative the $0$-connected case of a conjecture of Francis-Gaitsgory.

This derived Koszul duality result can be thought of as the spectral algebra analog of the fundamental work of Quillen and Sullivan on the rational homotopy theory of spaces, and the subsequent $p$-adic and integral work of Goerss and Mandell on cochains and homotopy type---the following are corollaries of our main result: (i) $0$-connected $\capO$-algebra spectra are weakly equivalent if and only if their $\TQ$-homology spectra are weakly equivalent as
derived $\K$-coalgebras, and (ii) if a $\K$-coalgebra spectrum is
$0$-connected and cofibrant, then it comes from the $\TQ$-homology spectrum
of an $\capO$-algebra. We construct the spectral algebra analog of the unstable Adams spectral sequence that starts from the $\TQ$-homology groups $\TQ_*(X)$ of an $\capO$-algebra $X$, and prove that it converges strongly to $\pi_*(X)$ when $X$ is $0$-connected.
\end{abstract}

\maketitle

\section{Introduction}

Some of the most well-developed and powerful tools for studying spaces are those that relate their two primary invariants---homotopy groups and homology groups. The aim of this paper is to resolve in the affirmative the $0$-connected case of a conjecture of Francis-Gaitsgory \cite{Francis_Gaitsgory}, and subsequently, to prove the spectral algebra analogs of several foundational results known for spaces concerning completion, the unstable Adams spectral sequence, Quillen and Sullivan theory on (co)homology and homotopy type, and Koszul duality phenomena; e.g., in the new context of ring spectra, $E_n$ ring spectra, and more generally, algebras parametrized by operads $\capO$ of spectra. At the heart of our attack on the Francis-Gaitsgory conjecture is the construction of new tools and methods for studying spectral algebras---these powerful new tools for higher algebra are in the spirit of the work of Dundas \cite{Dundas_relative_K_theory} and Goodwillie \cite{Goodwillie_calculus_2}, where Dundas exploits to great effect the magic of Goodwillie's powerful higher cubical diagram theorems for spaces.

With the rapid developments in derived algebraic geometry (e.g., Francis \cite{Francis_tangent_complex}, Lurie \cite{Lurie_dag, Lurie_higher_algebra, Lurie_higher_topos}, and To\"en-Vezzosi \cite{Toen_Vezzosi_hag_1, Toen_Vezzosi_hag_2}) and algebraic $K$-theory (e.g., Hesselholt-Madsen \cite{Hesselholt_Madsen} and Rognes \cite{Rognes_logarithmic}), and the central nature of topological Quillen homology as a primary notion of a ``homology'' invariant that is sensitive to the algebraic-topological structure, the development of these new tools for spectral algebras will have rich potential payoffs for any applications exploiting structured ring spectra.

\subsection{Topological Quillen homology}
In \cite[II.5]{Quillen}, Quillen defined a notion of \emph{homology} for objects, in a wide variety of homotopical settings, to be the total left derived functor of abelianization, if it exists. For a given algebraic structure, Quillen homology comes in several flavors, depending on whether one is working in a reduced or relative setting (Section \ref{sec:quillen_homology_reduced_versus_relative}). Quillen homology, in the reduced setting of augmented commutative algebras, was studied and exploited to great effect by H.R. Miller \cite{Miller} in his proof of the Sullivan conjecture. In Miller's reduced setting, abelianization of an augmented commutative algebra $A$ is the indecomposable quotient $QX$ of the augmentation ideal $X$ of $A$, and Quillen homology in this reduced setting is the derived indecomposable quotient (or derived indecomposables for short) $\LL Q(X)$ of $X$. The subsequent work of Goerss \cite{Goerss_f2_algebras} is an extensive development, exploiting a homotopy point of view, of several properties of Quillen homology, in the same reduced setting used by Miller \cite{Miller}, of augmented commutative $\mathbb{F}_2$-algebras. In this context, Goerss \cite[4.3, 4.12]{Goerss_f2_algebras} also provides a comparison between Quillen homology in the reduced and relative settings, using the suspension $\Sigma A$ of an augmented commutative $\mathbb{F}_2$-algebra $A$ (see \cite[p. 51]{Goerss_f2_algebras}) exploited in Miller \cite[Section 5]{Miller}; see also Dwyer-Spalinski \cite[11.3]{Dwyer_Spalinski} for a useful discussion.

Topological Quillen homology, which we often refer to as $\TQ$-homology for short, is the precise topological analog of Quillen homology in the higher algebra setting of structured ring spectra. For a given algebraic structure on spectra (resp. $\capR$-modules), $\TQ$-homology  comes in several flavors, depending on the setting in which one is working (e.g., in a reduced or relative setting).

\subsection{$\TQ$-homology in the relative setting}
\label{sec:quillen_homology_reduced_versus_relative}
There is a close connection, which is worth pointing out, between Quillen homology in the reduced setting and Quillen homology in the relative setting. Quillen homology was originally developed by Andr\'e \cite{Andre} and Quillen \cite{Quillen_rings}  for commutative algebras, in the relative setting, which is now called Andr\'e-Quillen homology. In the relative setup, Quillen homology of a commutative algebra takes the form of a cotangent complex. Basterra \cite[4.1]{Basterra} originally studied the following construction of topological Quillen homology, in the relative context. The \emph{topological Andr\'e-Quillen homology} of $A$ over $R$ is the $A$-module
\begin{align}
\label{eq:topological_andre_quillen_of_A_over_R}
  \TAQ^R(A):=\TAQ(A/R):=\LL Q(\RR I(A\Smash_R A))
\end{align}
where $R$ (resp. $A$) is a cofibrant commutative $S$-algebra (resp. $R$-algebra) and $S$ is the sphere spectrum. In other words, \eqref{eq:topological_andre_quillen_of_A_over_R} is the derived indecomposables (of the derived augmentation ideal) of $A\Smash_R A$; this is the \emph{cotangent complex} of $A$ relative to $R$ and is denoted $\mathsf{L\Omega}_R A$ in Basterra-Mandell \cite[Section 2]{Basterra_Mandell}; for a useful introduction to Quillen homology, from a homotopy point of view, see Goerss \cite{Goerss_f2_algebras}, Goerss-Hopkins \cite{Goerss_Hopkins}, Goerss-Schemmerhorn \cite{Goerss_Schemmerhorn}, and Miller \cite{Miller}. For a useful introduction to $\TAQ^R$, together with a development of its connections to algebraic $\K$-theory, see Rognes \cite[Section 10]{Rognes_logarithmic}.

\subsection{$\TQ$-homology in the reduced setting}
Topological Quillen homology in the reduced setting is where derived Koszul duality naturally lives and is precisely the setting of a conjecture in Francis-Gaitsgory \cite{Francis_Gaitsgory} on homotopy pro-nilpotent $\capO$-algebras. The main aim of this paper is to resolve in the affirmative the $0$-connected case of their conjecture (Theorem \ref{MainTheorem}).

\begin{assumption}
\label{assumption:commutative_ring_spectrum}
From now on in this paper, we assume that $\capR$ is any commutative ring spectrum; i.e., we assume that $\capR$ is any commutative monoid object in the category $(\Spectra,\tensor_S,S)$ of symmetric spectra (e.g., Hovey-Shipley-Smith \cite{Hovey_Shipley_Smith} and Schwede \cite{Schwede_book_project}). We work mostly in the category of $\capR$-modules which we denote by $\ModR$.
\end{assumption}

Henceforth we work in the reduced spectral algebra setting of $\capO$-algebras, where $\capO$ is an operad of $\capR$-modules with trivial $0$-ary operations; this is, of course, equivalent to working in the augmented context, but for technical reasons is better behaved. In this setting, the $\TQ$-homology spectrum of an $\capO$-algebra $X$ is the derived indecomposables spectrum $\LL Q(X)$ of $X$ (Section \ref{sec:TQ_homology_completion}).

\subsection{Hurewicz maps for $\TQ$-homology}
It turns out that $\TQ$-homology of $\capO$-algebras behaves remarkably like the ordinary homology of spaces. This intuition has been powerfully developed and exploited in the work of Goerss \cite{Goerss_f2_algebras}, Goerss-Hopkins \cite{Goerss_Hopkins}, and Miller \cite{Miller}. In recent work, Harper-Hess \cite{Harper_Hess} prove a spectral algebra analog of Serre's finiteness theorem for spaces and Ching-Harper \cite{Ching_Harper} prove a spectral algebra analog of Goodwillie's \cite{Goodwillie_calculus_2} higher Blakers-Massey theorems.

Similar to the context of spaces and ordinary homology, $\TQ$-homology of $\capO$-algebras has a comparison map, which we call the Hurewicz map,
\begin{align}
\label{eq:TQ_hurewicz_map_introduction}
  \pi_*(X)\rarrow \TQ_*(X)
\end{align}
of graded abelian groups. Assume that $X$ is a cofibrant $\capO$-algebra (e.g, let $X'$ be any $\capO$-algebra and denote by $X\wequiv X'$ its cofibrant replacement). Then this comparison map comes from a Hurewicz map (Section \ref{sec:TQ_homology_completion}) on the level of $\capO$-algebras of the form
\begin{align}
\label{eq:TQ_hurewicz_map_point_set_level_introduction}
  X\rarrow UQ(X)
\end{align}
and applying $\pi_*$ recovers the map \eqref{eq:TQ_hurewicz_map_introduction}.

Once one has such a Hurewicz map on the level of $\capO$-algebras, it is natural to form a cosimplicial resolution (or Godement resolution) of $X$ with respect to $\TQ$-homology of the form
\begin{align}
\label{eq:TQ_homology_resolution_derived_functor_form_introduction}
\xymatrix{
  X\ar[r] &
  \TQ(X)\ar@<-0.5ex>[r]\ar@<0.5ex>[r] &
  \TQ^2(X)
  \ar@<-1.0ex>[r]\ar[r]\ar@<1.0ex>[r]\ar@<-2.0ex>[l] &
  \TQ^3(X)\cdots\ar@<-2.5ex>[l]\ar@<-3.5ex>[l]
  }
\end{align}
In other words, iterating the $\capO$-algebra level Hurewicz map \eqref{eq:TQ_hurewicz_map_point_set_level_introduction} results in a cosimplicial resolution of $X$  that can be thought of as encoding the spectrum level co-operations on the $\TQ$-homology spectra. The associated homotopy spectral sequence (Section \ref{sec:outline_of_the_argument}) is the precise spectral algebra analog of the unstable Adams spectral sequence of a space; see Bousfield-Kan \cite{Bousfield_Kan_spectral_sequence}, and the subsequent work of Bendersky-Curtis-Miller \cite{Bendersky_Curtis_Miller} and Bendersky-Thompson \cite{Bendersky_Thompson}.

As pointed out in the work of Hess \cite{Hess}, resolutions of the form \eqref{eq:TQ_homology_resolution_derived_functor_form_introduction} fit into an adjunction \eqref{eq:fundamental_adjunction_comparing_AlgO_with_coAlgK} comparing $\capO$-algebras with coalgebras over the associated comonad $\K$; this comonad naturally arises from the adjunction defining $\TQ$-homology---it formally drops out of the homotopy theory---and can be thought of as encoding co-operations on the $\TQ$-homology spectra (Section \ref{sec:TQ_homology_completion}). More precisely, applying $\holim_\Delta$ to \eqref{eq:TQ_homology_resolution_derived_functor_form_introduction}, regarded as a map of cosimplicial $\capO$-algebras, recovers the derived unit map associated to  the fundamental adjunction \eqref{eq:fundamental_adjunction_comparing_AlgO_with_coAlgK} that appears in a conjecture of Francis-Gaitsgory \cite{Francis_Gaitsgory}. The right-hand side of the resulting derived unit map
\begin{align*}
  X\rarrow X^\wedge_\TQ
\end{align*}
is the $\TQ$-homology completion of an $\capO$-algebra $X$ (Section \ref{sec:TQ_homology_completion}); this is the precise structured ring spectrum analog of the completions and localizations of spaces originally studied in Sullivan \cite{Sullivan_mit_notes, Sullivan_genetics}, and subsequently in Bousfield-Kan \cite{Bousfield_Kan} and Hilton-Mislin-Roitberg \cite{Hilton_Mislin_Roitberg}; there is an extensive literature---for a useful introduction, see also Bousfield \cite{Bousfield_localization_spaces}, Dwyer \cite{Dwyer_localizations}, and May-Ponto \cite{May_Ponto}

\subsection{The main result}
It is natural to ask the question of when the fundamental adjunction \eqref{eq:fundamental_adjunction_comparing_AlgO_with_coAlgK} comparing $\capO$-algebras with $\K$-coalgebras can be modified to turn it into an equivalence of homotopy theories. In \cite[3.4.5]{Francis_Gaitsgory}, J. Francis and D. Gaitsgory  conjecture that replacing $\capO$-algebras with the full subcategory of homotopy pro-nilpotent $\capO$-algebras will do the trick. In recent work of Harper-Hess \cite{Harper_Hess}, it is proved that every 0-connected $\capO$-algebra is homotopy pro-nilpotent; i.e., is the homotopy limit of a tower of nilpotent $\capO$-algebras.

In this paper we shall prove the following theorem, resolving in the affirmative the 0-connected case of the Francis-Gaitsgory conjecture; it implies that a $0$-connected $\capO$-algebra $X$ can be recovered from its $\TQ$-homology spectrum $\LL Q(X)$ together with its $\K$-coalgebra structure, but our result is much stronger---the homotopy theories are equivalent.

\begin{thm}
\label{MainTheorem}
Let $\capO$ be an operad in $\capR$-modules with trivial $0$-ary operations. Assume that $\capR,\capO$ are $(-1)$-connected. The fundamental adjunction \eqref{eq:fundamental_adjunction_comparing_AlgO_with_coAlgK} comparing $\capO$-algebras to $\K$-coalgebras, via topological Quillen homology,
\begin{align}
\label{eq:fundamental_adjunction_comparing_AlgO_with_coAlgK}
\xymatrix{
  \AlgO\ar@<0.5ex>[r]^-{Q} & \coAlgK\ar@<0.5ex>[l]^-{\lim_\Delta C}
}
\end{align}
induces an equivalence of homotopy theories, after restriction to the full subcategories of $0$-connected cofibrant objects. More precisely:
\begin{itemize}
\item[(a)] If $Y$ is a $0$-connected cofibrant $\K$-coalgebra, then the natural derived $\K$-coalgebra map of the form
\begin{align*}
  \LL Q\bigl(\holim\nolimits_\Delta C(Y)\bigr)\xrightarrow{\wequiv} Y
\end{align*}
is a weak equivalence; i.e., the derived counit map associated to the fundamental adjunction \eqref{eq:fundamental_adjunction_comparing_AlgO_with_coAlgK} is a weak equivalence (Definition \ref{defn:derived_counit_map}).
\item[(b)] If $X$ is a $0$-connected $\capO$-algebra, then the natural $\capO$-algebra map of the form
\begin{align*}
  X\xrightarrow{\wequiv}\holim\nolimits_\Delta C(\LL Q(X))= X^\wedge_\TQ
\end{align*}
is a weak equivalence; i.e., the derived unit map associated to the fundamental adjunction \eqref{eq:fundamental_adjunction_comparing_AlgO_with_coAlgK} is a weak equivalence.
\end{itemize}
Here, we assume that the operad $\capO$ satisfies Cofibrancy Condition \ref{CofibrancyCondition}; this mild cofibrancy condition is not restrictive in any way (Remark \ref{rem:no_loss_of_generality}).
\end{thm}

In particular, Theorem \ref{MainTheorem}(a) means that the $\TQ$-homology spectrum functor $\LL Q$ is homotopically essentially surjective on $0$-connected objects, and as shown in Arone-Ching \cite{Arone_Ching_classification} and elaborated in Section \ref{sec:homotopy_theory_K_coalgebras} for the context of this paper, it follows from Theorem \ref{MainTheorem}(b) that $\LL Q$ is homotopically fully faithful on $0$-connected objects; in other words, the $\TQ$-homology spectrum functor $\LL Q$ induces a weak equivalence
\begin{align*}
  \LL Q\colon\thinspace\Map_\AlgO(X,Y)\xrightarrow{\wequiv}\Map_\coAlgK\bigl(\LL Q(X),\LL Q(Y)\bigr)
\end{align*}
of mapping spaces, where $X$ is cofibrant (resp. $Y$ is fibrant) in $\AlgO$; equivalently, the cofibrancy (resp. fibrancy) conditions can be dropped by replacing the mapping complex $\Map_{\AlgO}$ with the Dwyer-Kan homotopy function complex $\Map^h_{\AlgO}$.

One may view the derived Koszul duality result in Theorem \ref{MainTheorem} as a spectral algebra analog of the fundamental work of Quillen \cite{Quillen_rational} and Sullivan \cite{Sullivan_infinitesimal} on the rational homotopy theory of spaces, and the subsequent $p$-adic and integral work of Goerss \cite{Goerss_simplicial_chains} and Mandell \cite{Mandell, Mandell_cochains_homotopy_type} on cochains and homotopy type; see also the earlier $p$-adic work of Kriz \cite{Kriz} for spaces.

\begin{rem}
Closely related examples of such duality phenomena have been uncovered in a variety of special cases and settings. For instance, (i) Moore duality \cite{Moore} between associative algebras and coalgebras, (ii) Quillen \cite{Quillen_rational}, Schlessinger-Stasheff \cite{Schlessinger_Stasheff}, and Sullivan \cite{Sullivan_infinitesimal} between differential graded Lie algebras and differential graded commutative coalgebras, (iii) Ginzburg-Kapranov \cite{Ginzburg_Kapranov} for differential graded binary quadratic operads, (iv) Getzler-Jones \cite{Getzler_Jones} for $n$-algebras and $n$-coalgebras, (v) Fresse \cite{Fresse_Koszul_duality} for differential graded $E_n$ operads, and (vi) Ayala-Francis \cite{Ayala_Francis} and Lurie \cite{Lurie_higher_algebra} for $E_n$ operads, among others; see Remark \ref{rem:elaboration_on_accessibility_of_K} where we further elaborate this range of connections.
\end{rem}

\begin{rem}
\label{rem:elaboration_on_accessibility_of_K}
The comonad $\K$ is quite accessible from a homotopy point of view. For instance, suppose that $Y$ is a cofibrant $\K$-coalgebra and assume for notational simplicity that $\tau_1\capO=I$, where $I$ denotes the initial operad. Then there is a zigzag of weak equivalences of $\capR$-modules
\begin{align*}
  \K(Y)\xleftarrow{\wequiv}|\BAR(J,\capO,Y)|
  \xrightarrow{\wequiv}B(\capO)\circ(Y)
  \Equal \coprod_{t\geq 1}B(\capO)[t]\Smash_{\Sigma_t}Y^{\wedge t}
\end{align*}
where $B(\capO)$ denotes the bar construction $|\BAR(I,\capO,I)|$ defined in \cite[5.30]{Harper_bar_constructions} (compare \cite{Ching_bar_constructions, Ching_bar_cobar_duality, Harper_Hess}). It follows that if $\capR, \capO$ are $(-1)$-connected and $Y$ is $0$-connected, then there is a zigzag of weak equivalences
\begin{align*}
  \K(Y)\wequiv B(\capO)\circ(Y)
  \wequiv \prod_{t\geq 1}B(\capO)[t]\Smash_{\Sigma_t}Y^{\wedge t}
\end{align*}
and hence a $\K$-coalgebra structure on $Y$ consists of, up to a zigzag of weak equivalences, a collection of suitable morphisms of the form
\begin{align*}
  Y\rarrow B(\capO)[t]\Smash_{\Sigma_t}Y^{\wedge t},\quad\quad t\geq 1.
\end{align*}
In other words, a $\K$-coalgebra structure on $Y$ can be thought of as a kind of \emph{divided power} generalized coalgebra with $t$-ary co-operations parametrized by $B(\capO)[t]$; here the ``divided power'' structure corresponds to the fact that the target of these morphisms are the coinvariants rather than the corresponding invariants. In the notation of Francis-Gaitsgory \cite[Section 6.2.5]{Francis_Gaitsgory}, the comonad $\K$ is called the Koszul dual comonad associated to the monad $\capO$ (i.e., the free algebra monad associated to the operad $\capO$).

Theorem \ref{MainTheorem} can therefore be viewed as a spectral algebra derived Koszul duality between $\capO$-algebras and divided power generalized coalgebras with $t$-ary co-operations parametrized by $B(\capO)[t]$; this is a spectral algebra analog (and generalization) of a variety of results in the differential graded setting. For instance, (i) Ginzburg-Kapranov \cite{Ginzburg_Kapranov}, Fresse \cite{Fresse}, and Moore \cite{Moore} between non-unital dg algebras and coalgebras: if $\capO$ denotes the operad parametrizing non-unital associative algebra spectra, then $\K$-coalgebras are the spectrum analogs of non-unital differential graded coassociative coalgebras, (ii) Ginzburg-Kapranov \cite{Ginzburg_Kapranov}, Fresse \cite{Fresse}, Quillen \cite{Quillen_rational}, Schlessinger-Stasheff \cite{Schlessinger_Stasheff} and Sullivan \cite{Sullivan_infinitesimal} between differential graded Lie algebras and differential graded cocommutative coalgebras: if $\capO$ denotes the operad parametrizing non-unital commutative algebra spectra, then $\K$-coalgebras are the spectrum analogs of differential graded Lie coalgebras, and (iii) Fresse \cite{Fresse_Koszul_duality} for differential graded $E_n$ algebras: if $\capO$ denotes the operad parametrizing non-unital $E_n$ algebra spectra, then $\K$-coalgebras are the spectrum analogs of non-unital differential graded $E_n$ coalgebras.

Theorem \ref{MainTheorem} provides evidence for the spectrum analog of the self-duality of the stable $E_n$ operad; compare Fresse \cite{Fresse_Koszul_duality}. Ayala-Francis \cite{Ayala_Francis} and Lurie \cite{Lurie_higher_algebra} have developed versions of Koszul duality for $E_n$ algebras in which the dual of an $E_n$ algebra lies in the category of $E_n$ coalgebras. Our result, on the other hand, sets up a duality between $E_n$ algebras and divided power generalized coalgebras with $t$-ary co-operations parametrized by $B(E_n)[t]$. Theorem \ref{MainTheorem} is therefore evidence for the conjecture that $B(E_n)$, associated to $E_n$, is equivalent to the Spanier-Whitehead dual of the operad $E_n$, suitably shifted in dimension.
\end{rem}

\subsection{Corollaries of the main result}

The following three results are corollaries of Theorem \ref{MainTheorem}. The following classification theorem shows that a pair of $0$-connected $\capO$-algebras are weakly equivalent if and only if their associated $\TQ$-homology spectra are weakly equivalent as derived $\K$-coalgebras.

\begin{thm}[Classification theorem]
\label{thm:classification}
A pair of $0$-connected $\capO$-algebras $X$ and $Y$ are weakly equivalent if and only if the $\TQ$-homology spectra $\LL Q(X)$ and $\LL Q(Y)$ are weakly equivalent as derived $\K$-coalgebras.
\end{thm}

The following theorem provides a characterization of when a $\K$-coalgebra arises, up to a zigzag of derived $\K$-coalgebra weak equivalences, as the $\TQ$-homology spectrum of an $\capO$-algebra. It shows that if a $\K$-coalgebra is $0$-connected and cofibrant, then it comes from the $\TQ$-homology spectrum of an $\capO$-algebra.

\begin{thm}[Characterization theorem]
\label{thm:characterization}
A cofibrant $\K$-coalgebra is weakly equivalent to the $\TQ$-homology spectrum of a $0$-connected $\capO$-algebra if and only if it is $0$-connected.
\end{thm}

The following theorem shows that $\TQ$-homology spectra do more than distinguish equivalence classes of $\capO$-algebras, they also distinguish equivalence classes of maps; compare, for instance, Mandell \cite[0.1]{Mandell_cochains_homotopy_type} in the case of integral cochains of a space. It is most clearly stated in terms of cofibrant $\capO$-algebras, but this is simply a matter of exposition and is no loss of generality (Remark \ref{rem:cofibrant_replacement_reminder_to_the_reader}).

\begin{thm}[Classification of maps theorem]
\label{thm:classification_maps}
Let $X,Y$ be cofibrant $\capO$-algebras. Assume that $Y$ is $0$-connected and fibrant.
\begin{itemize}
\item[(a)] Given any map $\phi$ in $[Q(X),Q(Y)]_\K$, there exists a map $f$ in $[X,Y]$ such that $\phi=Q(f)$.
\item[(b)] For each pair of maps $f,g$ in $[X,Y]$, $f=g$ if and only if $Q(f)=Q(g)$ in the homotopy category of $\K$-coalgebras.
\end{itemize}
\end{thm}

The following relatively weak cofibrancy condition is exploited in the proof of our main result.

\begin{rem}
\label{rem:no_loss_of_generality}
We do not lose any generality, since the comparison theorem in Harper-Hess \cite[3.26, 3.30]{Harper_Hess} shows that the operad $\capO$ can always be replaced by a weakly equivalent operad $\capO'$ that satisfies this cofibrancy condition; in fact, it satisfies a stronger $\Sigma$-cofibrancy condition. Most operads appearing in homotopy theoretic settings in mathematics already satisfy Cofibrancy Condition \ref{CofibrancyCondition} and therefore require no replacement;  for instance, this condition is satisfied by every operad in simplicial sets that is regarded as an operad in $\capR$-modules by adding a disjoint basepoint and tensoring with $\capR$ (\cite[4.1]{Harper_Hess}).
\end{rem}

\begin{cofibrancy}
\label{CofibrancyCondition}
If $\capO$ is an operad in $\capR$-modules, consider the unit map $\function{\eta}{I}{\capO}$ of the operad $\capO$ (see \cite[2.16]{Harper_Hess}) and assume that  $I[r]\rarrow\capO[r]$ is a flat stable cofibration (see \cite[7.7]{Harper_Hess}) between flat stable cofibrant objects in $\ModR$ for each $r\geq 0$.
\end{cofibrancy}

\begin{rem}
This is the same as assuming that $I[1]\rarrow\capO[1]$ is a flat stable cofibration in $\ModR$ and $\capO[r]$ is flat stable cofibrant in $\ModR$ for each $r\geq 0$. It can be thought of as the spectral algebra analog of the following cofibrancy condition: if $X$ is a pointed space, assume that the unique map $*\rightarrow X$ of pointed spaces is a cofibration.
\end{rem}

\subsection{Organization of the paper}

In Section \ref{sec:outline_of_the_argument} we outline the argument of our main result. The proof naturally breaks up into five subsidiary results. In Section \ref{sec:simplicial_structures} we recall some preliminaries on simplicial structures on $\capO$-algebras and verify that the natural transformations associated to $\TQ$-homology mesh well with the simplicial structures. In Section \ref{sec:TQ_homology_completion} we setup the cosimplicial cobar constructions associated to $\TQ$-homology completion and the fundamental adjunction; these lie at the technical heart of the paper. In Section \ref{sec:homotopy_limit_towers} we develop the needed results on towers of $\capO$-algebras associated to the truncation filtration of the simplicial category $\Delta$ and develop the properties that we need. In Section \ref{sec:homotopy_theory_K_coalgebras} we setup the homotopy theory of $\K$-coalgebras, which plays a central role in both the statement and proof of our main result. Section \ref{sec:cubical_diagrams_homotopical_analysis} is where we prove the connectivity estimates that we need. Finally in Section \ref{sec:proofs} we conclude with the proofs of several technical results needed in the paper.

\subsection*{Acknowledgments}

The second author would like to thank Bill Dwyer, Emmanuel Farjoun, Kathryn Hess, Rick Jardine, Jim McClure, Haynes Miller, and Crichton Ogle for useful suggestions and remarks and Greg Arone, Andy Baker, Clark Barwick, Kristine Bauer, Mark Behrens, Martin Bendersky, Lee Cohn, Bjorn Dundas, Benoit Fresse, Paul Goerss, Jens Hornbostel, Mark Hovey, Brenda Johnson, Niles Johnson, Tyler Lawson, Assaf Libman, Mike Mandell, Guido Mislin, Dmitri Pavlov, Luis Pereira, Nath Rao, Jakob Scholbach, Stefan Schwede, Dev Sinha, and Dennis Sullivan for helpful comments. The second author is grateful to Haynes Miller for a stimulating and enjoyable visit to the Department of Mathematics at the Massachusetts Institute of Technology in early spring 2015, and for his invitation which made this possible. The first author was supported in part by National Science Foundation grant DMS-1308933.

\section{Outline of the argument}
\label{sec:outline_of_the_argument}

We will now outline the proof of our main result (Theorem \ref{MainTheorem}). It naturally breaks up into five subsidiary results: one for the homotopical analysis of the derived unit map, and four associated to the homotopical analysis of the derived counit map. As corollaries we also obtain strong convergence results for the associated homotopy spectral sequences.

\begin{rem}
\label{rem:cofibrant_replacement_reminder_to_the_reader}
It is worth recalling here that the total left derived functor $\LL Q$ (Section \ref{sec:TQ_homology_completion}) can be calculated on any $\capO$-algebra $X$ as follows: $\LL Q(X)\wequiv Q(X')$ where $X'$ is a cofibrant replacement of $X$; hence in our arguments below it suffices to assume that an $\capO$-algebra $X$ is cofibrant. We make this assumption whenever it serves to clarify the statements or the constructions involving $Q(X)$; in all instances, the reader may simply drop the cofibrancy assumption and replace $Q(X)$ with $\LL Q(X)$ everywhere.
\end{rem}

\subsection{Homotopical analysis of the derived unit map}

Assume that $X$ is a cofibrant $\capO$-algebra. Then the $\TQ$-homology resolution \eqref{eq:TQ_homology_resolution_derived_functor_form_introduction} is precisely the coaugmented cosimplicial $\capO$-algebra
\begin{align*}
  X\rarrow C(QX)
\end{align*}
It follows that the $\TQ$-homology completion $X^{\wedge}_\TQ$ of $X$ can be expressed as
\begin{align*}
  X^\wedge_\TQ
  \wequiv\holim\nolimits_{\Delta}C(QX)
\end{align*}

The following observation is the first step in our attack on analyzing the cosimplicial resolution of an $\capO$-algebra with respect to $\TQ$-homology; it points out the homotopical significance of the coface $n$-cubes $\capX_n$. This is proved in Carlsson \cite{Carlsson} and Sinha \cite{Sinha_cosimplicial_models} in the contexts of spectra and spaces, respectively, and the same argument verifies it in our context; for the convenience of the reader, we include a concise proof in Section \ref{sec:cubical_diagrams_homotopical_analysis}.

\begin{prop}
\label{prop:homotopy_fiber_calculation}
Let $Z$ be a cosimplicial $\capO$-algebra coaugmented by $\function{d^0}{Z^{-1}}{Z^0}$. If $n\geq 0$, then there are natural zigzags of weak equivalences
\begin{align*}
  \hofib(Z^{-1}\rarrow\holim\nolimits_{\Delta^{\leq n}}Z)
  \wequiv
  (\iter\hofib)\capX_{n+1}
\end{align*}
where $\capX_{n+1}$ denotes the canonical $(n+1)$-cube associated to the coface maps of
\begin{align*}
\xymatrix{
  Z^{-1}\ar[r]^-{d^0} &
  Z^0\ar@<-0.5ex>[r]_-{d^1}\ar@<0.5ex>[r]^-{d^0} &
  Z^1\ \cdots\ Z^n
}
\end{align*}
the $n$-truncation of $Z^{-1}\rarrow Z$. We sometimes refer to $\capX_{n+1}$ as the \emph{coface} $(n+1)$-cube associated to the coaugmented cosimplicial $\capO$-algebra $Z^{-1}\rarrow Z$.
\end{prop}

The following proposition proved in Section \ref{sec:cubical_diagrams_homotopical_analysis} gives the connectivity estimates that we need.

\begin{prop}
\label{prop:connectivities_for_total_homotopy_fiber_homotopy_completion_cube_no_zigzags}
Let $X$ be a cofibrant $\capO$-algebra and $n\geq 1$. Denote by $\capX_n$ the coface $n$-cube associated to the cosimpicial $\TQ$-homology resolution $X\rarrow C(QX)$ of $X$. If $X$ is $0$-connected, then the total homotopy fiber of $\capX_n$ is $n$-connected.
\end{prop}

\begin{rem}
It is important to note that the total homotopy fiber of an $n$-cube of $\capO$-algebras is weakly equivalent to its iterated homotopy fiber, and in this paper we use the terms interchangeably; we use the convention that the iterated homotopy fiber of a $0$-cube $\capX$ (or object $\capX_\emptyset$) is the homotopy fiber of the unique map $\capX_\emptyset\rarrow *$ and hence weakly equivalent to $\capX_\emptyset$.
\end{rem}

\begin{thm}
\label{thm:connectivity_estimates_canonical_map_n_th_stage_TQ_completion}
If $X$ is a $0$-connected cofibrant $\capO$-algebra and $n\geq 0$, then the natural map
\begin{align}
\label{eq:canonical_map_into_n_th_stage_holim_tower_out_of_X}
  X\longrightarrow\holim\nolimits_{\Delta^{\leq n}} C(QX)
\end{align}
is $(n+2)$-connected; for $n=0$ this is the Hurewicz map.
\end{thm}

\begin{proof}
By Proposition \ref{prop:homotopy_fiber_calculation} the homotopy fiber of the map \eqref{eq:canonical_map_into_n_th_stage_holim_tower_out_of_X} is weakly equivalent to the total homotopy fiber of the coface $(n+1)$-cube $\capX_{n+1}$ associated to $X\rightarrow C(QX)$, hence it follows from Proposition \ref{prop:connectivities_for_total_homotopy_fiber_homotopy_completion_cube_no_zigzags} that the map \eqref{eq:canonical_map_into_n_th_stage_holim_tower_out_of_X} is $(n+2)$-connected.
\end{proof}

\begin{proof}[Proof of Theorem \ref{MainTheorem}(b)]
Assume that $X$ is a cofibrant $\capO$-algebra. We want to verify that the natural map
$
  X\rarrow\holim_\Delta C(QX)
$
is a weak equivalence.  It suffices to verify that the connectivity of the natural map \eqref{eq:canonical_map_into_n_th_stage_holim_tower_out_of_X} into the $n$-th stage of the associated $\holim$ tower (Section \ref{sec:homotopy_limit_towers}) is strictly increasing with $n$, and Theorem \ref{thm:connectivity_estimates_canonical_map_n_th_stage_TQ_completion} completes the proof.
\end{proof}

\subsection{Homotopical analysis of the derived counit map}

The following calculation points out the homotopical significance of the codegeneracy $n$-cubes $\capY_n$ associated to a cosimplicial $\capO$-algebra. This is proved in Bousfield-Kan \cite[X.6.3]{Bousfield_Kan} for the $\Tot$ tower of a Reedy fibrant cosimplicial pointed space, and the same argument verifies it in our context; for the convenience of the reader, we include a concise proof in Section \ref{sec:cubical_diagrams_homotopical_analysis}.

\begin{prop}
\label{prop:iterated_homotopy_fibers_calculation}
Let $Z$ be a cosimplicial $\capO$-algebra and $n\geq 0$. There are natural zigzags of weak equivalences
\begin{align*}
  \hofib(\holim_{\Delta^{\leq n}}Z\rarrow\holim_{\Delta^{\leq n-1}}Z)
  \wequiv\Omega^n(\iter\hofib)\capY_n
\end{align*}
where $\capY_n$ denotes the canonical $n$-cube built from the codegeneracy maps of
\begin{align*}
\xymatrix{
  Z^0 &
  Z^1
  \ar[l]_-{s^0} &
  Z^2\ar@<-0.5ex>[l]_-{s^0}\ar@<0.5ex>[l]^-{s^1}
  \ \cdots\ Z^n
}
\end{align*}
 the $n$-truncation of $Z$; in particular, $\capY_0$ is the object (or $0$-cube) $Z^0$. Here, $\Omega^n$ is weakly equivalent, in the underlying category $\ModR$, to the $n$-fold desuspension $\Sigma^{-n}$ functor. We often refer to $\capY_n$ as the \emph{codegeneracy} $n$-cube associated to $Z$.
\end{prop}

The following proposition proved in Section \ref{sec:cubical_diagrams_homotopical_analysis} gives the connectivity estimates that we need.

\begin{prop}
\label{prop:multisimplicial_calculation_of_iterated_hofiber_codegeneracy_cube_no_zigzags}
Let $Y$ be a cofibrant $\K$-coalgebra and $n\geq 1$. Denote by $\capY_n$ the the codegeneracy $n$-cube associated to the cosimplicial cobar construction $C(Y)$ of $Y$. If $Y$ is $0$-connected, then the total homotopy fiber of $\capY_n$ is $2n$-connected.
\end{prop}

\begin{thm}
\label{thm:estimating_connectivity_of_maps_in_tower_C_of_Y}
If $Y$ is a $0$-connected cofibrant $\K$-coalgebra and $n\geq 1$, then the natural map
\begin{align}
\label{eq:tower_map_from_n_th_stage_to_next_lower_stage}
  \holim_{\Delta^{\leq n}}C(Y)&\longrightarrow
  \holim_{\Delta^{\leq n-1}}C(Y)
\end{align}
is an $(n+1)$-connected map between $0$-connected objects.
\end{thm}

\begin{proof}
The homotopy fiber of the map \eqref{eq:tower_map_from_n_th_stage_to_next_lower_stage} is weakly equivalent to $\Omega^n$ of the total homotopy fiber of the codegeneracy $n$-cube $\capY_{n}$ associated to $C(Y)$ by Proposition \ref{prop:iterated_homotopy_fibers_calculation}, hence by Proposition \ref{prop:multisimplicial_calculation_of_iterated_hofiber_codegeneracy_cube_no_zigzags} the map \eqref{eq:tower_map_from_n_th_stage_to_next_lower_stage} is $(n+1)$-connected.
\end{proof}

\begin{thm}
\label{thm:connectivities_for_map_into_n_th_stage}
If $Y$ is a $0$-connected cofibrant $\K$-coalgebra and $n\geq 0$, then the natural maps
\begin{align}
\label{eq:canonical_map_needed_to_discuss}
  \holim\nolimits_\Delta C(Y)&\longrightarrow
  \holim\nolimits_{\Delta^{\leq n}}C(Y)\\
\label{eq:the_other_canonical_map_needed}
  \LL Q\holim\nolimits_\Delta C(Y)&\longrightarrow
  \LL Q\holim\nolimits_{\Delta^{\leq n}}C(Y)
\end{align}
are $(n+2)$-connected maps between $0$-connected objects.
\end{thm}

\begin{proof}
Consider the first part. By Theorem \ref{thm:estimating_connectivity_of_maps_in_tower_C_of_Y} each of the maps in the holim tower $\{\holim_{\Delta^{\leq n}}C(Y)\}_n$ (Section \ref{sec:homotopy_limit_towers}), above level $n$, is at least $(n+2)$-connected. It follows that the map \eqref{eq:canonical_map_needed_to_discuss} is $(n+2)$-connected. The second part follows from the first part, since by the $\TQ$-Hurewicz theorems in Harper-Hess \cite[1.8, 1.9]{Harper_Hess}, $\TQ$-homology preserves such connectivities.
\end{proof}

The following theorem proved in Section \ref{sec:cubical_diagrams_homotopical_analysis} gives the connectivity estimates that we need. At the technical heart of the proof lies the spectral algebra higher dual Blakers-Massey theorem proved in Ching-Harper \cite[1.11]{Ching_Harper}; these structured ring spectrum analogs of Goodwillie's \cite{Goodwillie_calculus_2} higher cubical diagram theorems were worked out in \cite{Ching_Harper} as part of an arsenal of new tools for attacking the Francis-Gaitsgory conjecture, and for ultimately proving Theorem \ref{thm:connectivities_for_map_that_commutes_Q_into_inside_of_holim} below.

The input to the spectral algebra higher dual Blakers-Massey theorem \cite[1.11]{Ching_Harper} requires the homotopical analysis of an $\infty$-cartesian $(n+1)$-cube associated to the $n$-th stage, $\holim_{\Delta^{\leq n}}C(Y)$, of the $\holim$ tower associated to $C(Y)$; it is built from coface maps in $C(Y)$. The needed homotopical analysis is worked out in Section \ref{sec:cubical_diagrams_homotopical_analysis} by leveraging the connectivity estimates in Proposition \ref{prop:multisimplicial_calculation_of_iterated_hofiber_codegeneracy_cube_no_zigzags} with the fact that codegeneracy maps provide retractions for the appropriate coface maps.

\begin{thm}
\label{thm:connectivities_for_map_that_commutes_Q_into_inside_of_holim}
If $Y$ is a $0$-connected cofibrant $\K$-coalgebra and $n\geq 1$, then the natural map
\begin{align}
\label{eq:commuting_LQ_past_holim_delta}
  \LL Q\holim\nolimits_{\Delta^{\leq n}} C(Y)\longrightarrow
  \holim\nolimits_{\Delta^{\leq n}} \LL Q\,C(Y),
\end{align}
is $(n+4)$-connected; the map is a weak equivalence for $n=0$.
\end{thm}

The following is a corollary of these connectivity estimates, together with a left cofinality argument in Dror-Dwyer \cite{Dror_Dwyer_long_homology}.

\begin{thm}
\label{thm:FQC_commutes_with_desired_holim}
If $Y$ is a $0$-connected cofibrant $\K$-coalgebra, then the natural maps
\begin{align}
\label{eq:FQC_commutes_with_desired_holim}
  \LL Q\holim\nolimits_\Delta C(Y)&\xrightarrow{\wequiv}
  \holim\nolimits_\Delta \LL Q\,C(Y)\xrightarrow{\wequiv}Y
\end{align}
are weak equivalences.
\end{thm}

\begin{proof}
Consider the left-hand map. It suffices to verify that the connectivities of the natural maps \eqref{eq:the_other_canonical_map_needed} and \eqref{eq:commuting_LQ_past_holim_delta}
 are strictly increasing with $n$, and Theorems \ref{thm:connectivities_for_map_into_n_th_stage} and \ref{thm:connectivities_for_map_that_commutes_Q_into_inside_of_holim} complete the proof. Consider the case of the right-hand map. Since the cosimplicial cobar construction $\Cobar(\K,\K,Y)$, which is weakly equivalent to $\LL Q\, C(Y)$, has extra codegeneracy maps $s^{-1}$ (Dwyer-Miller-Neisendorfer \cite[6.2]{Dwyer_Miller_Neisendorfer}), it follows from the cofinality argument in Dror-Dwyer \cite[3.16]{Dror_Dwyer_long_homology} that the right-hand map is a weak equivalence.
\end{proof}

\begin{proof}[Proof of Theorem \ref{MainTheorem}(a)]
We want to verify that the natural map
\begin{align*}
  \LL Q\holim\nolimits_\Delta C(Y)\xrightarrow{\wequiv} Y
\end{align*}
is a weak equivalence; since this is the composite \eqref{eq:FQC_commutes_with_desired_holim}, Theorem \ref{thm:FQC_commutes_with_desired_holim} completes the proof.
\end{proof}

\subsection{Strong convergence of the $\TQ$-completion spectral sequence}

The following strong convergence result for the $\TQ$-homology completion spectral sequence (Section \ref{sec:homotopy_limit_towers}) is a corollary of the connectivity estimates in Theorem \ref{thm:connectivity_estimates_canonical_map_n_th_stage_TQ_completion}.

\begin{thm}
\label{thm:strong_convergence_TQ_homology_completion}
Let $\function{f}{X}{Y}$ be a map of $\capO$-algebras.
\begin{itemize}
\item[(a)] If $X$ is $0$-connected, then the $\TQ$-homology completion spectral sequence
\begin{align*}
  E^2_{-s,t} &= \pi^s\pi_t\TQ_\cdot(X)
  \Longrightarrow
  \pi_{t-s}X^{\wedge}_\TQ
\end{align*}
converges strongly (Remark \ref{rem:strong_convergence}).
\item[(b)] If $X$ is $0$-connected, then the natural  coaugmentation $X\wequiv X^\wedge_\TQ$ is a weak equivalence.
\item[(c)] If $f$ induces a weak equivalence $\TQ(X)\wequiv\TQ(Y)$ on topological Quillen homology, then $f$ induces a weak equivalence $X^\wedge_\TQ\wequiv Y^\wedge_\TQ$ on $\TQ$-homology completion.
\end{itemize}
Here, the $\TQ$-homology cosimplicial resolution $\TQ_\cdot(X)$ of $X$ denotes a functorial Reedy fibrant replacement of the cosimplicial $\capO$-algebra $C(Q(X^c))$, where $X^c$ denotes a functorial cofibrant replacement of $X$.
\end{thm}

\begin{proof}
This follows from the connectivity estimates in Theorem \ref{thm:connectivity_estimates_canonical_map_n_th_stage_TQ_completion}.
\end{proof}

\begin{rem}
The completion result $X\wequiv X^\wedge_\TQ$ in Theorem \ref{thm:strong_convergence_TQ_homology_completion}(b) shows that $X$ can be recovered from its $\TQ$-homology. This can be thought of as the spectral algebra analog of several closely related known examples of completion phenomena in the context of spaces. For instance, assuming that $Y$ is a simply connected pointed space, (i) $Y\wequiv Y^\wedge_\ZZ$ in Bousfield-Kan \cite{Bousfield_Kan}, (ii) $Y\wequiv Y^\wedge_{\Omega^n\Sigma^n}$ in Bousfield \cite{Bousfield_cosimplicial_space}, (iii) $Y\wequiv Y^\wedge_{\Omega^\infty\Sigma^\infty}$ in Carlsson \cite{Carlsson_equivariant} and subsequently in Arone-Kankaanrinta \cite{Arone_Kankaanrinta}, and (iv) $Y\wequiv \holim_\Delta \mathbf{C}(Y,Y)$ where $\mathbf{C}(Y,Y)^n\wequiv\Sigma Y\vee\dots\vee\Sigma Y$ ($n$-copies) in Hopkins \cite{Hopkins_iterated_suspension} (recovers $Y$ from $\Sigma Y$), among others.
\end{rem}

\begin{rem}
\label{rem:strong_convergence}
By \emph{strong convergence} of $\{E^r\}$ to $\pi_*(X^\wedge_\TQ)$ we mean that (i) for each $(-s,t)$, there exists an $r$ such that $E^r_{-s,t}=E^\infty_{-s,t}$ and (ii) for each $i$, $E^\infty_{-s,s+i}=0$ except for finitely many $s$. Strong convergence implies that for each $i$, $\{E^\infty_{-s,s+i}\}$ is the set of filtration quotients from a finite filtration of $\pi_i(X^\wedge_\TQ)$; see, for instance, Bousfield-Kan \cite[IV.5.6, IX.5.3, IX.5.4]{Bousfield_Kan} and Dwyer \cite{Dwyer_strong_convergence}.
\end{rem}

\subsection{Strong convergence for the $\holim_\Delta C(Y)$ spectral sequence}

The following strong convergence result for the homotopy spectral sequence (Section \ref{sec:homotopy_limit_towers}) associated to the cosimplicial cobar construction $C(Y)$ of a $\K$-coalgebra $Y$ is a corollary of the connectivity estimates in Theorem \ref{thm:estimating_connectivity_of_maps_in_tower_C_of_Y}.

\begin{thm}
If $Y$ is a $0$-connected cofibrant $\K$-coalgebra, then the homotopy spectral sequence
\begin{align*}
  E^2_{-s,t} &= \pi^s\pi_t C(Y)^\mathrm{f}
  \Longrightarrow
  \pi_{t-s}\bigl(\holim\nolimits_\Delta C(Y)\bigr)
\end{align*}
converges strongly (Remark \ref{rem:strong_convergence}); here, $C(Y)^\mathrm{f}$ denotes a functorial Reedy fibrant replacement of the cosimplicial $\capO$-algebra $C(Y)$.
\end{thm}

\begin{proof}
This follows from the connectivity estimates in Theorem \ref{thm:estimating_connectivity_of_maps_in_tower_C_of_Y}.
\end{proof}

\section{Simplicial structures and $\capO$-algebras}
\label{sec:simplicial_structures}

First we recall the simplicial structure on $\capO$-algebras (Harper-Hess \cite[6.1]{Harper_Hess}).

\begin{defn}
\label{defn:simplicial_structure_mapping_object}
Let $\capO$ be an operad in $\capR$-modules. Let $X,X'$ be $\capO$-algebras and $K$ a simplicial set. The \emph{mapping object} $\hombold_{\AlgO}(K,X)$ in $\AlgO$ is defined by
\begin{align*}
  \hombold_\AlgO(K,X):=\Map(K_+,X)
\end{align*}
with left $\capO$-action map induced by $\function{m}{\capO\circ(X)}{X}$, together with the natural maps $K\rarrow K^{\times t}$ in $\sSet$ for $t\geq 0$; these are the diagonal maps for $t\geq 1$ and the constant map for $t=0$. For ease of notation purposes, we sometimes drop the $\AlgO$ decoration from the notation and simply denote the mapping object by $\hombold(K,X)$.
\end{defn}

\begin{defn}
\label{defn:simplicial_structure_tensor_product}
Let $\capO$ be an operad in $\capR$-modules. Let $X,X'$ be $\capO$-algebras and $K$ a simplicial set. The \emph{tensor product} $X\tensordot K$ in $\AlgO$ is defined by the reflexive coequalizer
\begin{align}
\label{eq:tensordot_definition_as_coequalizer}
  X\tensordot K &:=\colim
  \Bigl(
  \xymatrix{
  \capO\circ(X\Smash K_+) &
  \capO\circ\bigl(\capO\circ(X)\Smash K_+\bigr)
  \ar@<0.5ex>[l]^-{d_1}\ar@<-0.5ex>[l]_-{d_0}
  }
  \Bigr)
\end{align}
in $\AlgO$, with $d_0$ induced by operad multiplication $\function{m}{\capO\circ\capO}{\capO}$ and the natural map $\function{\nu}{\capO\circ(X)\Smash K_+}{\capO\circ(X\Smash K_+)}$ (see \cite[6.1]{Harper_Hess}), while $d_1$ is induced by the left $\capO$-action map $\function{m}{\capO\circ(X)}{X}$.
\end{defn}

\begin{defn}
\label{defn:simplicial_structure_mapping_space}
Let $\capO$ be an operad in $\capR$-modules. Let $X,X'$ be $\capO$-algebras. The \emph{mapping space} $\Hombold_{\AlgO}(X,X')$ in $\sSet$ is defined degreewise by
\begin{align*}
  \Hombold_\AlgO(X,X')_n := \hom_\AlgO(X\tensordot\Delta[n],X')
\end{align*}
For ease of notation purposes, we often drop the $\AlgO$ decoration from the notation and simply denote the mapping space by $\Hombold(X,X')$.
\end{defn}

\begin{prop}
Let $\capO$ be an operad in $\capR$-modules. Consider $\AlgO$ with the positive flat stable model structure. Then $\AlgO$ is a simplicial model category (see  Goerss-Jardine \cite[II.3]{Goerss_Jardine}) with the above definitions of mapping object, tensor product, and mapping space.
\end{prop}

\begin{proof}
This is proved in \cite[6.18]{Harper_Hess}.
\end{proof}

\begin{rem}
\label{rem:useful_adjunction_isomorphisms_simplicial_structure}
In particular, there are isomorphisms
\begin{align}
\label{eq:tensordot_adjunction_isomorphisms}
  \hom_\AlgO(X\tensordot K,X')
  &\Iso\hom_\AlgO(X,\hombold(K,X'))\\
\notag
  &\Iso\hom_\sSet(K,\Hombold(X,X'))
\end{align}
in $\Set$, natural in $X,K,X'$, that extend to isomorphisms
\begin{align*}
  \Hombold_\AlgO(X\tensordot K,X')
  &\Iso\Hombold_\AlgO(X,\hombold(K,X'))\\
  &\Iso\Hombold_\sSet(K,\Hombold(X,X'))
\end{align*}
in $\sSet$, natural in $X,K,X'$.
\end{rem}

\subsection{Simplicial functors and natural transformations}

Recall that if $\function{f}{\capO}{\capO'}$ is a map of operads in $\capR$-modules,   then the change of operads adjunction
\begin{align}
\label{eq:change_of_operads_adjunction_general}
\xymatrix{
  \Alg_{\capO}\ar@<0.5ex>[r]^-{f_*} & \Alg_{\capO'}\ar@<0.5ex>[l]^-{f^*}
}
\end{align}
is a Quillen adjunction with left adjoint on top and $f^*$ the forgetful functor (or the ``restriction along $f$ of the operad action''); in particular, for each $\capO$-algebra $X$ and $\capO'$-algebra $Y$ there is an isomorphism
\begin{align}
\label{eq:hom_set_adjunction_change_of_operads_general}
  \hom_{\Alg_{\capO'}}(f_*(X),Y)\Iso\hom_\AlgO(X,f^*(Y))
\end{align}
in $\Set$, natural in $X,Y$.

The following proposition, whose proof we defer to Section \ref{sec:proofs}, is fundamental to this paper. It verifies that the change of operads adjunction \eqref{eq:change_of_operads_adjunction_general} meshes nicely with the simplicial structure; this is closely related to Goerss-Jardine \cite[II.2.9]{Goerss_Jardine}.

\begin{prop}
\label{prop:useful_properties_of_the_adjunction}
Let $\function{f}{\capO}{\capO'}$ be a map of operads in $\capR$-modules. Let $X$ be an $\capO$-algebra, $Y$ an $\capO'$-algebra, and $K,L$ simplicial sets.  Then
\begin{itemize}
\item[(a)] there is a natural isomorphism
$
  \sigma\colon\thinspace f_*(X)\tensordot K \xrightarrow{\Iso}f_*(X\tensordot K)
$;
\item[(b)] there is an isomorphism
\begin{align*}
  \Hombold(f_*(X),Y)\Iso\Hombold(X,f^*(Y))
\end{align*}
in $\sSet$, natural in $X,Y$, that extends the adjunction isomorphism in  \eqref{eq:hom_set_adjunction_change_of_operads_general};
\item[(c)] there is an isomorphism
\begin{align*}
  f^*\hombold(K,Y)\Equal\hombold(K,f^*Y)
\end{align*}
in $\AlgO$, natural in $K,Y$.
\item[(d)] there is a natural map
$
  \function{\sigma}{f^*(Y)\tensordot K}{f^*(Y\tensordot K)}
$
induced by $f$.
\item[(e)] the functors $f_*$ and $f^*$ are simplicial functors (Remark \ref{rem:simplicial_functors}) with the structure maps $\sigma$ of (a) and (d), respectively.
\end{itemize}
\end{prop}

\begin{rem}
\label{rem:simplicial_functors}
For a useful reference on simplicial functors in the context of homotopy theory, see Hirschhorn \cite[9.8.5]{Hirschhorn}.
\end{rem}

Recall the following notion of a simplicial natural transformation (see Goerss-Jardine \cite[IX.1]{Goerss_Jardine}); note that evaluating diagram \eqref{eq:simplicial_natural_transformation} at simplicial degree $0$ recovers the usual naturality condition required for a collection of maps $\function{\tau_X}{FX}{GX}$, $X\in\AlgO$, to define a natural transformation of the form $\function{\tau}{F}{G}$.

\begin{defn}
Let $\capO,\capO'$ be operads in $\capR$-modules. Let $\function{F,G}{\AlgO}{\Alg_{\capO'}}$ be simplicial functors (Remark \ref{rem:simplicial_functors}). A natural transformation $\function{\tau}{F}{G}$ is called a \emph{simplicial natural transformation} if the diagram
\begin{align}
\label{eq:simplicial_natural_transformation}
\xymatrix{
  \Hombold(X,Y)\ar[d]_-{G}\ar[r]^-{F} & \Hombold(FX,FY)\ar[d]^-{(\id,\tau_Y)}\\
  \Hombold(GX,GY)\ar[r]_(0.48){(\tau_X,\id)} & \Hombold(FX,GY)
}
\end{align}
of mapping spaces in $\sSet$ commutes, for every $X,Y\in\AlgO$.
\end{defn}

The following proposition, which is an exercise left to the reader, provides a useful characterization of simplicial natural transformations.

\begin{prop}
\label{prop:characterization_of_simplicial_natural_transformations}
Let $\capO,\capO'$ be operads in $\capR$-modules. Let $\function{F,G}{\AlgO}{\Alg_{\capO'}}$ be simplicial functors. A natural transformation $\function{\tau}{F}{G}$ is a simplicial natural transformation if and only if it respects the simplicial structure maps; i.e., if and only if the left-hand diagram
\begin{align*}
\xymatrix{
  F(X)\tensordot\Delta[n]\ar[d]_-{\tau_X\tensordot\id}\ar[r]^-{\sigma} &
  F(X\tensordot\Delta[n])\ar[d]^-{\tau_{X\tensordot\Delta[n]}}\\
  G(X)\tensordot\Delta[n]\ar[r]_-{\sigma} &
  G(X\tensordot\Delta[n])
}\quad\quad
\xymatrix{
  F(X)\tensordot K\ar[d]_-{\tau_X\tensordot\id}\ar[r]^-{\sigma} &
  F(X\tensordot K)\ar[d]^-{\tau_{X\tensordot K}}\\
  G(X)\tensordot K\ar[r]_-{\sigma} &
  G(X\tensordot K)
}
\end{align*}
in $\Alg_{\capO'}$ commutes for every $X\in\AlgO$ and $n\geq 0$, if and only if the right-hand diagram in $\Alg_{\capO'}$ commutes for every $X\in\AlgO$ and $K\in\sSet$.  Here, we denote by $\sigma$ the indicated simplicial structure maps (see Hirschhorn \cite[9.8.5]{Hirschhorn}).
\end{prop}

The following proposition, whose proof we defer to Section \ref{sec:proofs}, is fundamental to this paper. In particular, it verifies that the natural transformations (see \eqref{eq:TQ_homology_spectrum_functor_natural_transformations}) associated to the $\TQ$-homology spectrum functor respect the simplicial structure maps.

\begin{prop}
\label{prop:unit_and_counit_are_simplicial}
Let $\function{f}{\capO}{\capO'}$ be a map of operads in $\capR$-modules. Consider the monad $f^*f_*$ on $\AlgO$ and the comonad $f_*f^*$ on $\Alg_{\capO'}$ associated to the adjunction $(f_*,f^*)$ in \eqref{eq:change_of_operads_adjunction_general}. The associated natural transformations
\begin{align*}
  \id\xrightarrow{\eta} f^*f_*&\quad\text{(unit)},\quad\quad\quad
  &\id\xleftarrow{\varepsilon}f_*f^*& \quad\text{(counit)}, \\
  f^*f_*f^*f_*\rarrow f^*f_*&\quad\text{(multiplication)},\quad\quad\quad
  &f_*f^*f_*f^*\xleftarrow{m}f_*f^*& \quad\text{(comultiplication)}
\end{align*}
are simplicial natural transformations.
\end{prop}

\section{$\TQ$-homology completion}
\label{sec:TQ_homology_completion}

The purpose of this section is to introduce the $\TQ$-homology cosimplicial resolutions that are used in this paper, and the closely related cosimplicial cobar constructions for $\K$-coalgebras.

\subsection{Indecomposable quotient functors for $\capO$-algebras}
\label{sec:derived_indecomposables}

Basterra-Mandell \cite[Section 8]{Basterra_Mandell} construct an indecomposables functor $Q$ for spectra equipped with an action of an operad in spaces, via suspension spectra, with trivial $0$-ary operations. While many naturally occurring operads $\capO$ of spectra can be written as the suspension spectra of operads in spaces, this is not true in general.

Harper \cite[Section 1]{Harper_bar_constructions} constructs an indecomposables functor $I\circ_\capO-$ for $\capO$-algebras, motivated by the work of Fresse \cite{Fresse} and Rezk \cite{Rezk}, where $\capO$ is any operad of spectra satisfying $\tau_1\capO=I$ and whose total left derived functor is a model for $\TQ$-homology; here, $I$ denotes the initial operad. Subsequently, Harper-Hess \cite[3.15]{Harper_Hess} construct an indecomposables functor $\tau_1\capO\circ_\capO-$ for $\capO$-algebras, motivated by the work of Lawson \cite{Lawson}, where $\capO$ is now any operad of spectra with trivial $0$-ary operations, and whose total left derived functor is a model for $\TQ$-homology. In other words, the indecomposables functor constructed in Harper-Hess \cite{Harper_Hess} has a very simple and useful description: it is the left adjoint of the change of operads adjunction associated to the canonical map $\capO\rarrow\tau_1\capO$ of operads.

The correct construction of the indecomposables functor for $\capO$-algebras, in this generality, had not previously been well-understood. In particular, the indecomposables functor lands in the category of $\tau_1\capO$-algebras, which is isomorphic to the category of left $\capO[1]$-modules. There is a twisted group ring structure that comes into play when writing down a correct model for the indecomposables of $\capO$-algebras, and this structure is precisely encoded by the relative circle product $\tau_1\capO\circ_\capO-$ functor, an observation that seems to have been overlooked in the past.

\subsection{Simplicial bar constructions and $\TQ$-homology}
Basterra \cite{Basterra} shows that the total left derived functor of indecomposables on non-unital commutative algebra spectra can be calculated by a simplicial bar construction. Harper
\cite[1.10]{Harper_bar_constructions} shows that the total left derived functor (of the left-adjoint in any change of operads adjunction) can be calculated by a simplicial bar construction, and since the indecomposables construction in Harper-Hess \cite{Harper_Hess} is the left adjoint of a change of operads adjunction, it follows immediately that the $\TQ$-homology of any $\capO$-algebra can be calculated by a simplicial bar construction; this is a key ingredient in the proof of our main result.

\subsection{$\TQ$-homology completion and the fundamental adjunction}
Let $\capO$ be an operad in $\capR$-modules such that $\capO[0]$ is the null object $*$ and consider any $\capO$-algebra $X$.

In order to work with the cosimplicial $\TQ$-homology resolution \eqref{eq:TQ_homology_resolution_derived_functor_form_introduction}, it will be useful to introduce some notation. Consider any factorization of the canonical map $\capO\rarrow\tau_1\capO$ in the category of operads as
\begin{align*}
  \capO\xrightarrow{g}J\xrightarrow{h}\tau_1\capO
\end{align*}
a cofibration followed by a weak equivalence (\cite[3.16]{Harper_Hess}) with respect to the positive flat stable model structure on $\ModR$. These maps induce change of operads adjunctions
\begin{align}
\label{eq:induced_adjunctions_from_g_and_h}
\xymatrix{
  \AlgO\ar@<0.5ex>[r]^-{g_*} &
  \Alg_{J}\ar@<0.5ex>[l]^-{g^*}\ar@<0.5ex>[r]^-{h_*} &
  \Alg_{\tau_1\capO}\ar@<0.5ex>[l]^-{h^*}
}
\end{align}
with left adjoints on top and $g^*,h^*$ the forgetful functors. It is important to note that since $h$ is a weak equivalence, the right-hand adjunction $(h_*,h^*)$ is a Quillen equivalence (Harper \cite[1.4]{Harper_symmetric_spectra}). To simplify notation we denote by $Q:=g_*$ the indicated left adjoint and $U:=g^*$ the indicated forgetful functor. Associated to the $(Q,U)$ adjunction in \eqref{eq:induced_adjunctions_from_g_and_h} is the monad $UQ$ on $\AlgO$ and the comonad $\K:=QU$ on $\Alg_J$ of the form
\begin{align}
\label{eq:TQ_homology_spectrum_functor_natural_transformations}
  \id\xrightarrow{\eta} UQ&\quad\text{(unit)},\quad\quad\quad
  &\id\xleftarrow{\varepsilon}\K& \quad\text{(counit)}, \\
  \notag
  UQUQ\rarrow UQ&\quad\text{(multiplication)},\quad\quad\quad
  &\K\K\xleftarrow{m}\K& \quad\text{(comultiplication)}.
\end{align}

A fundamental observation emphasized in the homotopic descent work of Hess \cite[2.11]{Hess}, and subsequently in Arone-Ching \cite{Arone_Ching_classification} and the Francis-Gaitsgory conjecture \cite[3.4.5]{Francis_Gaitsgory}, is that there is a factorization of adjunctions
\begin{align}
\label{eq:factorization_of_adjunctions_TQ_homology}
\xymatrix@1{
  \AlgO\ar@<0.5ex>[rr]^-{Q}\ar@<0.5ex>[dr]^(0.6){Q} &&
  \Alg_{J}\ar@<0.5ex>[ll]^-{U}\ar@<0.5ex>[dl]^-{\K}\\
  & \coAlgK\ar@<0.5ex>[ul]^(0.4){\lim_\Delta C}\ar@<0.5ex>[ur]^-{}
}
\end{align}
with left adjoints on top and $\coAlgK\rarrow\Alg_J$ the forgetful functor. In other words, if we restrict to cofibrant objects in $\AlgO$, then the $\TQ$-homology spectrum $QX$ underlying the $\TQ$-homology $\capO$-algebra $UQ(X)$ of $X$ is naturally equipped with a $\K$-coalgebra structure. While we defer the definition of $C$ to the next subsection (Definition \ref{defn:cobar_construction}), to understand the comparison in \eqref{eq:factorization_of_adjunctions_TQ_homology} between $\AlgO$ and $\coAlgK$ it suffices to know that $\lim_\Delta C(Y)$ is naturally isomorphic to an equalizer of the form
\begin{align*}
  \lim_\Delta C(Y)\Iso
  \lim\Bigl(
  \xymatrix{
    UY\ar@<0.5ex>[r]^-{d^0}\ar@<-0.5ex>[r]_-{d^1} &
    U\K Y
  }
  \Bigr)
\end{align*}
where $d^0=m\id$, $d^1=\id m$, $\function{m}{U}{U\K=UQU}$ denotes the $\K$-coaction map on $U$ (defined by $m:=\eta\id$), and $\function{m}{Y}{\K Y}$ denotes the $\K$-coaction map on $Y$; this is because of the following property of cosimplicial objects (see Definition \ref{defn:cobar_construction}).

\begin{prop}
\label{prop:lim_of_cosimplicial_object}
Let $\M$ be a category with all small limits. If $X\in\M^\Delta$ (resp. $Y\in\M^{\Delta_\res}$), then its limit is naturally isomorphic to an equalizer of the form
\begin{align*}
  \lim_{\Delta}X\Iso
  \lim\bigl(
  \xymatrix{
    X^0\ar@<0.5ex>[r]^-{d^0}\ar@<-0.5ex>[r]_-{d^1} &
    X^1
  }
  \bigr)\quad\quad
  \Bigl(
  \text{resp.}\quad
  \lim_{\Delta_\res}Y\Iso
  \lim\bigl(
  \xymatrix{
    Y^0\ar@<0.5ex>[r]^-{d^0}\ar@<-0.5ex>[r]_-{d^1} &
    Y^1
  }
  \bigr)
  \Bigr)
\end{align*}
in $\M$, with $d^0$ and $d^1$ the indicated coface maps of $X$ (resp. $Y$).
\end{prop}

\begin{proof}
This follows easily by using the cosimplicial identities \cite[I.1]{Goerss_Jardine} to verify the universal property of limits.
\end{proof}

\subsection{Cosimplicial cobar constructions}

It will be useful to interpret the cosimplicial $\TQ$-resolution of $X$ in terms of the following cosimplicial cobar construction involving the comonad $\K$ on $\Alg_J$. First note that associated to the adjunction $(Q,U)$ is a right $\K$-coaction $\function{m}{U}{U\K}$ on $U$ (defined by $m:=\eta\id$) and a left $\K$-coaction (or $\K$-coalgebra structure) $\function{m}{QX}{\K QX}$ on $QX$ (defined by $m=\id\eta\id$), for any $X\in\AlgO$.

\begin{defn}
Denote by $\function{\eta}{\id}{F}$ and $\function{m}{FF}{F}$ the unit and multiplication maps of the simplicial fibrant replacement monad $F$ on $\AlgJ$ (Blumberg-Riehl \cite[6.1]{Blumberg_Riehl}).
\end{defn}

\begin{defn}
\label{defn:cobar_construction}
Let $Y$ be an object in $\coAlgK$. The \emph{cosimplicial cobar constructions} (or two-sided cosimplicial cobar constructions) $C(Y):=\Cobar(U,\K,Y)$ and $\mathfrak{C}(Y):=\Cobar(U,F\K,FY)$ in $(\AlgO)^{\Delta}$ look like (showing only the coface maps)
\begin{align}
\label{eq:cobar_construction}
&\xymatrix{
  C(Y): \quad\quad
  UY\ar@<0.5ex>[r]^-{d^0}\ar@<-0.5ex>[r]_-{d^1} &
  U\K Y
  \ar@<1.0ex>[r]\ar[r]\ar@<-1.0ex>[r] &
  U\K\K Y
  \cdots
}\\
\label{eq:cobar_construction_fattened_up}
&\xymatrix{
  \mathfrak{C}(Y): \quad\quad
  UFY\ar@<0.5ex>[r]^-{d^0}\ar@<-0.5ex>[r]_-{d^1} &
  U(F\K)FY
  \ar@<1.0ex>[r]\ar[r]\ar@<-1.0ex>[r] &
  U(F\K)(F\K)FY
  \cdots
}
\end{align}
and are defined objectwise by $C(Y)^n:=U\K^n Y$ and $\mathfrak{C}(Y)^n:=U(F\K)^n FY$ with the obvious coface and codegeneracy maps; see, for instance, the face and degeneracy maps in the simplicial bar constructions described in \cite[A.1]{Gugenheim_May}, \cite[Section 7]{May_classifying_spaces}, and dualize. For instance, in \eqref{eq:cobar_construction} the indicated coface maps are defined by $d^0:=m\id$ and $d^1:=\id m$, and similarly for \eqref{eq:cobar_construction_fattened_up} where they are modified in the obvious way by insertion of the unit map $\function{\eta}{\id}{F}$ of the simplicial fibrant replacement monad $F$.
\end{defn}

\begin{rem}
It may be helpful to note that while $F\K$ does not inherit the structure of a monad from $\K$, it does inherit the structure of a non-unital monad from $\K$; nevertheless, it is easy to verify that the cosimplicial cobar construction \eqref{eq:cobar_construction_fattened_up} is a well-defined cosimplicial $\capO$-algebra.
\end{rem}

\begin{prop}
\label{prop:fattened_version_of_C}
Let $Y$ be a $\K$-coalgebra. The unit map $\function{\eta}{\id}{F}$ induces a well-defined natural map of the form $C(Y)\rarrow\mathfrak{C}(Y)$ of cosimplicial $\capO$-algebras.
\end{prop}

\begin{proof}
This is an exercise left to the reader.
\end{proof}

If $Y$ is a cofibrant $\K$-coalgebra, then the unit map $\function{\eta}{\id}{F}$ of the simplicial fibrant replacement monad induces a well-defined natural map
\begin{align*}
  C(Y)\xrightarrow{\wequiv}
  \mathfrak{C}(Y)
\end{align*}
of cosimplicial $\capO$-algebras; this map is an objectwise weak equivalence. The reason we introduce the cosimplicial cobar construction $\mathfrak{C}(Y)$, which can be thought of as a ``fattened up'' version of $C(Y)$, will become clear in Section \ref{sec:homotopy_theory_K_coalgebras}; it is needed for technical reasons involving the construction of homotopically meaningful mapping spaces of $\K$-coalgebras.

\section{Homotopy limit towers and cosimplicial $\capO$-algebras}
\label{sec:homotopy_limit_towers}

The purpose of this section is to make precise the several towers of $\capO$-algebras,   associated to a given cosimplicial $\capO$-algebra, that are needed in this paper. Most of our arguments involve $\holim_{\Delta}$ and $\holim_{\Delta^{\leq n}}$, which are defined in terms of the Bousfield-Kan homotopy limit functors $\holim^\BK_{\Delta}$ and $\holim^\BK_{\Delta^{\leq n}}$, which in turn are defined in terms of the totalization functor $\Tot$ for cosimplicial $\capO$-algebras. For technical reasons that arise in the homotopy theory of $\K$-coalgebras, we also require use of the restricted totalization $\Tot^\res$ functor. Furthermore, the construction of the homotopy spectral sequence associated to a cosimplicial $\capO$-algebra is defined in terms of the associated $\Tot$ tower of $\capO$-algebras, and the Bousfield-Kan identification of the resulting $E^2$ term requires the fundamental pullback diagrams constructing $\Tot^n$ from $\Tot^{n-1}$ for $\capO$-algebras. In other words, this section makes precise the various definitions and constructions for cosimplicial $\capO$-algebras what readers familiar with Bousfield-Kan \cite{Bousfield_Kan} will recognize as favorite tools from the context of cosimplicial pointed spaces.

\begin{defn}
A cosimplicial $\capO$-algebra $Z$ is \emph{coaugmented} if it comes with a map
\begin{align}
\label{eq:cosimplicial_object_equipped_with_a_coaugmentation}
  \function{d^0}{Z^{-1}}{Z^0}
\end{align}
of $\capO$-algebras such that $\function{d^0d^0=d^1d^0}{Z^{-1}}{Z^1}$; in this case, it follows easily from the cosimplicial identities (\cite[I.1]{Goerss_Jardine}) that \eqref{eq:cosimplicial_object_equipped_with_a_coaugmentation} induces a map
\begin{align*}
  Z^{-1}\rarrow Z
\end{align*}
of cosimplicial $\capO$-algebras, where $Z^{-1}$ denotes the constant cosimplicial $\capO$-algebra with value $Z^{-1}$; i.e., via the inclusion $Z^{-1}\in\AlgO\subset(\AlgO)^\Delta$ of constant diagrams.
\end{defn}

We follow Dror-Dwyer \cite{Dror_Dwyer_long_homology}
in use of the terms \emph{restricted cosimplicial objects} for $\Delta_\res$-shaped diagrams, and \emph{restricted simplicial category} $\Delta_\res$ to denote the subcategory of $\Delta$ with objects the totally ordered sets $[n]$ for $n\geq 0$ and morphisms the strictly monotone maps of sets $\function{\xi}{[n]}{[n']}$; i.e., such that $k<l$ implies $\xi(k)<\xi(l)$.

\begin{defn}
\label{defn:totalization_and_restricted_totalization}
Let $\capO$ be an operad in $\capR$-modules. The \emph{totalization} functor $\Tot$ for cosimplicial $\capO$-algebras and the \emph{restricted totalization} (or fat totalization) functor $\Tot^\res$ for restricted cosimplicial $\capO$-algebras are defined objectwise by the ends
\begin{align*}
  \function{\Tot}{(\AlgO)^\Delta}{\AlgO},\quad\quad
  &X\mapsto\hombold(\Delta[-],X)^\Delta\\
  \function{\Tot^\res}{(\AlgO)^{\Delta_\res}}{\AlgO},\quad\quad
  &Y\mapsto\hombold(\Delta[-],Y)^{\Delta_\res}
\end{align*}
We often drop the adjective ``restricted'' and simply refer to both functors as \emph{totalization} functors. It follows from the universal property of ends that $\Tot(X)$ is naturally isomorphic to an equalizer diagram of the form
\begin{align*}
  \Tot(X)&\Iso
  \lim\Bigl(
  \xymatrix{
    \prod\limits_{[n]\in\Delta}\hombold(\Delta[n],X^n)
    \ar@<0.5ex>[r]\ar@<-0.5ex>[r] &
    \prod\limits_{\substack{[n]\rightarrow [n']\\ \text{in}\,\Delta}}\hombold(\Delta[n],X^{n'})
  }
  \Bigr)
\end{align*}
in $\AlgO$, and similarly for $\Tot^\res(Y)$ by replacing $\Delta$ with $\Delta_\res$.  We sometimes refer to the natural maps $\Tot(X)\rarrow\hombold(\Delta[n],X^n)$ and $\Tot^\res(Y)\rarrow\hombold(\Delta[n],Y^n)$ as the \emph{projection} maps.
\end{defn}

\begin{prop}
\label{prop:tot_adjunctions}
Let $\capO$ be an operad in $\capR$-modules. The totalization functors $\Tot$ and $\Tot^\res$ fit into adjunctions
\begin{align}
\label{eqref:tot_adjunctions}
\xymatrix{
  \Alg_{\capO}\ar@<0.5ex>[r]^-{-\tensordot\Delta[-]} &
  (\Alg_\capO)^\Delta\ar@<0.5ex>[l]^-{\Tot}
},\quad\quad
\xymatrix{
  \Alg_{\capO}\ar@<0.5ex>[r]^-{-\tensordot\Delta[-]} &
  (\Alg_\capO)^{\Delta_\res}\ar@<0.5ex>[l]^-{\Tot^\res}
}
\end{align}
with left adjoints on top.
\end{prop}

\begin{proof}
Consider the case of $\Tot$ (resp. $\Tot^\res)$. Using the universal property of ends, it is easy verify that the functor given objectwise by $A\tensordot\Delta[-]$ (resp. $A\tensordot\Delta[-]$) is a left adjoint of $\Tot$ (resp. $\Tot^\res$).
\end{proof}

\begin{defn}
Let $\capO$ be an operad in $\capR$-modules and $\DD$ a small category. The \emph{Bousfield-Kan homotopy limit} functor $\holim_\DD^\BK$ for $\DD$-shaped diagrams in $\AlgO$ is defined objectwise by
\begin{align*}
  \function{\holim\nolimits_\DD^\BK}{(\AlgO)^\DD}{\AlgO},\quad\quad
  &X\mapsto\Tot\xymatrix{\prod\nolimits^*_\DD} X
\end{align*}
We will sometimes suppress $\DD$ from the notation and simply write $\holim^\BK$ and $\prod^*$. Here, the cosimplicial replacement functor $\function{\prod\nolimits^*}{(\AlgO)^{\DD}}{(\AlgO)^\Delta}$ is defined objectwise by
\begin{align*}
\xymatrix{
  \prod^n X:= \prod\limits_{\substack{a_0\rightarrow\cdots\rightarrow a_n\\ \text{in}\,\DD}}X(a_n)
}
\end{align*}
with the obvious coface $d^i$ and codegeneracy maps $s^j$; in other words, such that  $d^i$ ``misses $i$'' and $s^j$ ``doubles $j$'' on the projection maps inducing these maps; compare, Bousfield-Kan \cite[XI.5]{Bousfield_Kan}. For a useful introduction in the dual setting of homotopy colimits, realization, and the simplicial replacement functor, in the context of spaces, see Dwyer-Henn \cite{Dwyer_Henn}.
\end{defn}

\begin{rem}
The basic idea behind the cosimplicial replacement $\prod^*_\DD X\in(\AlgO)^\Delta$ of a $\DD$-shaped diagram $X$ is that it arises as a natural cosimplicial resolution of $\lim_\DD X$; in other words, $\lim_\DD X$ is naturally isomorphic to an equalizer of the form
\begin{align*}
  \lim_\DD X\Iso
  \lim\Bigl(
  \xymatrix{
    \prod\limits_{a_0\in\DD}X(a_0)
    \ar@<0.5ex>[r]^-{d^0}\ar@<-0.5ex>[r]_-{d^1} &
    \prod\limits_{\substack{a_0\rightarrow a_1\\ \text{in}\,\DD}}X(a_1)
  }
  \Bigr)
  \Iso\xymatrix{\lim_\Delta\prod\nolimits^*_\DD X}
\end{align*}
This description of $\lim_\DD X$ naturally arises when verifying existence of $\lim_\DD X$ in terms of existence of small products and equalizers; i.e., the description that arises by writing down the desired universal property of the limiting cone of $\lim_\DD X$.
\end{rem}

\begin{prop}
\label{prop:fat_point_description_of_holim_bk}
Let $\capO$ be an operad in $\capR$-modules and $\DD$ a small category. Let $X\in(\AlgO)^\DD$. Then $\holim_\DD^\BK X$ is naturally isomorphic to the end construction
\begin{align*}
  \holim\nolimits_\DD^\BK X
  \Iso&\hombold_\DD(B(\DD/-),X)\\
  \Equal&\hombold(B(\DD/-),X)^\DD
\end{align*}
in $\AlgO$; the end construction $\hombold_\DD(B(\DD/-),X)$ can be thought of as the mapping object of $\DD$-shaped diagrams.
\end{prop}

\begin{rem}
Here, $\DD/-$ denotes the over category $\DD\!\downarrow\!-$ (or comma category) functor, $\function{B}{\Cat}{\sSet}$ the nerve functor, and $\Cat$ the category of small categories (see \cite[II.6]{MacLane_Categories} and \cite[I.1.4]{Goerss_Jardine}).
\end{rem}

\begin{proof}[Proof of Proposition \ref{prop:fat_point_description_of_holim_bk}]
This is proved in Bousfield-Kan \cite[XI.5, XI.3]{Bousfield_Kan} in the context of spaces, and the same argument verifies it in our context.
\end{proof}

\begin{prop}
\label{prop:holim_BK_adjunctions}
Let $\capO$ be an operad in $\capR$-modules and $\DD$ a small category. The Bousfield-Kan homotopy limit functor $\holim_\DD^\BK$ fits into an adjunction
\begin{align}
\label{eqref:holim_BK_adjunctions}
\xymatrix@1{
  \Alg_{\capO}\ar@<0.5ex>[rr]^-{-\tensordot B(\DD/-)} &&
  (\Alg_\capO)^\DD\ar@<0.5ex>[ll]^-{\holim_\DD^\BK}
}
\end{align}
with left adjoint on top. Furthermore, this adjunction is a Quillen adjunction with respect to the projective model structure on $\DD$-shaped diagrams induced from $\AlgO$.
\end{prop}

\begin{proof}
This is proved in Bousfield-Kan \cite[XI.3, XI.8]{Bousfield_Kan} in the context of spaces, and the same argument verifies it in our context. For instance, to verify it is a Quillen adjunction, it suffices to verify that cosimplicial replacement $\prod^*$ sends objectwise (acyclic) fibrations in $(\AlgO)^\DD$ to (acyclic) Reedy fibrations in $(\AlgO)^\Delta$; this is proved in \cite[XI.5.3]{Bousfield_Kan} (see also Jardine \cite{Jardine_galois_descent} for a useful development) and exactly the same argument verifies it in our context.
\end{proof}

\begin{defn}
Let $\capO$ be an operad in $\capR$-modules and $\DD$ a small category. The $\holim_\DD$ functor for $\DD$-shaped diagrams in $\AlgO$ is defined objectwise by
\begin{align*}
  \function{\holim\nolimits_\DD}{(\AlgO)^\DD}{\AlgO},\quad\quad
  &X\mapsto \holim\nolimits_\DD^\BK X^f
\end{align*}
where $X^f$ denotes a functorial fibrant replacement of $X$ in $(\AlgO)^\DD$ with respect to the projective model structure on $\DD$-shaped diagrams induced from $\AlgO$.
\end{defn}

\begin{rem}
It follows that there is a natural weak equivalence
\begin{align*}
  \holim\nolimits_\DD X\wequiv\RR\holim\nolimits^\BK_\DD X
\end{align*}
and if furthermore, $X$ is objectwise fibrant, then the natural map
\begin{align*}
  \holim\nolimits^\BK_\DD X\xrightarrow{\wequiv}\holim\nolimits_\DD X
\end{align*}
is a weak equivalence. Here, $\RR\holim^\BK_\DD$ denotes the total right derived functor of $\holim^\BK_\DD$.
\end{rem}

\subsection{Truncation filtration and the associated $\holim^\BK$ tower in $\AlgO$}

The simplicial category $\Delta$ has a natural filtration by its truncated subcategories $\Delta^{\leq n}$ of the form
\begin{align*}
  \emptyset\subset\Delta^{\leq 0}\subset\Delta^{\leq 1}
  \subset\cdots\subset\Delta^{\leq{n}}\subset
  \Delta^{\leq n+1}\subset\cdots\subset\Delta
\end{align*}
where $\Delta^{\leq n}\subset\Delta$ denotes the full subcategory of objects $[m]$ such that $m\leq n$; we use the convention that $\Delta^{\leq -1}=\emptyset$ is the empty category. This leads to the following $\holim^\BK$ tower of a cosimplicial $\capO$-algebra.

\begin{prop}
Let $\capO$ be an operad in $\capR$-modules. If $X$ is a cosimplicial $\capO$-algebra, then $\holim\nolimits^\BK_\Delta X$ is naturally isomorphic to a limit of the form
\begin{align*}
  \holim\nolimits^\BK_\Delta X &\Iso\lim
  \bigl({*}\larrow
  \holim\nolimits^\BK_{\Delta^{\leq 0}}X\larrow
  \holim\nolimits^\BK_{\Delta^{\leq 1}}X\larrow
  \holim\nolimits^\BK_{\Delta^{\leq 2}}X\larrow
  \cdots\bigr)
\end{align*}
in $\AlgO$; in particular, $\holim^\BK_{\Delta^{\leq -1}} X=*$ and $\holim^\BK_{\Delta^{\leq 0}} X\Iso X^0$. We usually refer to the tower $\{\holim\nolimits^\BK_{\Delta^{\leq s}}\}_{s\geq -1}$ as the $\holim^\BK$ tower of $X$.
\end{prop}

\begin{proof}
We know that $\Delta\Iso\colim_s\Delta^{\leq s}$, hence it follows that
\begin{align*}
  \xymatrix{\hombold_\Delta(\Delta[-],\prod\nolimits^*_\Delta X)}
    \Iso&\xymatrix{\lim_s\hombold_\Delta(\Delta[-],
  \prod\nolimits^*_{\Delta^{\leq s}} X)}
  \Equal\lim_s\holim\nolimits^\BK_{\Delta^{\leq s}} X
\end{align*}
Here, we have used the natural isomorphism $\prod^*_\Delta X\Iso\lim_s\prod^*_{\Delta^{\leq s}}X$ of cosimplicial $\capO$-algebras.
\end{proof}

\begin{prop}
\label{prop:full_subcategory_induced_map_on_cosimplicial_replacements}
Let $\capO$ be an operad in $\capR$-modules and $\DD$ a small category. If $\DD'\subset\DD$ is a subcategory and $X\in(\AlgO)^\DD$ is objectwise fibrant, then the natural map
\begin{align*}
  \xymatrix{\prod\nolimits^*_{\DD'} X\leftarrow
  \prod\nolimits^*_\DD X}
\end{align*}
in $(\AlgO)^\Delta$ is a Reedy fibration.
\end{prop}

\begin{proof}
This is an exercise left to the reader. It is closely related to Bousfield-Kan \cite[XI.5.2]{Bousfield_Kan} and the same argument works in our context; compare Dundas-Goodwillie-McCarthy \cite[A.7.2.4]{Dundas_Goodwillie_McCarthy}.
\end{proof}

\begin{prop}
\label{prop:full_subcategory_induced_map_on_holim_BK}
Let $\capO$ be an operad in $\capR$-modules and $\DD$ a small category. If $\DD'\subset\DD$ is a subcategory and $X\in(\AlgO)^\DD$ is objectwise fibrant, then the natural map
\begin{align*}
  \holim\nolimits^\BK_{\DD'} X\leftarrow
  \holim\nolimits^\BK_{\DD} X
\end{align*}
in $\AlgO$ is a fibration.
\end{prop}

\begin{proof}
This follows from Proposition \ref{prop:full_subcategory_induced_map_on_cosimplicial_replacements}.
\end{proof}

\begin{prop}
Let $\capO$ be an operad in $\capR$-modules.
\begin{itemize}
\item[(a)] If $X\in(\AlgO)^\Delta$ is objectwise fibrant, then the tower $\{\holim\nolimits^\BK_{\Delta^{\leq s}} X\}_{s\geq -1}$ is a tower of fibrations, and the natural map $\holim^\BK_\Delta X\rarrow\holim\nolimits^\BK_{\Delta^{\leq s}} X$ is a fibration for each $s\geq -1$; in particular, $\holim^\BK_\Delta X$ is fibrant.
\item[(b)] If $X\rarrow X'$ in $(\AlgO)^\Delta$ is a weak equivalence between objectwise fibrant objects, then $\holim^\BK_\Delta X\rarrow\holim^\BK_\Delta X'$ is a weak equivalence.
\end{itemize}
\end{prop}

\begin{proof}
Part (a) follows from Proposition  \ref{prop:full_subcategory_induced_map_on_holim_BK} and part (b) follows from the Quillen adjunction \eqref{eqref:holim_BK_adjunctions}; see, for instance, Dwyer-Spalinski \cite[9.8, 9.9]{Dwyer_Spalinski}; it can also be argued exactly as in \cite[XI.5.4]{Bousfield_Kan}.
\end{proof}

The following proposition is proved in Dror-Dwyer \cite[3.17]{Dror_Dwyer_long_homology}.

\begin{prop}
\label{prop:left_cofinal_delta_restricted_to_delta}
The inclusion of categories $\Delta_\res\subset\Delta$ is left cofinal (i.e.,  homotopy initial).
\end{prop}

\begin{prop}
\label{prop:comparing_holim_with_Tot_and_Tot_restricted}
If $X\in(\AlgO)^\Delta$ is objectwise fibrant and $Y\in(\AlgO)^\Delta$ is Reedy fibrant, then the natural maps
\begin{align*}
  \holim\nolimits_\Delta^\BK X\xrightarrow{\wequiv}
  \holim\nolimits_{\Delta_\res}^\BK X
  \xleftarrow{\wequiv}\Tot^\res X,\quad\quad
  \holim\nolimits_{\Delta}^\BK Y\xleftarrow{\wequiv}\Tot Y,
\end{align*}
in $\AlgO$ are weak equivalences.
\end{prop}

\begin{proof}
This follows from left cofinality (Proposition \ref{prop:left_cofinal_delta_restricted_to_delta}) and Bousfield-Kan \cite{Bousfield_Kan} in the context of spaces, and exactly the same argument verifies it in our context.
\end{proof}

\subsection{Skeletal filtration and the associated $\Tot$ towers in $\AlgO$}

\begin{defn}
\label{defn:totalization_functors}
Let $\capO$ be an operad in $\capR$-modules and $s\geq -1$. The functors $\Tot_s$ and $\Tot^\res_s$ are defined objectwise by the ends
\begin{align*}
  \function{\Tot_s}{(\AlgO)^\Delta}{\AlgO},\quad\quad
  &X\mapsto\hombold(\Sk_s\Delta[-],X)^\Delta\\
  \function{\Tot^\res_s}{(\AlgO)^{\Delta_\res}}{\AlgO},\quad\quad
  &Y\mapsto\hombold(\Sk_s\Delta[-],Y)^{\Delta_\res}
\end{align*}
Here we use the convention that the $(-1)$-skeleton of a simplicial set is the empty simplicial set. In particular, $\Sk_{-1}\Delta[n]=\emptyset$ for each $n\geq 0$; it follows immediately that $\Tot_{-1}(X)=*$ and $\Tot^\res_{-1}(Y)=*$.
\end{defn}

\begin{prop}
\label{prop:adjunctions_stages_of_tot_towers}
Let $\capO$ be an operad in $\capR$-modules and $s\geq 0$. The functors $\Tot_s$ and $\Tot^\res_s$ fit into adjunctions
\begin{align*}
\xymatrix@1{
  \Alg_{\capO}\ar@<0.5ex>[rr]^-{-\tensordot\Sk_s\Delta[-]} &&
  (\Alg_\capO)^\Delta\ar@<0.5ex>[ll]^-{\Tot_s}
},\quad\quad
\xymatrix@1{
  \Alg_{\capO}\ar@<0.5ex>[rr]^-{-\tensordot\Sk_s\Delta[-]} &&
  (\Alg_\capO)^{\Delta_\res}\ar@<0.5ex>[ll]^-{\Tot^\res_s}
}
\end{align*}
with left adjoints on top. It follows that there are natural isomorphisms $\Tot_0(X)\Iso X^0$ and $\Tot^\res_0(Y)\Iso Y^0$.
\end{prop}

\begin{proof}
The first part follows exactly as in the proof of Proposition \ref{prop:tot_adjunctions}. The second part follows from uniqueness of right adjoints, up to isomorphism. For instance, to verify the natural isomorphism $\Tot_0(X)\Iso X^0$, it suffices to verify that giving a map $A\rarrow X^0$ in $\AlgO$ is the same as giving a map $A\tensordot\Sk_0\Delta[-]\rarrow X$ in $(\AlgO)^\Delta$; this follows from the universal property of ends together with the natural isomorphisms\begin{align*}
  \xymatrix{A\tensordot\Sk_0\Delta[n]\Iso\coprod\limits_{\substack{[0]\rightarrow [n]\\ \text{in}\,\Delta}}A}
\end{align*}
in $\AlgO$. The case of $\Tot^\res_0(Y)\Iso Y$ is similar by replacing $\Delta$ with $\Delta_\res$.
\end{proof}

The following can be thought of as the analog of the $\holim^\BK$ tower of a cosimplicial $\capO$-algebra; under appropriate fibrancy conditions these towers are naturally weakly equivalent.

\begin{prop}
Let $\capO$ be an operad in $\capR$-modules. Let $X$ be a cosimplicial (resp. restricted cosimplicial) $\capO$-algebra. The totalization $\Tot(X)$ (resp. $\Tot^\res(X)$) is naturally isomorphic to a limit of the form
\begin{align*}
  \Tot(X)&\Iso\lim
  \bigl({*}\larrow\Tot_0(X)\larrow\Tot_1(X)\larrow\Tot_2(X)\larrow\cdots\bigr)\\
  \text{resp.}\quad
  \Tot^\res(X)&\Iso\lim
  \Bigl({*}\larrow\Tot^\res_0(X)\larrow\Tot^\res_1(X)\larrow\Tot^\res_2(X)\larrow\cdots
  \Bigr)
\end{align*}
in $\AlgO$. We often refer to the tower $\{\Tot_s(X)\}_{s\geq -1}$ (resp. $\{\Tot^\res_s(X)\}_{s\geq -1}$) as the $\Tot$ tower (resp. $\Tot^\res$ tower) of $X$ .
\end{prop}

\begin{proof}
We know that $\Delta[-]\Iso\colim_s\Sk_s\Delta[-]$ in $\sSet^\Delta$. Since the contravariant functor $\function{\hombold(-,X)^\Delta}{\sSet^\Delta}{\AlgO}$ sends colimiting cones to limiting cones (see Remark \ref{rem:useful_adjunction_isomorphisms_simplicial_structure}), it follows that there are natural isomorphisms
\begin{align*}
  \hombold(\Delta[-],X)^\Delta\Iso
  \hombold(\colim_s\Sk_s\Delta[-],X)^\Delta\Iso
  \lim_s\hombold(\Sk_s\Delta[-],X)^\Delta
\end{align*}
which finishes the proof for the case of $\Tot(X)$. The case of $\Tot^\res(X)$ is similar by replacing $\Delta$ with $\Delta_\res$.
\end{proof}

The following two propositions, which construct $\Tot_s$ from $\Tot_{s-1}$ (resp. $\Tot^\res_s$ from $\Tot^\res_{s-1}$), play a key role in the homotopical analysis of the totalization functors in this paper; it is useful to contrast the two pullback constructions as a helpful way to understand the difference between $\Tot$ and $\Tot^\res$.

\begin{prop}
\label{prop:key_pullback_diagram_for_Tot_tower}
Let $\capO$ be an operad in $\capR$-modules. Let $X$ be a cosimplicial $\capO$-algebra and $s\geq 0$. There is a pullback diagram of the form
\begin{align*}
\xymatrix{
  \Tot_s(X)\ar[d]\ar[r] & \hombold(\Delta[s],X^s)\ar[d]\\
  \Tot_{s-1}(X)\ar[r] &
  \hombold(\partial\Delta[s],X^s)\times_{\hombold(\partial\Delta[s],M^{s-1}X)}
  \hombold(\Delta[s],M^{s-1}X)
}
\end{align*}
in $\AlgO$. Here, we are using the matching object notation $M^{s-1}X$ for a cosimplicial object (see \cite[VIII.1]{Goerss_Jardine}); in particular, $M^{-1}X=*$ and $M^{0}X\Iso X^0$.
\end{prop}

\begin{proof}
This follows from Goerss-Jardine \cite[VIII.1]{Goerss_Jardine}.
\end{proof}

\begin{prop}
\label{prop:key_pullback_diagram_for_restricted_Tot_tower}
Let $\capO$ be an operad in $\capR$-modules. Let $Y$ be a restricted cosimplicial $\capO$-algebra and $s\geq 0$. There is a pullback diagram of the form
\begin{align*}
\xymatrix{
  \Tot^\res_s(Y)\ar[d]\ar[r] & \hombold(\Delta[s],Y^s)\ar[d]\\
  \Tot^\res_{s-1}(Y)\ar[r] & \hombold(\partial\Delta[s],Y^s)
}
\end{align*}
in $\AlgO$.
\end{prop}

\begin{proof}
This follows from Goerss-Jardine \cite[VIII.1]{Goerss_Jardine} because of the following: By uniqueness of right adjoints, up to isomorphism, $\Tot^\res(Y)$ is naturally isomorphic to $\Tot$ composed with the right Kan extension of $Y$ along the inclusion $\Delta_\res\subset\Delta$.
\end{proof}

\begin{rem}
In the special case of $s=0$, note that Propositions \ref{prop:key_pullback_diagram_for_Tot_tower} and \ref{prop:key_pullback_diagram_for_restricted_Tot_tower} are simply the assertions that the natural maps
\begin{align*}
  \Tot_0(X)\rarrow\hombold(\Delta[0],X^0)\Iso X^0,\quad\quad
  \Tot^\res_0(Y)\rarrow\hombold(\Delta[0],Y^0)\Iso Y^0,
\end{align*}
are isomorphisms (see Proposition \ref{prop:adjunctions_stages_of_tot_towers}); this is because it follows from our conventions (see Definition \ref{defn:totalization_functors}) that $\Tot_{-1}(X)=*$ and $\Tot^\res_{-1}(Y)=*$.
\end{rem}

The following proposition serves to contrast the properties of $\Tot$ and $\Tot^\res$ in the context of $\capO$-algebras. For the convenience of the reader, and because these properties are fundamental to the main results of this paper, we include a concise proof in Section \ref{sec:proofs}.

\begin{prop}
\label{prop:basic_properties_of_Tot_towers_etc}
Let $\capO$ be an operad in $\capR$-modules.
\begin{itemize}
\item[(a)] If $X\in(\AlgO)^\Delta$ is Reedy fibrant, then the $\Tot$ tower $\{\Tot_s(X)\}$ is a tower of fibrations, and the natural map $\Tot(X)\rarrow\Tot_s(X)$ is a fibration for each $s\geq -1$; in particular, $\Tot(X)$ is fibrant.
\item[(b)] If $Y\in(\AlgO)^{\Delta_\res}$ is objectwise fibrant, then the $\Tot^\res$ tower $\{\Tot^\res_s(Y)\}$ is a tower of fibrations, and the natural map $\Tot^\res(Y)\rarrow\Tot^\res_s(Y)$ is a fibration for each $s\geq -1$; in particular, $\Tot^\res(Y)$ is fibrant.
\item[(c)] If $X\rarrow X'$ in $(\AlgO)^\Delta$ is a weak equivalence between Reedy fibrant objects, then $\Tot(X)\rarrow\Tot(X')$ is a weak equivalence.
\item[(d)] If $Y\rarrow Y'$ in $(\AlgO)^{\Delta_\res}$ is a weak equivalence between objectwise fibrant diagrams, then $\Tot^\res(Y)\rarrow\Tot^\res(Y')$ is a weak equivalence.
\end{itemize}
\end{prop}

\subsection{Comparing the $\holim^\BK$ and $\Tot$ towers in $\AlgO$}

The following is proved in Carlsson \cite[Section 6]{Carlsson} and Sinha \cite[6.7]{Sinha_cosimplicial_models}, and plays a key role in this paper.

\begin{prop}
\label{prop:left_cofinality_truncated_delta}
Let $n\geq 0$. The composite
\begin{align*}
  \capP_0([n])\Iso P\Delta[n]\longrightarrow\Delta_\res^{\leq n}
  \subset\Delta^{\leq n}
\end{align*}
is left cofinal (i.e., homotopy initial). Here, $\capP_0([n])$ denotes the poset of all nonempty subsets of $[n]$; this notation agrees with Goodwillie \cite{Goodwillie_calculus_2} and Ching-Harper \cite[3.1]{Ching_Harper}.
\end{prop}

\begin{rem}
Note that $\capP_0([n])\Iso P\Delta[n]$ naturally arises as the category whose nerve is the subdivision $\mathrm{sd}\Delta[n]$ of $\Delta[n]$, and that applying realization recovers the usual barycentric subdivision of $\Delta^n$ (see Goerss-Jardine \cite[III.4]{Goerss_Jardine}).
\end{rem}

\begin{rem}
One can ask if left cofinality of the inclusion $\Delta_\res\subset\Delta$  (Proposition \ref{prop:left_cofinal_delta_restricted_to_delta}) remains true when $\Delta$ is replaced by its $n$-truncation $\Delta^{\leq n}$. It turns out that the inclusion $\Delta^{\leq n}_\res\subset\Delta^{\leq n}$ is not left cofinal for $n\geq 1$; Proposition \ref{prop:left_cofinality_truncated_delta} can be thought of as the $n$-truncated analog of Proposition \ref{prop:left_cofinal_delta_restricted_to_delta}.

The upshot is that if one further restricts to the poset $P\Delta[n]$ of non-degenerate simplices of $\Delta[n]$, ordered by the face relation (see \cite[III.4]{Goerss_Jardine}), then the resulting composite $P\Delta[n]\rarrow\Delta^{\leq n}$ is left cofinal. This observation plays a key role in this paper; see Proposition \ref{prop:punctured_cube_calculation_of_holim_truncated_delta}.
\end{rem}

\begin{prop}
\label{prop:punctured_cube_calculation_of_holim_truncated_delta}
Let $\capO$ be an operad in $\capR$-modules. If $X\in(\AlgO)^\Delta$ is objectwise fibrant and $Y\in(\AlgO)^\Delta$ is Reedy fibrant, then the natural maps
\begin{align*}
  \holim\nolimits_{\Delta^{\leq n}}^\BK X&\xrightarrow{\wequiv}
  \holim\nolimits_{P\Delta[n]}^\BK X\Iso
  \holim\nolimits_{\capP_0([n])}^\BK X\\
  \Tot_n Y&\xrightarrow{\wequiv}
  \holim\nolimits_{\Delta^{\leq n}}^\BK Y
\end{align*}
in $\AlgO$ are weak equivalences.
\end{prop}

\begin{proof}
This appears in Carlsson \cite{Carlsson} and Sinha \cite{Sinha_cosimplicial_models} in the contexts of spectra and spaces, respectively, and remains true in our context. The first case follows from Proposition \ref{prop:left_cofinality_truncated_delta} and Bousfield-Kan \cite{Bousfield_Kan}, and the second case follows from Bousfield-Kan \cite{Bousfield_Kan}, where the indicated natural weak equivalence is the composite
\begin{align*}
  \hombold_\Delta(\Sk_n\Delta[-],Y)
  \Iso&\hombold_{\Delta^{\leq n}}(\Delta^{\leq n}[-],Y)\\
  \xrightarrow{\wequiv}&\hombold_{\Delta^{\leq n}}(B(\Delta^{\leq n}/-),Y)\Iso\holim\nolimits^\BK_{\Delta^{\leq n}}Y
\end{align*}
in $\AlgO$; note that the natural map
$B(\Delta^{\leq n}/-)\xrightarrow{\wequiv}\Delta^{\leq n}[-]$ in $(\sSet)^{\Delta^{\leq n}}$ is a weak equivalence between Reedy cofibrant objects.
\end{proof}

\subsection{Homotopy spectral sequence of a tower of $\capO$-algebras}
\label{sec:homotopy_spectral_sequence}

Consider any tower $\{X_s\}$ of fibrations of $\capO$-algebras such that $X_{-1}=*$ and denote by $F_s\subset X_s$ the fiber of $X_s\rarrow X_{s-1}$, where $F_0=X_0$.

\begin{rem}
For ease of notational purposes, we will regard such towers as indexed by the integers such that $X_s=*$ (and hence $F_s=*$) for every $s<0$.
\end{rem}

Following the notation in Bousfield-Kan \cite[IX.4]{Bousfield_Kan} as closely as possible, recall that the collection of \emph{$r$-th derived long exact sequences}, $r\geq 0$,  associated to the collection of homotopy fiber sequences $F_s\rightarrow X_s\rightarrow X_{s-1}$, $s\in\ZZ$, has the form
\begin{align}
\label{eq:derived_long_exact_sequences}
   \cdots\rarrow\pi_{i+1}F_{s-r}^{(r)}&\rarrow
   \pi_{i+1}X_{s-r}^{(r)}\rarrow
   \pi_{i+1}X_{s-r-1}^{(r)}\rarrow\\
   \notag
   \pi_i F_s^{(r)}&\rarrow
   \pi_i X_s^{(r)}\rarrow
   \pi_iX_{s-1}^{(r)}\rarrow\\
   \notag
   \pi_{i-1}F_{s+r}^{(r)}&\rarrow
   \pi_{i-1}X_{s+r}^{(r)}\rarrow
   \pi_{i-1}X_{s+r-1}^{(r)}\rarrow\cdots
\end{align}
where we define
\begin{align}
  \label{eq:derived_groups}
  \pi_i X_s^{(r)} &:= \im(\pi_iX_s\larrow\pi_iX_{s+r}),
  \quad\quad r\geq 0,\\
  \label{eq:iterated_homologies_associated_to_exact_couple}
  \pi_i F_s^{(r)} &:= \frac{\ker(\pi_iF_s\rarrow\pi_iX_s/\pi_iX_s^{(r)})}
  {\partial_*\ker(\pi_{i+1}X_{s-1}\rarrow\pi_{i+1}X_{s-(r+1)})},
  \quad\quad r\geq 0,
\end{align}
and
$  \pi_i X_s^{(0)} = \pi_i X_s$ and
$  \pi_i F_s^{(0)} = \pi_i F_s$, for each $i\in\ZZ$. Here, we denote by $\partial$ the natural boundary maps appearing in the $0$-th derived long exact sequences (i.e., the long exact sequences in $\pi_*$) associated to each $F_s\rarrow X_s\rarrow X_{s-1}$.

In other words, the collection of long exact sequences in $\pi_*$ associated to the collection of homotopy fiber sequences $F_s\rarrow X_s\rarrow X_{s-1}$, $s\in\ZZ$, gives an exact couple $(\pi_*X_*, \pi_*F_*)$ of $\ZZ$-bigraded abelian groups of the form
\begin{align}
\label{eq:exact_couples}
\xymatrix{
  \pi_*X_*\ar[r] & \pi_* X_*\ar[d]\\
  & \pi_*F_*\ar@/^0.5pc/[ul]
}\quad\quad
\xymatrix{
  D\ar[r]^{(1,-1)} & D\ar[d]^{(0,0)}\\
  & E\ar@/^0.5pc/[ul]^{(-1,0)}
}
\end{align}
If we denote this exact couple by $(D,E)$, with bigradings defined by $E_{-s,t}:=\pi_{t-s}F_s$ and $D_{-s,t}:=\pi_{t-(s-1)}X_{s-1}$, then it is easy to verify that the three maps depicted on the right-hand side of \eqref{eq:exact_couples} have the indicated bidegrees. It follows that the associated collection of $r$-th derived exact couples, $r\geq 0$,  determines a left-half plane homologically graded spectral sequence $(E^r,d^r)$, $r\geq 1$ (see MacLane \cite[XI.5]{MacLane_homology} for a useful development); in particular, $d^r$ has bidegree $(-r, r-1)$. We will sometimes denote $E^r$ by $E^r\{X_s\}$ to emphasize the tower $\{X_s\}$ in the notation.

This is the \emph{homotopy spectral sequence of the tower of fibrations} $\{X_s\}$ of $\capO$-algebras and satisfies
\begin{align}
\label{eq:the_Er_page}
  E^r_{-s,s+i} := \pi_{i}F_s^{(r-1)},\quad\quad
  &\bigl(\text{equiv.}\quad\quad E^r_{-s,t} := \pi_{t-s}F_s^{(r-1)}\bigr), \quad\quad r\geq 1,
\end{align}
with differentials
$
  \function{d^r}{E^r_{-s,s+i}}{E^r_{-(s+r),s+r+i-1}}
$
given by the composite maps
\begin{align}
\label{eq:the_dr_differentials}
  \pi_i F_s^{(r-1)}\rarrow
  \pi_i X_s^{(r-1)}\rarrow
  \pi_{i-1}F_{s+r}^{(r-1)}.
\end{align}
It is essentially identical to the homotopy spectral sequence of a tower of fibrations in pointed spaces described in Bousfield-Kan \cite{Bousfield_Kan}, except it is better behaved in the sense that is has no ``fringing'' (e.g., Bousfield \cite{Bousfield_obstruction}, Goerss-Jardine \cite[VI.2]{Goerss_Jardine}).

\begin{defn}
\label{defn:model_structure_on_towers}
Let $\M$ be a model category with all small limits and let $\DD$ be the category $\{0\leftarrow 1\leftarrow 2\leftarrow\cdots\}$ with objects the non-negative integers and a single morphism $i\leftarrow j$ for each $i\leq j$. Consider the category $\M^\DD$ of $\DD$-shaped diagrams (or towers) in $\M$ with the injective model structure \cite[VI.1.1]{Goerss_Jardine}. The \emph{homotopy limit} functor $\function{\holim}{\Ho({\M^\DD})}{\Ho(\M)}$ is the total right derived functor of the limit functor $\function{\lim}{\M^\DD}{\M}$.
\end{defn}

\begin{defn}
Let $\capO$ be an operad in $\capR$-modules. Suppose $\{X_s\}$ is a tower in $\AlgO$. The \emph{homotopy spectral sequence $(E^r,d^r)$, $r\geq 1$, associated to the tower} $\{X_s\}$ is the homotopy spectral sequence of the functorial fibrant replacement $\{X_s\}^f$ of $\{X_s\}$ in the category of towers in $\AlgO$. We sometimes denote $E^r$ by $E^r\{X_s\}$ to emphasize in the notation the tower $\{X_s\}$.
\end{defn}

The following Milnor type short exact sequences are well known in stable homotopy theory (see, for instance, Dwyer-Greenlees \cite{Dwyer_Greenlees_Iyengar}); they can be established as a consequence of Bousfield-Kan \cite[IX]{Bousfield_Kan}.

\begin{prop}
\label{prop:short_exact_sequence}
Consider any tower $B_0\leftarrow B_1\leftarrow B_2\leftarrow\cdots$ of $\capO$-algebras. There are natural short exact sequences
\begin{align*}
  &0\rightarrow\lim\nolimits^1_k\pi_{i+1}B_k\rightarrow
  \pi_i\holim_k B_k\rightarrow
  \lim\nolimits_k\pi_i B_k\rightarrow 0.
\end{align*}
\end{prop}

The following is a spectral algebra analog of the Bousfield-Kan connectivity lemma \cite[IX.5.1]{Bousfield_Kan} for spaces.

\begin{prop}
\label{prop:connectivity_lemma}
Let $\capO$ be an operad in $\capR$-modules. Suppose $\{X_s\}$ is a tower in $\AlgO$. Let $n\in\ZZ$ and $r\geq 1$. If $E^r_{-s,s+i}=0$ for each $i\leq n$ and $s$, then
\begin{itemize}
\item[(a)] $\holim_sX_s$ is $n$-connected,
\item[(b)] $\lim_s\pi_iX_s = 0 = \lim\nolimits^1_s\pi_{i+1}X_s$ for each $i\leq n$.
\end{itemize}
\end{prop}

\begin{proof}
This is proved in Bousfield-Kan \cite[IX.5.1]{Bousfield_Kan} in the context of pointed spaces, and exactly the same argument verifies it in our context.
\end{proof}

The following is a spectral algebra analog of the Bousfield-Kan mapping lemma \cite[IX.5.2]{Bousfield_Kan} for spaces.

\begin{prop}
\label{prop:mapping_lemma}
Let $\capO$ be an operad in $\capR$-modules. Let $\function{f}{\{X_s\}}{\{X'_s\}}$ be a map of towers in $\AlgO$. Let $r\geq 1$. If $f$ induces an $E^r$-isomorphism $E^r\{X_s\}\Iso E^r\{X'_s\}$ between homotopy spectral sequences, then
\begin{itemize}
\item[(a)] $f$ induces a weak equivalence $\holim_sX_s\wequiv\holim_sX'_s$,
\item[(b)] $f$ induces isomorphisms
\begin{align*}
  \lim_s\pi_iX_s\Iso\lim_s\pi_iX'_s,
  \quad\quad
  \lim\nolimits^1_s\pi_iX_s\Iso\lim\nolimits^1_s\pi_iX'_s,
  \quad\quad i\in\ZZ,
\end{align*}
\item[(c)] $f$ induces a pro-isomorphism $\{\pi_iX_s\}\rarrow\{\pi_iX'_s\}$ for each $i\in\ZZ$.
\end{itemize}
\end{prop}

\begin{proof}
This is proved in Bousfield-Kan \cite[IX.5.2]{Bousfield_Kan} in the context of pointed spaces, and exactly the same argument verifies it in our context.
\end{proof}

The following can be thought of as a relative connectivity lemma for towers of $\capO$-algebras; it is closely related to Bousfield-Kan \cite[I.6.2, IV.5.1]{Bousfield_Kan} in the context of spaces.

\begin{prop}
\label{prop:relative_connectivity_lemma}
Let $\capO$ be an operad in $\capR$-modules. Let $\function{f}{\{X_s\}}{\{X'_s\}}$ be a map of towers in $\AlgO$. Let $n\in\ZZ$ and $r\geq 1$. Assume that $f$ induces isomorphisms $E^r_{-s,s+i}\Iso {E'}^r_{-s,s+i}$ for each $i\leq n-1$ and $s$, and surjections $E^r_{-s,s+i}\rarrow {E'}^r_{-s,s+i}$ for $i=n$ and each $s$. Then
\begin{itemize}
\item[(a)] $f$ induces an $(n-1)$-connected map $\holim_s X_s\rarrow\holim_s X'_s$,
\item[(b)] $f$ induces a surjection
$
\lim\nolimits^1_s\pi_n X_s\rarrow\lim\nolimits^1_s\pi_n X'_s
$
and isomorphisms
\begin{align*}
  \lim_s\pi_iX_s\Iso\lim_s\pi_iX'_s,
  \quad\quad
  \lim\nolimits^1_s\pi_iX_s\Iso\lim\nolimits^1_s\pi_iX'_s,
  \quad\quad i\leq n-1.
\end{align*}
\end{itemize}
If furthermore, the towers $\{\pi_nX_s\},\{\pi_nX'_s\},\{\pi_{n+1}X'_s\}$ are pro-constant, then
\begin{itemize}
\item[(c)] $f$ induces an $n$-connected map $\holim_s X_s\rarrow\holim_s X'_s$,
\item[(d)] $f$ induces a surjection
$
\lim\nolimits^1_s\pi_{n+1}X_s\rarrow\lim\nolimits^1_s\pi_{n+1}X'_s
$
and isomorphisms
\begin{align*}
  \lim_s\pi_iX_s\Iso\lim_s\pi_iX'_s,
  \quad\quad
  \lim\nolimits^1_s\pi_iX_s\Iso\lim\nolimits^1_s\pi_iX'_s,
  \quad\quad i\leq n.
\end{align*}
\end{itemize}
Here, ${E'}^r$ denotes $E^r\{X'_s\}$ and $\holim_s X_s\rarrow\holim_s X'_s$ denotes the induced zigzag in the category of $\capO$-algebras with all backward facing maps weak equivalences.
\end{prop}

\begin{proof}
This is the spectral algebra analog of the argument that underlies \cite[I.6.2]{Bousfield_Kan} and \cite[IV.5.1]{Bousfield_Kan}; it follows from \cite[III.2]{Bousfield_Kan} and \cite[IX]{Bousfield_Kan}.
\end{proof}

\section{Homotopy theory of $\K$-coalgebras}
\label{sec:homotopy_theory_K_coalgebras}

The purpose of this section is to setup, in the context of symmetric spectra used in this paper, the homotopy theory of $\K$-coalgebras developed in Arone-Ching \cite{Arone_Ching_classification}.

\begin{defn}
\label{defn:model_cat_coAlgK_language} A morphism in $\coAlgK$ is a \emph{cofibration} if the underlying morphism in $\AlgJ$ is a cofibration. An object $Y$ in $\coAlgK$ is \emph{cofibrant} if the unique map $\emptyset\rarrow Y$ in $\coAlgK$ is a cofibration.
\end{defn}

\begin{rem}
In $\coAlgK$ the initial object $\emptyset$ and the terminal object $*$ are isomorphic; here, the terminal object is the trivial $\K$-coalgebra with underlying object $*$. This follows from the basic assumption that $\capO[0]=*$.
\end{rem}

Recall that a morphism of $\K$-coalgebras from $Y$ to $Y'$ is a map $\function{f}{Y}{Y'}$ in $\Alg_J$ that makes the diagram
\begin{align}
\label{eq:commutative_square_defining_K_coalgebra_map}
\xymatrix{
  Y\ar[d]_-{f}\ar[r]^-{m} & \K Y\ar[d]^-{\id f}\\
  Y'\ar[r]_-{m} & \K Y'
}
\end{align}
in $\AlgJ$ commute. This motivates the following cosimplicial resolution of the set of $\K$-coalgebra maps from $Y$ to $Y'$.

\begin{defn}
\label{defn:derived_K_coalgebra_maps_second_attempt}
Let $Y,Y'$ be cofibrant $\K$-coalgebras. The cosimplicial object $\Hombold_\AlgJ(Y,\K^\bullet Y')$ in $\sSet$ looks like (showing only the coface maps)
\begin{align}
\label{eq:derived_K_coalgebra_maps_second_attempt}
\xymatrix{
  \Hombold_\AlgJ(Y,Y')\ar@<0.5ex>[r]^-{d^0}\ar@<-0.5ex>[r]_-{d^1} &
  \Hombold_\AlgJ(Y,\K Y')
  \ar@<1.0ex>[r]\ar[r]\ar@<-1.0ex>[r] &
  \Hombold_{\AlgJ}(Y,\K\K Y')\cdots
}
\end{align}
and is defined objectwise by $\Hombold_\AlgJ(Y,\K^\bullet Y')^n:=\Hombold_\AlgJ(Y,\K^n Y')$ with the obvious coface and codegeneracy maps induced by the comultiplication and coaction maps, and counit map, respectively; see, Arone-Ching \cite[1.3]{Arone_Ching_classification}.
\end{defn}

The basic idea is that in simplicial degree $0$, the maps $d^0,d^1$ in \eqref{eq:derived_K_coalgebra_maps_second_attempt} send $f$ to the right-hand and left-hand composites in diagram \eqref{eq:commutative_square_defining_K_coalgebra_map}, respectively, and that furthermore, the pair of maps $d^0,d^1$ naturally extend to a cosimplicial diagram; by construction, taking the limit recovers the set of $\K$-coalgebra maps from $Y$ to $Y'$ (Remark \ref{rem:intuition_behind_resolution}).

But there is a difficulty that arises in our context: $\K^nY'$ may not be fibrant in $\AlgJ$, and hence \eqref{eq:derived_K_coalgebra_maps_second_attempt} is not homotopically meaningful. This is easily corrected by appropriately ``fattening up'' the entries in the codomains without destroying the cosimplicial structure. The following can be thought of as a homotopically meaningful cosimplicial resolution of $\K$-coalgebra maps; it is obtained by appropriately modifying \eqref{eq:derived_K_coalgebra_maps_second_attempt} with the simplicial fibrant replacement monad $F$.

\begin{defn}
\label{defn:derived_K_coalgebra_maps}
Let $Y,Y'$ be cofibrant $\K$-coalgebras. The cosimplicial object $\Hombold_\AlgJ\bigl(Y,(F\K)^\bullet FY'\bigr)$ in $\sSet$ looks like (showing only the coface maps)
\begin{align*}
\xymatrix{
  \Hombold_\AlgJ(Y,FY')\ar@<0.5ex>[r]^-{d^0}\ar@<-0.5ex>[r]_-{d^1} &
  \Hombold_\AlgJ\bigl(Y,(F\K) FY'\bigr)
  \ar@<1.0ex>[r]\ar[r]\ar@<-1.0ex>[r] &
  \Hombold_{\AlgJ}\bigl(Y,(F\K)^2 FY'\bigr)\cdots
}
\end{align*}
and is defined objectwise by $\Hombold_\AlgJ\bigl(Y,(F\K)^\bullet FY'\bigr)^n:=\Hombold_\AlgJ\bigl(Y,(F\K)^n FY'\bigr)$ with the obvious coface and codegeneracy maps.
\end{defn}

These two resolutions can be compared as follows.

\begin{prop}
\label{prop:map_induced_by_unit_on_mapping_spaces}
Let $Y,Y'$ be cofibrant $\K$-coalgebras. The unit map $\function{\eta}{\id}{F}$ induces a well-defined map
\begin{align}
\label{eq:level_zero_of_mapping_space_resolutions}
\hom_\AlgJ(Y,\K^\bullet Y')\longrightarrow
&\hom_\AlgJ\bigl(Y,(F\K)^\bullet FY'\bigr)\\
\label{eq:mapping_space_resolutions}
\text{resp.}\quad
\Hombold_\AlgJ(Y,\K^\bullet Y')\longrightarrow
&\Hombold_\AlgJ\bigl(Y,(F\K)^\bullet FY'\bigr)
\end{align}
of cosimplicial objects in $\Set$ (resp. in $\sSet$), natural in all such $Y,Y'$. Note that \eqref{eq:level_zero_of_mapping_space_resolutions} is obtained from \eqref{eq:mapping_space_resolutions} by evaluating at simplicial degree $0$.
\end{prop}

\begin{proof}
This is an exercise left to the reader.
\end{proof}

Recall the usual notion of realization of a simplicial set, regarded as taking values in the category of compactly generated Hausdorff spaces, denoted $\CGHaus$ (e.g., Goerss-Jardine \cite{Goerss_Jardine}).

\begin{defn}
\label{defn:realization_sSet}
The \emph{realization} functor $|-|$ for simplicial sets is defined objectwise by the coend
\begin{align*}
  \function{|-|}{\sSet}{\CGHaus},\quad\quad
  X\mapsto X \times_{\Delta}\Delta^{(-)}
\end{align*}
Here, $\Delta^n$ in $\CGHaus$ denotes the topological standard $n$-simplex for each $n\geq 0$ (see Goerss-Jardine \cite[I.1.1]{Goerss_Jardine}).
\end{defn}

\begin{defn}
Let $X,Y$ be $\capO$-algebras. The mapping space $\Map_\AlgO(X,Y)$ in $\CGHaus$ is defined by realization
\begin{align*}
  \Map_{\AlgO}(X,Y)
  &:=|\Hombold_\AlgO(X,Y)|
\end{align*}
of the indicated simplicial set.
\end{defn}

The following definition of the mapping space of derived $\K$-coalgebra maps appears in Arone-Ching \cite[1.10]{Arone_Ching_classification} and is a key ingredient in both the statements and proofs of our main results. Even though the simplicial fibrant replacement monad $F$ was not required in  \cite{Arone_Ching_classification}, the arguments and constructions there easily carry over to our context, provided that we set things up appropriately.

\begin{defn}
Let $Y,Y'$ be cofibrant $\K$-coalgebras. The \emph{mapping spaces} of derived $\K$-coalgebra maps $\Hombold_{\coAlgK}(Y,Y')$ in $\sSet$ and $\Map_{\coAlgK}(Y,Y')$ in $\CGHaus$ are defined by the restricted totalizations
\begin{align*}
  \Hombold_{\coAlgK}(Y,Y')
  &:=\Tot^\res\Hombold_\AlgJ\bigl(Y,(F\K)^\bullet FY'\bigr)\\
  \Map_{\coAlgK}(Y,Y')
  &:=\Tot^\res\Map_\AlgJ\bigl(Y,(F\K)^\bullet FY'\bigr)
\end{align*}
of the indicated cosimplicial objects.
\end{defn}

\begin{rem}
\label{rem:intuition_behind_resolution}
To help understand why this is an appropriate notion of the space of derived $\K$-coalgebra maps in our setting, note that there are natural isomorphisms and natural zigzags of weak equivalences
\begin{align*}
  \hom_{\coAlgK}(Y,Y')&\Iso\lim_\Delta\hom_\AlgJ\bigl(Y,\K^\bullet Y'\bigr),\\
  \Hombold_{\coAlgK}(Y,Y')
  &\wequiv\holim_{\Delta}\Hombold_\AlgJ\bigl(Y,(F\K)^\bullet FY'\bigr),
\end{align*}
respectively.
\end{rem}

The following two propositions appear in Ogle \cite{Ogle}; they play a key role in this paper. We defer their proofs to Section \ref{sec:proofs}, which develops a concise line of argument suggested by Dwyer \cite{Dwyer}.

\begin{prop}
\label{prop:tot_commutes_with_realization}
If $X\in(\sSet)^{\Delta_\res}$ is objectwise fibrant and $Y\in(\sSet)^\Delta$ is Reedy fibrant, then the natural maps
\begin{align*}
  |\Tot^\res X|&\xrightarrow{\wequiv}\Tot^\res|X|\\
  |\Tot Y|&\xrightarrow{\wequiv}\Tot|Y|
\end{align*}
in $\CGHaus$ are weak equivalences.
\end{prop}

\begin{prop}
\label{prop:holim_commutes_with_realization}
If $Y\in(\sSet)^\Delta$ is objectwise fibrant, then the natural map
\begin{align}
\label{eq:natural_map_realization_into_holim_nice}
  |\holim\nolimits^\BK_\Delta Y|\xrightarrow{\wequiv}
  \holim\nolimits^\BK_\Delta |Y|
\end{align}
in $\CGHaus$ is a weak equivalence.
\end{prop}

The following corollary plays a key role in this paper.

\begin{prop}
Let $Y,Y'$ be cofibrant $\K$-coalgebras. Then the natural map
\begin{align*}
  |\Hombold_{\coAlgK}(Y,Y')|\xrightarrow{\wequiv}
  \Map_{\coAlgK}(Y,Y')
\end{align*}
is a weak equivalence.
\end{prop}

\begin{proof}
This follows from Proposition \ref{prop:tot_commutes_with_realization}.
\end{proof}

The following provides a useful language for working with the spaces of derived $\K$-coalgebra maps; compare, Arone-Ching \cite[1.11]{Arone_Ching_classification}.

\begin{defn}
\label{defn:derived_K_coalgebra_map}
Let $Y,Y'$ be cofibrant $\K$-coalgebras. A \emph{derived $\K$-coalgebra map} $f$ of the form $Y\rarrow Y'$ is any map in $(\sSet)^{\Delta_\res}$ of the form
\begin{align*}
  \functionlong{f}{\Delta[-]}{\Hombold_\AlgJ\bigl(Y,(F\K)^\bullet FY'\bigr)}.
\end{align*}
A \emph{topological derived $\K$-coalgebra map} $g$ of the form $Y\rarrow Y'$ is any map in $(\CGHaus)^{\Delta_\res}$ of the form
\begin{align*}
  \functionlong{g}{\Delta^\bullet}{\Map_\AlgJ\bigl(Y,(F\K)^\bullet FY'\bigr)}.
\end{align*}
The \emph{underlying map} of a derived $\K$-coalgebra map $f$ is the map $\function{f_0}{Y}{FY'}$ that corresponds to the map $\function{f_0}{\Delta[0]}{\Hombold_\AlgJ(Y,FY')}$. Note that every derived $\K$-coalgebra map $f$ determines a topological derived $\K$-coalgebra map $|f|$ by realization.
\end{defn}

\begin{rem}
If $X,Y\in(\sSet)^\Delta$ their \emph{box product} $X\square Y\in(\sSet)^\Delta$ is the left Kan extension of objectwise product along ordinal sum (or concatenation)
\begin{align}
\label{eq:box_product_left_kan_extension}
\xymatrix{
  \Delta\times\Delta
  \ar[r]^-{X\times Y}\ar[d]^{\amalg} &
  \sSet\times\sSet\ar[r]^-{\times} & \sSet \\
  \Delta\ar[rr]^{X\square Y}_{\text{left Kan extension}} & & \sSet
}
\end{align}
and if $X,Y\in(\CGHaus)^\Delta$ their box product $X\square Y\in(\CGHaus)^\Delta$ is defined similarly by replacing $\sSet$ with $\CGHaus$ in \eqref{eq:box_product_left_kan_extension}. In particular, $(X\square Y)^0\Iso X^0\times Y^0$; see Arone-Ching \cite{Arone_Ching_classification} for an explicit calculation of $(X\square Y)^n$. A useful introduction and discussion of the box product is given in  McClure-Smith \cite{McClure_Smith_cosimplicial_objects}; it is a key ingredient in their proof of the Deligne Conjecture in \cite{McClure_Smith_solution}. The original source for this notion of a monoidal product on cosimplicial objects is Batanin \cite[3.2]{Batanin_coherent}; a dual version of the construction appears in Artin-Mazur \cite[III]{Artin_Mazur} for bisimplicial sets.
\end{rem}

\begin{prop}
\label{prop:composition_map}
Let $Y,Y',Y''$ be cofibrant $\K$-coalgebras. There is a natural map of the form
\begin{align*}
\xymatrix{
  \Hombold_\AlgJ\bigl(Y,(F\K)^\bullet FY'\bigr)\square
  \Hombold_\AlgJ\bigl(Y',(F\K)^\bullet FY''\bigr)\ar[d]^-{\mu}\\
  \Hombold_\AlgJ\bigl(Y,(F\K)^\bullet FY''\bigr)
}
\end{align*}
in $(\sSet)^\Delta$. We sometimes refer to $\mu$ as the \emph{composition} map.
\end{prop}

\begin{proof}
This is proved exactly as in \cite[1.6]{Arone_Ching_classification}; $\mu$ is the map induced by the collection of composites
\begin{align*}
\xymatrix{
  \Hombold_\AlgJ\bigl(Y,(F\K)^p FY'\bigr)\times
  \Hombold_\AlgJ\bigl(Y',(F\K)^q FY''\bigr)\ar[d]^-{\id\times(F\K)^p F}\\
  \Hombold_\AlgJ\bigl(Y,(F\K)^p FY'\bigr)\times
  \Hombold_\AlgJ\bigl((F\K)^p FY',(F\K)^p F(F\K)^q FY''\bigr)\ar[d]^-{\mathrm{comp}}\\
  \Hombold_\AlgJ\bigl(Y,(F\K)^p F(F\K)^q FY''\bigr)\ar[d]^-{\wequiv}\\
  \Hombold_\AlgJ\bigl(Y,(F\K)^{p+q} FY''\bigr)
}
\end{align*}
where $p,q\geq 0$; here, the indicated weak equivalence is the map induced by multiplication $FF\rarrow F$ of the simplicial fibrant replacement monad.
\end{proof}

\begin{prop}
\label{prop:realization_commutes_with_box_product}
Let $A,B\in(\sSet)^\Delta$. There is a natural isomorphism of the form
$
  |A\square B|\Iso|A|\square|B|
$
in $(\CGHaus)^\Delta$.
\end{prop}

\begin{proof}
This follows from the fact that realization commutes with finite products and all small colimits.
\end{proof}

\begin{prop}
\label{prop:composition_map_topological}
Let $Y,Y',Y''$ be cofibrant $\K$-coalgebras. There is a natural map of the form
\begin{align*}
\xymatrix{
  \Map_\AlgJ\bigl(Y,(F\K)^\bullet FY'\bigr)\square
  \Map_\AlgJ\bigl(Y',(F\K)^\bullet FY''\bigr)\ar[d]^-{\mu}\\
  \Map_\AlgJ\bigl(Y,(F\K)^\bullet FY''\bigr)
}
\end{align*}
in $(\CGHaus)^\Delta$. We sometimes refer to $\mu$ as the \emph{composition} map.
\end{prop}

\begin{proof}
This follows from Proposition \ref{prop:composition_map} by applying realization, together with Proposition \ref{prop:realization_commutes_with_box_product}.
\end{proof}

\begin{prop}
\label{prop:coaugmentation_by_a_point}
Let $Y$ be a cofibrant $\K$-coalgebra. There is a coaugmentation map $*\rarrow\Hombold_\AlgJ(Y,\K^\bullet Y)$ of the form (showing only the coface maps)
\begin{align*}
\xymatrix{
  {*}\ar[r]^-{d^0} &
  \Hombold_\AlgJ(Y,Y)\ar@<0.5ex>[r]^-{d^0}\ar@<-0.5ex>[r]_-{d^1} &
  \Hombold_\AlgJ(Y,\K Y)\cdots
}
\end{align*}
Here, the left-hand map picks out the identity map $Y\xrightarrow{\id}Y$ in simplicial degree $0$.
\end{prop}

\begin{proof}
It suffices to observe that $d^0d^0=d^1d^0$.
\end{proof}

\begin{defn}
Let $Y$ be a cofibrant $\K$-coalgebra. The \emph{unit map} $\iota$ (compare, \cite[1.6]{Arone_Ching_classification}) is the composite
\begin{align}
\label{eq:unit_map_for_operad_action}
  *\rightarrow\Map_\AlgJ(Y,\K^\bullet Y)\rightarrow\Map_\AlgJ(Y,(F\K)^\bullet FY)
\end{align}
in $(\CGHaus)^\Delta$. Here, the left-hand map is realization of the coaugmentation in Proposition \ref{prop:coaugmentation_by_a_point}, and the right-hand map is realization of the natural map \eqref{eq:mapping_space_resolutions}.
\end{defn}

\begin{defn}
The non-$\Sigma$ operad $\mathsf{A}$ in $\CGHaus$ is the coendomorphism operad of $\Delta^\bullet$ with respect to the box product $\square$ (\cite[1.12]{Arone_Ching_classification}) and is defined objectwise by the end construction
\begin{align*}
  \mathsf{A}(n):=
  \Map_{\Delta_\res}\bigl(\Delta^\bullet,(\Delta^\bullet)^{\square n}\bigr)
  :=\Map\bigl(\Delta^\bullet,(\Delta^\bullet)^{\square n}\bigr)^{\Delta_\res}
\end{align*}
In other words, $\mathsf{A}(n)$ is the space of restricted cosimplicial maps from $\Delta^\bullet$ to $(\Delta^\bullet)^{\square n}$; in particular, note that $\mathsf{A}(0)=*$.
\end{defn}

Consider the natural collection of maps (\cite[1.13]{Arone_Ching_classification})
\begin{align}
\label{eq:A_infinity_composition_maps}
  \mathsf{A}(n)\times
  \Map_\coAlgK(Y_0,Y_1)\times\cdots\times
  \Map_\coAlgK(Y_{n-1},Y_n)\\
  \notag
  \longrightarrow
  \Map_\coAlgK(Y_0,Y_n),\quad\quad
  n\geq 0,
\end{align}
induced by (iterations of) the composition map $\mu$ (Proposition \ref{prop:composition_map_topological}); in particular, in the case $n=0$, note that \eqref{eq:A_infinity_composition_maps} denotes the map
\begin{align*}
  *=\mathsf{A}(0)&\longrightarrow\Map_\coAlgK(Y_0,Y_0),\quad\quad n=0,
\end{align*}
that is $\Tot^\res$ applied to the unit map \eqref{eq:unit_map_for_operad_action}.

\begin{prop}
The collection of maps \eqref{eq:A_infinity_composition_maps} determine a topological $A_\infty$ category with objects the cofibrant $\K$-coalgebras and morphism spaces the mapping spaces $\Map_\coAlgK(Y,Y')$.
\end{prop}

\begin{proof}
This is proved exactly as in Arone-Ching \cite[1.14]{Arone_Ching_classification}.
\end{proof}

\begin{defn}
The \emph{homotopy category} of $\K$-coalgebras (compare, \cite[1.15]{Arone_Ching_classification}), denoted $\Ho(\coAlgK)$, is the category with objects the cofibrant $\K$-coalgebras and morphism sets $[X,Y]_\K$ from $X$ to $Y$ the path components
\begin{align*}
  [X,Y]_\K := \pi_0\Map_\coAlgK(X,Y)
\end{align*}
of the indicated mapping spaces.
\end{defn}

\begin{prop}
Let $X,Y$ be cofibrant $\K$-coalgebras. There is a natural map of morphism spaces of the form
\begin{align}
\label{eq:hom_space_comparison_map}
  \hom_\coAlgK(X,Y)\rarrow\Map_\coAlgK(X,Y)
\end{align}
\end{prop}

\begin{proof}
This is the composite of natural maps
\begin{align*}
  \hom_\coAlgK(X,Y)
  \Iso&\lim_\Delta\hom_\AlgJ(X,\K^\bullet Y)
  \rightarrow\lim_\Delta\Map_\AlgJ(X,\K^\bullet Y)\\
  \Iso&\lim_{\Delta_\res}\Map_\AlgJ(X,\K^\bullet Y)
  \rightarrow\Tot^\res\Map_\AlgJ(X,\K^\bullet Y)\\
  \rightarrow&\Tot^\res\Map_\AlgJ(X,(F\K)^\bullet FY)
  \Equal\Map_\coAlgK(X,Y)
\end{align*}
in $\CGHaus$.
\end{proof}

\begin{prop}
Let $X,Y$ be cofibrant $\K$-coalgebras. There is a natural map of morphism sets of the form
\begin{align}
\label{eq:hom_set_comparison_map}
  \hom_\coAlgK(X,Y)\rarrow [X,Y]_\K
\end{align}
\end{prop}

\begin{proof}
This follows by applying the path components functor $\pi_0$ to \eqref{eq:hom_space_comparison_map}.
\end{proof}

\begin{prop}
There is a well-defined functor
\begin{align*}
  \function{\gamma}{\coAlg_K|_\mathrm{cof}}{\Ho(\coAlgK)}
\end{align*}
that is the identity on objects and is the map \eqref{eq:hom_set_comparison_map} on morphisms; here, $\coAlg_K|_\mathrm{cof}\subset\coAlgK$ denotes the full subcategory of cofibrant $\K$-coalgebras.
\end{prop}

\begin{proof}
This is proved exactly as in \cite[1.14]{Arone_Ching_classification}.
\end{proof}

\begin{defn}
\label{defn:weak_equivalence_of_K_coalgebras}
A derived $\K$-coalgebra map $f$ of the form $X\rarrow Y$ is a \emph{weak equivalence} if the underlying map $\function{f_0}{X}{FY}$ is a weak equivalence.
\end{defn}

The following verifies that this definition is homotopically meaningful.

\begin{prop}
\label{prop:weak_equivalence_if_and_only_if_iso_in_homotopy_category}
Let $X,Y$ be cofibrant $\K$-coalgebras. A derived $\K$-coalgebra map $f$ of the form $X\rarrow Y$ is a weak equivalence if and only if the induced map $\gamma(f)$ in $[X,Y]_\K$ is an isomorphism in the homotopy category of $\K$-coalgebras.
\end{prop}

\begin{proof}
The is proved exactly as in Arone-Ching \cite[1.16]{Arone_Ching_classification}.
\end{proof}

\subsection{Derived counit map}

The basic idea behind Definition \ref{defn:derived_counit_map} (see also Remark \ref{rem:homotopical_analog_of_the_usual_counit_map}) is to look for a naturally occurring derived $\K$-coalgebra map of the form $\LL Q\holim_\Delta C(Y)\rightarrow Y$. The construction follows immediately from the observation that there are natural zigzags of weak equivalences
\begin{align*}
  \holim_\Delta C(Y)\wequiv
  \holim_\Delta\mathfrak{C}(Y)
  \wequiv\Tot^\res\mathfrak{C}(Y)
\end{align*}
of $\capO$-algebras.

\begin{defn}
Denote by $c$ the simplicial cofibrant replacement functor on $\AlgO$ (Farjoun \cite[I.C.11]{Farjoun_LNM}, Rezk-Schwede-Shipley \cite[6.3]{Rezk_Schwede_Shipley}).
\end{defn}

\begin{defn}
\label{defn:derived_counit_map}
The \emph{derived counit map} associated to the fundamental adjunction \eqref{eq:fundamental_adjunction_comparing_AlgO_with_coAlgK} is the derived $\K$-coalgebra map of the form
\begin{align}
\label{eq:derived_counit_map_of_the_form}
  Qc\Tot^\res\mathfrak{C}(Y)\longrightarrow Y
\end{align}
corresponding to the composite
\begin{align}
\label{eq:derived_identity_map}
  c\Tot^\res\mathfrak{C}(Y)\xrightarrow{\wequiv}
  \id\Tot^\res\mathfrak{C}(Y)\xrightarrow{\id}
  \Tot^\res\mathfrak{C}(Y)
\end{align}
in $\AlgO$, via the adjunctions \eqref{eqref:tot_adjunctions} and \eqref{eq:tensordot_adjunction_isomorphisms}; this composite can be thought of as a ``fattened up'' version of the identity map on $\Tot^\res\mathfrak{C}(Y)$ (see Remark \ref{rem:homotopical_analog_of_the_usual_counit_map}). Here, the indicated weak equivalence is the augmentation of the simplicial cofibrant replacement functor $c$. More precisely, the derived counit map is the derived $\K$-coalgebra map defined by the composite
\begin{align}
\label{eq:derived_counit_map}
  \Delta[-]\xrightarrow{(*)}
  &\Hombold_\AlgO\bigl(c\Tot^\res\mathfrak{C}(Y),\mathfrak{C}(Y)\bigr)\\
  \notag
  \Iso
  &\Hombold_\AlgJ\bigl(Qc\Tot^\res\mathfrak{C}(Y),(F\K)^\bullet FY\bigr)
\end{align}
in $(\sSet)^{\Delta_\res}$, where $(*)$ corresponds to the composite \eqref{eq:derived_identity_map} in $\AlgO$, via the adjunctions \eqref{eqref:tot_adjunctions} and \eqref{eq:tensordot_adjunction_isomorphisms}. In other words, \eqref{eq:derived_counit_map_of_the_form} is a derived $\K$-coalgebra map of the form
\begin{align*}
  \LL Q\bigl(\holim\nolimits_\Delta C(Y)\bigr)\longrightarrow Y
\end{align*}
that represents the derived counit map associated to the adjunction \eqref{eq:fundamental_adjunction_comparing_AlgO_with_coAlgK}.
\end{defn}

\begin{rem}
\label{rem:homotopical_analog_of_the_usual_counit_map}
Note that the derived counit map in Definition \ref{defn:derived_counit_map} can be thought of as a homotopical analog of the counit map $Q\lim_\Delta C(Y)\rarrow Y$ in $\coAlgK$, that corresponds to the identity map $\function{\id}{\lim_\Delta C(Y)}{\lim_\Delta C(Y)}$ in $\AlgO$ by \eqref{eq:fundamental_adjunction_comparing_AlgO_with_coAlgK}; compare with \eqref{eq:derived_identity_map}.
\end{rem}

\subsection{Mapping spaces of $\capO$-algebras and $\K$-coalgebras}

\begin{prop}
\label{prop:induced_map_on_mapping_spaces_built_from_Q}
Let $X,Y$ be cofibrant $\capO$-algebras.
There are natural morphisms of mapping spaces of the form
\begin{align*}
  Q\colon\thinspace\Hombold_\AlgO(X,Y)\rarrow&\Hombold_\coAlgK(QX,QY),\\
  Q\colon\thinspace\Map_\AlgO(X,Y)\rarrow&\Map_\coAlgK(QX,QY),
\end{align*}
in $\sSet$ and $\CGHaus$, respectively.
\end{prop}

\begin{proof}
In the first case, this is the composite
\begin{align*}
  \Hombold_\AlgO(X,Y)\Iso
  \lim_{\Delta_\res}\Hombold_\AlgO(X,Y)\rightarrow
  \Tot^\res\Hombold_\AlgO(X,Y)\\
  \xrightarrow{(*)}
  \Tot^\res\Hombold_\AlgO\bigl(X,\mathfrak{C}(QY)\bigr)\Iso
  \Hombold_\coAlgK(QX,QY)
\end{align*}
where $(*)$ is induced by the natural coaugmentation map $Y\rarrow\mathfrak{C}(QY)$ in $(\AlgO)^\Delta$. The second case is the composite
\begin{align*}
  \Map_\AlgO(X,Y)\xrightarrow{|Q|}
  |\Hombold_\coAlgK(QX,QY)|\xrightarrow{\wequiv}\Map_\coAlgK(QX,QY)
\end{align*}
where we have used Proposition \ref{prop:tot_commutes_with_realization} for the right-hand weak equivalence. Here, $\Hombold_\AlgO(X,Y)$ is regarded as a constant cosimplicial object in $\sSet$.
\end{proof}

\begin{prop}
There is an induced functor
\begin{align*}
  \function{\LL Q}{\Ho(\AlgO|_\mathsf{cof})}{\Ho(\coAlgK)}
\end{align*}
which on objects is the map $X\mapsto Q(X)$ and on morphisms is the map
\begin{align*}
  \function{\pi_0Q}{\pi_0\Map_\AlgO(X,Y)}{\pi_0\Map_\coAlgK(QX,QY)}
\end{align*}
which sends $[f]$ to $\gamma(Qf)=[Qf]$.
\end{prop}

\begin{proof}
This follows from Proposition \ref{prop:induced_map_on_mapping_spaces_built_from_Q}.
\end{proof}

The following three propositions are fundamental to this paper; they verify that the cosimplicial resolutions of $\K$-coalgebra mapping spaces respect the adjunction isomorphisms associated to the $(Q,U)$ adjunction (Proposition \ref{prop:cosimplicial_resolutions_of_K_coalgebras_respect_adjunction_isos}).

\begin{prop}
\label{prop:adjunction_isos_respect_cosimplicial_relations}
Let $X\in\AlgO$ and $Y\in\coAlgK$. The adjunction isomorphisms associated to the $(Q,U)$ adjunction induce well-defined isomorphisms
\begin{align*}
  \hom_{\AlgJ}(QX,\K^\bullet Y)&\xrightarrow{\Iso}
  \hom_{\AlgO}(X,U\K^\bullet Y)\\
  \hom_{\AlgJ}\bigl(QX,(F\K)^\bullet FY\bigr)&\xrightarrow{\Iso}
  \hom_{\AlgO}\bigl(X,U(F\K)^\bullet FY\bigr)
\end{align*}
of cosimplicial objects in $\Set$, natural in $X,Y$.
\end{prop}

\begin{proof}
This is an exercise left to the reader.
\end{proof}

\begin{prop}
\label{prop:sigma_is_a_map_of_K_coalgebras_general_statement}
If $Y\in\coAlgK$ with comultiplication map $\function{m}{Y}{\K Y}$ and $L\in\sSet$, then $Y\tensordot L$ in $\AlgJ$ has a natural $\K$-coalgebra structure with  comultiplication map $\function{m}{Y\tensordot L}{\K(Y\tensordot L)}$ given by the composite
\begin{align*}
\xymatrix{
  Y\tensordot L\ar[r]^-{m\tensordot\id} &
  \K(Y)\tensordot L\ar[r]^-{\sigma} &
  \K(Y\tensordot L)
}
\end{align*}
\end{prop}

\begin{proof}
This is an exercise left to the reader.
\end{proof}

\begin{prop}
\label{prop:natural_isomorphism_sigma_respects_cosimplicial_relations}
Let $X\in\AlgO$ and $Y\in\coAlgK$. The natural isomorphism $\function{\sigma}{Q(X)\tensordot L}{Q(X\tensordot L)}$ induces well-defined isomorphisms
\begin{align*}
  \hom_{\AlgJ}\bigl(Q(X\tensordot L),\K^\bullet Y\bigr)&
  \xrightarrow[\Iso]{(\sigma,\id)}
  \hom_{\AlgO}\bigl(Q(X)\tensordot L,\K^\bullet Y\bigr)\\
  \hom_{\AlgJ}\bigl(Q(X\tensordot L),(F\K)^\bullet FY\bigr)&
  \xrightarrow[\Iso]{(\sigma,\id)}
  \hom_{\AlgO}\bigl(Q(X)\tensordot L,(F\K)^\bullet FY\bigr)
\end{align*}
of cosimplicial objects in $\Set$, natural in $X,Y$.
\end{prop}

\begin{proof}
This is an exercise left to the reader.
\end{proof}

\begin{prop}
\label{prop:cosimplicial_resolutions_of_K_coalgebras_respect_adjunction_isos}
Let $X\in\AlgO$ and $Y\in\coAlgK$. The adjunction isomorphisms associated to the $(Q,U)$ adjunction induce well-defined isomorphisms
\begin{align*}
  \Hombold_{\AlgJ}(QX,\K^\bullet Y)&\xrightarrow{\Iso}
  \Hombold_{\AlgO}(X,U\K^\bullet Y)\\
  \Hombold_{\AlgJ}\bigl(QX,(F\K)^\bullet FY\bigr)&\xrightarrow{\Iso}
  \Hombold_{\AlgO}\bigl(X,U(F\K)^\bullet FY\bigr)
\end{align*}
of cosimplicial objects in $\sSet$, natural in $X,Y$.
\end{prop}

\begin{proof}
Consider the first case. It suffices to verify that the composite
\begin{align*}
  \hom(Q(X)\tensordot\Delta[n],\K^\bullet Y)\Iso
  \hom(Q(X\tensordot\Delta[n]),\K^\bullet Y)\Iso
  \hom(X\tensordot\Delta[n],U\K^\bullet Y)
\end{align*}
is a well-defined map of cosimplicial objects in $\Set$, natural in $X,Y$, for each $n\geq 0$; this follows from Propositions \ref{prop:adjunction_isos_respect_cosimplicial_relations} and \ref{prop:natural_isomorphism_sigma_respects_cosimplicial_relations}. The case involving the simplicial fibrant replacement monad $F$ is similar.
\end{proof}

\begin{prop}
\label{prop:zigzag_of_weak_equivalences_tot_and_tq_completion}
If $X$ is a cofibrant $\capO$-algebra, then there is a zigzag of weak equivalences of the form
\begin{align*}
  X^\wedge_\TQ\wequiv\Tot^\res\mathfrak{C}(QX)
\end{align*}
in $\AlgO$, natural with respect to all such $X$.
\end{prop}

\begin{proof}
This follows from the zigzags of weak equivalences
\begin{align*}
  X^\wedge_\TQ\wequiv
  \holim\nolimits_\Delta C(QX)
  \wequiv\holim\nolimits_{\Delta}\mathfrak{C}(QX)
  \wequiv\Tot^\res\mathfrak{C}(QX)
\end{align*}
in $\AlgO$, natural with respect to all such $X$; here we used the fact that $\mathfrak{C}(QX)$ is objectwise fibrant.
\end{proof}

\begin{defn}
An $\capO$-algebra $Y$ is called \emph{$\TQ$-complete} if the natural coaugmentation $Y\wequiv Y^\wedge_\TQ$ is a weak equivalence.
\end{defn}

The following amounts to the observation that mapping into fibrant $\TQ$-complete objects induces the indicated weak equivalence on mapping spaces; this was pointed out in the homotopic descent work of Hess \cite[5.5]{Hess}, and subsequently in Arone-Ching \cite[2.15]{Arone_Ching_classification}. In particular, the $\TQ$-homology spectrum functor is homotopically fully faithful, after restriction to a full subcategory of objects, provided that each object in the subcategory is $\TQ$-complete.

\begin{prop}
\label{prop:formal_adjunction_and_iso_argument}
Let $X,Y$ be cofibrant $\capO$-algebras. If $Y$ is $\TQ$-complete and fibrant, then there is a natural zigzag
\begin{align*}
  Q\colon\thinspace\Map_\AlgO(X,Y)\xrightarrow{\wequiv}\Map_\coAlgK(QX,QY)
\end{align*}
of weak equivalences; applying $\pi_0$ gives the map $[f]\mapsto\gamma(Qf)=[Qf]$.
\end{prop}

\begin{proof}
This is proved exactly as in \cite[2.15]{Arone_Ching_classification}; $Y\wequiv Y^\wedge_\TQ$ by assumption, and we know that $Y^\wedge_\TQ\wequiv\Tot^\res\mathfrak{C}(QY)$ by Proposition \ref{prop:zigzag_of_weak_equivalences_tot_and_tq_completion}, hence there are zigzags of weak equivalences of the form
\begin{align*}
  \Hombold_\AlgO(X,Y)&\wequiv
  \Hombold_\AlgO(X,Y^\wedge_\TQ)\\
  &\wequiv\Hombold_\AlgO\bigl(X,\Tot^\res\mathfrak{C}(QY)\bigr)\\
  &\Iso\Tot^\res\Hombold_\AlgO\bigl(X,U(F\K)^\bullet FQY\bigr)\\
  &\Iso\Hombold_\coAlgK(QX,QY)
\end{align*}
in $\sSet$; applying realization, together with Proposition \ref{prop:tot_commutes_with_realization} finishes the proof.
\end{proof}

The following two propositions will be needed in the proof of our main result.

\begin{prop}
\label{prop:natural_tot_map_induced_by_simplicial_functor}
Let $\capO,\capO'$ be operads in $\capR$-modules. Let $\function{F}{\AlgO}{\Alg_{\capO'}}$ be a simplicial functor and $X$ a cosimplicial (resp. restricted cosimplicial) $\capO$-algebra. There is a natural map of the form
\begin{align*}
  F\Tot(X)\rarrow\Tot(FX),\quad\quad
  \Bigl(
  \text{resp.}\quad
  F\Tot^\res(X)\rarrow\Tot^\res(FX),
  \Bigr)
\end{align*}
in $\Alg_{\capO'}$ induced by the simplicial structure maps of $F$.
\end{prop}

\begin{proof}
In both cases, the indicated map is induced by the composite maps
\begin{align*}
  F\bigl(\hombold(\Delta[n],X^n)\bigr)\tensordot\Delta[n]
  \xrightarrow{\sigma}
  F\bigl(\hombold(\Delta[n],X^n)\tensordot\Delta[n]\bigr)
  \xrightarrow{\id(\ev)}
  F(X^n),\quad\quad n\geq 0,
\end{align*}
via the natural isomorphisms in \eqref{eq:tensordot_adjunction_isomorphisms}.
\end{proof}

\begin{prop}
\label{prop:simplicial_natural_transformations_play_nicely_with_Tot}
Let $\capO,\capO'$ be operads in $\capR$-modules. Let $\function{F,G}{\AlgO}{\Alg_{\capO'}}$ be simplicial functors, $\function{\tau}{F}{G}$ a simplicial natural transformation, $X$ a cosimplicial $\capO$-algebra, and $Y$ a restricted cosimplicial $\capO$-algebra. Then the following diagrams
\begin{align*}
\xymatrix{
  F\Tot(X)\ar[d]\ar[r] & G\Tot(X)\ar[d]\\
  \Tot(FX)\ar[r] & \Tot(GX)
}\quad\quad
\xymatrix{
  F\Tot^\res(Y)\ar[d]\ar[r] & G\Tot^\res(Y)\ar[d]\\
  \Tot^\res(FY)\ar[r] & \Tot^\res(GY)
}
\end{align*}
commute.
\end{prop}

\begin{proof}
This is an exercise left to the reader.
\end{proof}

\subsection{Another description of the derived counit map}

Consider the collection of maps
\begin{align}
\label{eq:objectwise_composition_of_canonical_map}
  \Delta[n]
  &\xrightarrow{\quad\ }
  \Hombold_\AlgJ\bigl(Qc\Tot^\res \mathfrak{C}(Y),(F\K)^n FY\bigr),
  \quad\quad
  n\geq 0,
\end{align}
in $\sSet$ described in \eqref{eq:derived_counit_map} associated to the derived counit map.

\begin{prop}
\label{prop:canonical_maps_for_derived_counit}
The maps in \eqref{eq:objectwise_composition_of_canonical_map} correspond with the maps
\begin{align}
\label{eq:canonical_maps_for_derived_counit}
  \bigl(Qc\Tot^\res \mathfrak{C}(Y)\bigr)\tensordot\Delta[n]
  \longrightarrow (F\K)^n FY,
  \quad\quad
  n\geq 0,
\end{align}
in $\AlgJ$, defined by the composite
\begin{align*}
  \bigl(Qc\Tot^\res \mathfrak{C}(Y)\bigr)\tensordot\Delta[n]
  \rightarrow\bigl(Q\id\Tot^\res \mathfrak{C}(Y)\bigr)\tensordot\Delta[n]
  \rightarrow
  \bigl(\Tot^\res Q \mathfrak{C}(Y)\bigr)\tensordot\Delta[n]\\
  \xrightarrow{\mathrm{proj}}
  \K(F\K)^n FY\xrightarrow{\varepsilon(\id)^n\id^2}\id(F\K)^n FY
\end{align*}
where $\mathrm{proj}$ denotes the indicated projection map.
\end{prop}

\begin{proof}
This follows from the adjunction isomorphisms.
\end{proof}

\begin{prop}
\label{prop:another_description_of_the_canonical_derived_counit_map}
For each $n\geq 0$, the map \eqref{eq:canonical_maps_for_derived_counit} is equal to the composite
\begin{align*}
  \bigl(Qc\Tot^\res \mathfrak{C}(Y)\bigr)\tensordot\Delta[n]
  \xrightarrow{\Equal}
  &\bigl(\id Qc\Tot^\res \mathfrak{C}(Y)\bigr)\tensordot\Delta[n]\\
  \xrightarrow{\wequiv}
  &\bigl(FQc\Tot^\res \mathfrak{C}(Y)\bigr)\tensordot\Delta[n]\\
  \xrightarrow{\quad\ }
  &\bigl(\Tot^\res FQc \mathfrak{C}(Y)\bigr)\tensordot\Delta[n]\\
  \xrightarrow{\wequiv}
  &\bigl(\Tot^\res FQ\id \mathfrak{C}(Y)\bigr)\tensordot\Delta[n]\\
  \xrightarrow{\mathrm{proj}}
  &F\K(F\K)^n FY\xrightarrow{s^{-1}}(F\K)^n FY
\end{align*}
where $\mathrm{proj}$ denotes the indicated projection map; here, it may be helpful to note that $FQ\mathfrak{C}(Y)=\Cobar(F\K,F\K,Y)$.
\end{prop}

\begin{proof}
This follows from Proposition \ref{prop:simplicial_natural_transformations_play_nicely_with_Tot}.
\end{proof}

\begin{rem}
\label{rem:elaboration_on_the_proof_of_main_theorem}
For the convenience of the reader, we will further elaborate on the connection between the maps appearing in the proof of Theorem \ref{MainTheorem}(a) and the derived counit map. Consider the following commutative diagram
\begin{align*}
\xymatrix{
  \LL Q\holim\nolimits_\Delta C(Y)\ar[d]^{\wequiv}\ar[r]^-{(*)'} &
  \holim\nolimits_\Delta \LL Q\, C(Y)\ar[d]^-{\wequiv}\ar[r]^-{(*)''} &
  Y\ar@/^1pc/[ddd]^(0.3){\wequiv}_-(0.3){(**)} \\
  FQc\holim\nolimits^\BK_\Delta\mathfrak{C}(Y)\ar[d]_-{(*)}^-{\wequiv}
  \ar[r] &
  \holim\nolimits^\BK_\Delta FQ\mathfrak{C}(Y)\ar[d]_-{(*)}^-{\wequiv}
  \ar@/^2pc/[ddr]\\
  FQc\holim\nolimits^\BK_{\Delta_\res}\mathfrak{C}(Y)\ar[r] &
  \holim\nolimits^\BK_{\Delta_\res} FQ\mathfrak{C}(Y)\ar[dr]\\
  FQc\Tot^\res\mathfrak{C}(Y)\ar[r]\ar[u]^-{(*)}_-{\wequiv} &
  \Tot^\res FQ\mathfrak{C}(Y)\ar[r]\ar[u]^-{(*)}_-{\wequiv} & FY\\
  Qc\Tot^\res\mathfrak{C}(Y)\ar[u]^-{(**)}_-{\wequiv}\ar@/_1pc/[urr]_-{(\#)}
}
\end{align*}
where the arrow $(\#)$ is defined by the indicated composition; here, the top horizontal maps are the maps in \eqref{eq:FQC_commutes_with_desired_holim}. By Proposition \ref{prop:another_description_of_the_canonical_derived_counit_map}, the map underlying the derived counit map \eqref{eq:derived_counit_map_of_the_form} is precisely the map $(\#)$. The maps $(*)$ are weak equivalences by Proposition \ref{prop:comparing_holim_with_Tot_and_Tot_restricted}, the maps $(**)$ are weak equivalences since each is the unit of the simplicial fibrant replacement monad $F$. Hence to verify that the map $(\#)$ underlying the derived counit map is a weak equivalence, it suffices to verify that $(*)'$ and $(*)''$ are weak equivalences; this is verified in the proof of Theorem \ref{MainTheorem}(a). This reduction argument can be thought of as a homotopical Barr-Beck comonadicity theorem; see Arone-Ching \cite[2.20]{Arone_Ching_classification}.
\end{rem}

\section{Homotopical analysis of the cubical diagrams}
\label{sec:cubical_diagrams_homotopical_analysis}

The purpose of this section is to prove Propositions \ref{prop:homotopy_fiber_calculation}, \ref{prop:connectivities_for_total_homotopy_fiber_homotopy_completion_cube_no_zigzags}, \ref{prop:iterated_homotopy_fibers_calculation}, \ref{prop:multisimplicial_calculation_of_iterated_hofiber_codegeneracy_cube_no_zigzags}, and Theorem \ref{thm:connectivities_for_map_that_commutes_Q_into_inside_of_holim} that were needed in the proof of our main result (Section \ref{sec:outline_of_the_argument}). Aspects of our approach are in the same spirit as the work of Dundas \cite{Dundas_relative_K_theory}, where Goodwillie's higher Blakers-Massey theorems \cite{Goodwillie_calculus_2} for spaces are exploited to great effect.

Munson-Volic \cite{Munson_Volic_book_project} have pointed out that the following observation provides a useful building block for proving, with minimal effort, several useful propositions below.

\begin{prop}
\label{prop:retraction_two_cube_argument}
Consider any $2$-cube $\capX$ of the form
\begin{align*}
\xymatrix{
  X\ar@{=}[d]\ar[r]^-{f} & Y\ar[d]^-{g}\\
  X\ar@{=}[r] & X
}
\end{align*}
in $\AlgO$; in other words, suppose $g$ is a retraction of $f$. There are natural weak equivalences $\hofib(f)\wequiv\Omega\hofib(g)$; here $\Omega$ is weakly equivalent, in the underlying category $\ModR$, to the desuspension $\Sigma^{-1}$ functor.
\end{prop}

\begin{proof}
This is proved, for instance, in Munson-Volic \cite{Munson_Volic_book_project} in the context of spaces, and the same argument verifies it in our context. It follows easily by calculating the iterated homotopy fibers of $\capX$ in two different ways, by starting in the horizontal direction versus the vertical direction.
\end{proof}

\begin{prop}[Proposition \ref{prop:homotopy_fiber_calculation} restated]
\label{prop:homotopy_fiber_calculation_other_section}
Let $Z$ be a cosimplicial $\capO$-algebra coaugmented by $\function{d^0}{Z^{-1}}{Z^0}$. If $n\geq 0$, then there are natural zigzags of weak equivalences
\begin{align*}
  \hofib(Z^{-1}\rarrow\holim\nolimits_{\Delta^{\leq n}}Z)
  \wequiv
  (\iter\hofib)\capX_{n+1}
\end{align*}
where $\capX_{n+1}$ denotes the canonical $(n+1)$-cube associated to the coface maps of
\begin{align*}
\xymatrix{
  Z^{-1}\ar[r]^-{d^0} &
  Z^0\ar@<-0.5ex>[r]_-{d^1}\ar@<0.5ex>[r]^-{d^0} &
  Z^1\ \cdots\ Z^n
}
\end{align*}
the $n$-truncation of $Z^{-1}\rarrow Z$. We sometimes refer to $\capX_{n+1}$ as the \emph{coface} $(n+1)$-cube associated to the coaugmented cosimplicial $\capO$-algebra $Z^{-1}\rarrow Z$.
\end{prop}

\begin{proof}
This is proved in Carlsson \cite{Carlsson} and Sinha \cite{Sinha_cosimplicial_models} in the contexts of spectra and spaces, respectively, and the same argument verifies it in our context. It follows easily from the cosimplicial identities, together with repeated application of Proposition \ref{prop:retraction_two_cube_argument} by using the fact that codegeneracy maps provide retractions for the appropriate coface maps.
\end{proof}

The following homotopy spectral sequence for a simplicial symmetric spectrum is well known; for a recent reference, see \cite[X.2.9]{EKMM} and \cite[4.3]{Jardine_generalized_etale}.

\begin{prop}
\label{prop:homotopy_spectral_sequence}
Let $Y$ be a simplicial symmetric spectrum. There is a natural homologically graded spectral sequence in the right-half plane such that
\begin{align*}
  E^2_{p,q} = H_p(\pi_q(Y))\Longrightarrow\pi_{p+q}(|Y|)
\end{align*}
Here, $\pi_q(Y)$ denotes the simplicial abelian group obtained by applying $\pi_q$ levelwise to $Y$.
\end{prop}

The following calculations, which use the notation in Harper-Hess \cite[Section 2]{Harper_Hess}, encode all of the combinatorics needed for calculating explicitly the iterated homotopy fibers of the coface and codegeneracy $n$-cubes discussed below. In particular, these combinatorics allow us to efficiently calculate connectivities by using exactly the same line of arguments that were used in Harper-Hess \cite[Proof of 1.8]{Harper_Hess} for calculating the homotopy fibers of certain $1$-cubes, but now iterated up to calculate the total homotopy fibers of certain $n$-cubes.

\begin{prop}
\label{prop:calculations_for_iterated_fibers_part1}
Consider symmetric sequences in $\ModR$. Let $Y\in\SymSeq$ such that $Y[0]=*$. If $n\geq 1$, then the diagram
\begin{align*}
\xymatrix{
  (Y^{\tensorcheck n})^{>n}\ar[r]^-{\subset}\ar[d] & Y^{\tensorcheck n}\ar[d]^{(*)}\\
  {*}\ar[r] & (\tau_1 Y)^{\tensorcheck n}
}
\end{align*}
is a pushout diagram, and a pullback diagram, in $\SymSeq$. In particular, the fiber of the map $(*)$ is a symmetric sequence that starts at level $n+1$.
\end{prop}

\begin{proof}
This is an exercise left to the reader.
\end{proof}

\begin{prop}
\label{prop:calculations_for_iterated_fibers_part2}
Consider symmetric sequences in $\ModR$. Let $W,Z\in\SymSeq$ such that $W[0]=*$ and $Z[0]=*$. If $m\geq 0$, then
$W^{>m}\circ Z = (W^{>m}\circ Z)^{>m}$. In other words, if $W$ starts at level $m+1$, then $W\circ Z$ starts at level $m+1$.
\end{prop}

\begin{proof}
This is an exercise left to the reader.
\end{proof}

\begin{prop}
\label{prop:calculations_for_iterated_fibers_part3}
Consider symmetric sequences in $\ModR$. Let $W,Y\in\SymSeq$ such that $W[0]=*$ and $Y[0]=*$. If $m\geq 0$, then the left-hand diagram
\begin{align*}
\xymatrix{
  \widetilde{W}=\widetilde{W}^{>m+1}\ar[d]\ar[r]^-{\subset} &
  W^{>m}\circ Y\ar[d]^{(*)}
  &
  \widetilde{W} := \coprod\limits_{t>m}W[t]
  \Smash_{\Sigma_t}(Y^{\tensorcheck t})^{>t}\\
  {*}\ar[r] & W^{>m}\circ\tau_1 Y
}
\end{align*}
is a pushout diagram, and a pullback diagram, in $\SymSeq$. In particular, the fiber of the map $(*)$ is a symmetric sequence that starts at level $m+2$.
\end{prop}

\begin{proof}
This is an exercise left to the reader.
\end{proof}

The following notation will be useful below.

\begin{defn} Let $n\geq 1$ and denote by $\Cube_n$ the category with objects the vertices $(v_1,\dotsc,v_n)\in\{0,1\}^n$ of the unit $n$-cube. There is at most one morphism between any two objects, and there is a morphism
$
  (v_1,\dotsc,v_n)\rarrow  (v_1',\dotsc,v_n')
$
if and only if $v_i\leq v_i'$ for each $1\leq i\leq n$. In particular, $\Cube_n$ is the category associated to a partial order on the set $\{0,1\}^n$.
\end{defn}

For each $n\geq 1$, denote by $\capX_n$ the coface $n$-cube associated to the coaugmentation $X\rarrow C(QX)$. The reason for introducing the $n$-cubes $\capX'_n$, $\capX''_n$, $\capX'''_n$ in Definiton \ref{defn:sequence_of_weakly_equivalent_models_of_the_coface_cube} below is that we can homotopically analyze the iterated homotopy fibers of $\capX'''_n$ using the combinatorics developed above; since these intermediate $n$-cubes fit into a zigzag of weak equivalences $\capX_n\wequiv\capX_n'''$ of $n$-cubes, the combinatorial analysis results in a homotopical analysis of $\capX_n$ (Proposition \ref{prop:connectivities_for_total_homotopy_fiber_homotopy_completion_cube}).

\begin{defn}
\label{defn:sequence_of_weakly_equivalent_models_of_the_coface_cube}
If $X\in\AlgO$ is cofibrant and $n\geq 1$, denote by $\capX_n$ the coface $n$-cube associated to $X\rarrow C(QX)$, and by $\capX'_n$, $\capX''_n$, $\capX'''_n$ the following weakly equivalent $n$-cubes. The functor $\function{\capX'_n}{\Cube_{n}}{\AlgO}$ is defined by
\begin{align*}
  (v_1,v_2,\dotsc,v_n)&\longmapsto A_1\circ_\capO A_{2}\circ_\capO\cdots A_n\circ_\capO(X)\\
  \text{where}
  \quad\quad
  A_i &:=
  \left\{
    \begin{array}{ll}
    \capO,&\text{for $v_i=0$,}\\
    J,&\text{for $v_i=1$,}
  \end{array}
  \right.
\end{align*}
with maps induced by $\capO\rarrow J$; the functor $\function{\capX''_n}{\Cube_{n}}{\AlgO}$ is defined by
\begin{align*}
  (v_1,v_2,\dotsc,v_n)&\longmapsto|\BAR(A_1,\capO,|\BAR(A_{2},\capO,\dots,|\BAR(A_n,\capO,X)|\dotsb)|)|\\
  \text{where}
  \quad\quad
  A_i &:=
  \left\{
    \begin{array}{ll}
    \capO,&\text{for $v_i=0$,}\\
    J,&\text{for $v_i=1$,}
  \end{array}
  \right.
\end{align*}
with maps induced by $\capO\rarrow J$; the functor $\function{\capX'''_n}{\Cube_{n}}{\AlgO}$ is defined by
\begin{align*}
  (v_1,v_2,\dotsc,v_n)&\longmapsto|\BAR(A_1,\capO,|\BAR(A_{2},\capO,\dots,|\BAR(A_n,\capO,X)|\dotsb)|)|\\
  \text{where}
  \quad\quad
  A_i &:=
  \left\{
    \begin{array}{ll}
    \capO,&\text{for $v_i=0$,}\\
    \tau_1\capO,&\text{for $v_i=1$,}
  \end{array}
  \right.
\end{align*}
with maps induced by $\capO\rarrow\tau_1\capO$.
\end{defn}

\begin{rem}
For instance, $\capX_2$ is the $2$-cube associated to $X\rarrow C(QX)$ of the form
\begin{align*}
\xymatrix{
  X\ar[d]^-{d^0}\ar[r]^-{d^0} &
  J\circ_\capO(X)\ar[d]^-{d^0}\\
  J\circ_\capO(X)\ar[r]^-{d^1} &
  J\circ_\capO J\circ_\capO(X)
}
\end{align*}
$\capX'_2$ is the associated $2$-cube of the form
\begin{align*}
\xymatrix{
  \capO\circ_\capO\capO\circ_\capO(X)\ar[d]^-{d^0}\ar[r]^-{d^0} &
  J\circ_\capO\capO\circ_\capO(X)\ar[d]^-{d^0}\\
  \capO\circ_\capO J\circ_\capO(X)\ar[r]^-{d^1} &
  J\circ_\capO J\circ_\capO(X)
}
\end{align*}
$\capX''_2$ is the associated $2$-cube of the form
\begin{align*}
\xymatrix{
  |\BAR(\capO,\capO,|\BAR(\capO,\capO,X)|)|\ar[d]^-{d^0}\ar[r]^-{d^0} &
  |\BAR(J,\capO,|\BAR(\capO,\capO,X)|)|\ar[d]^-{d^0}\\
  |\BAR(\capO,\capO,|\BAR(J,\capO,X)|)|\ar[r]^-{d^1} &
  |\BAR(J,\capO,|\BAR(J,\capO,X)|)|
}
\end{align*}
and $\capX'''_2$ is the associated $2$-cube of the form
\begin{align}
\label{eq:diagram_n_equals_2_case_capX_2_cube}
\xymatrix{
  |\BAR(\capO,\capO,|\BAR(\capO,\capO,X)|)|\ar[d]^-{d^0}\ar[r]^-{d^0} &
  |\BAR(\tau_1\capO,\capO,|\BAR(\capO,\capO,X)|)|\ar[d]^-{d^0}\\
  |\BAR(\capO,\capO,|\BAR(\tau_1\capO,\capO,X)|)|\ar[r]^-{d^1} &
  |\BAR(\tau_1\capO,\capO,|\BAR(\tau_1\capO,\capO,X)|)|
}
\end{align}
\end{rem}

\begin{prop}
\label{prop:connectivities_for_total_homotopy_fiber_homotopy_completion_cube}
If $X$ is a cofibrant $\capO$-algebra and $n\geq 1$, then the natural maps
\begin{align}
\label{eq:zigzag_of_codegeneracy_cubes}
  \capX_n\Iso\capX'_n\xleftarrow{\wequiv}\capX''_n\xrightarrow{\wequiv}\capX'''_n
\end{align}
of $n$-cubes are objectwise weak equivalences. If furthermore, $X$ is $0$-connected, then the iterated homotopy fiber of $\capX'''_n$ is $n$-connected, and hence the total homotopy fiber of $\capX_n$ is $n$-connected.
\end{prop}

\begin{proof}
The case $n=1$ is proved in Harper-Hess \cite[1.8]{Harper_Hess}. By using exactly the same arguments as in \cite[1.8]{Harper_Hess}, but for the objectwise calculation of the iterated fibers of the corresponding $n$-multisimplicial objects, it follows easily but tediously from the combinatorics in Propositions \ref{prop:calculations_for_iterated_fibers_part1}, \ref{prop:calculations_for_iterated_fibers_part2}, and  \ref{prop:calculations_for_iterated_fibers_part3} that the iterated homotopy fiber is weakly equivalent to the realization of a simplicial object which is $n$-connected in each simplicial degree; by Proposition \ref{prop:homotopy_spectral_sequence} the realization of such an object is $n$-connected. To finish off the proof, it suffices to verify the maps of cubes are objectwise weak equivalences; this follows from Propositions \ref{prop:homotopical_analysis_of_forgetful_functors} and \ref{prop:reedy_cofibrant_for_bar_constructions} below.
\end{proof}

\begin{rem}
For instance, consider the case $n=2$. Removing the outer realization in diagram \eqref{eq:diagram_n_equals_2_case_capX_2_cube}, evaluating at (horizontal) simplicial degree $r$, and calculating the fibers, together with the induced map $(*)$, we obtain the commutative diagram
\begin{align}
\label{eq:peeling_off_outer_realization_for_iterated_hofiber}
\xymatrix{
  \capO^{>1}\circ\capO^{\circ r}\circ Z'\ar@{.>}[d]^-{(*)}\ar[r]^-{\subset} &
  \capO\circ\capO^{\circ r}\circ Z'\ar[d]\ar[r] &
  \tau_1\capO\circ\capO^{\circ r}\circ Z'\ar[d]\\
  \capO^{>1}\circ\capO^{\circ r}\circ Z\ar[r]^-{\subset} &
  \capO\circ\capO^{\circ r}\circ Z\ar[r] &
  \tau_1\capO\circ\capO^{\circ r}\circ Z
}
\end{align}
where $Z':=|\BAR(\capO,\capO,X)|$, $Z:=|\BAR(\tau_1\capO,\capO,X)|$, and the rows are cofiber sequences in $\ModR$. We want to calculate the fiber of the map $(*)$. Removing the realization from $(*)$ in diagram \eqref{eq:peeling_off_outer_realization_for_iterated_hofiber} and evaluating at (vertical) simplicial degree $s$, we obtain the commutative diagram
\begin{align*}
\xymatrix{
  (\iter\hofib)_{r,s}\ar@{=}[r] &
  \bigl(\widetilde{W(r)}\bigr)^{\,>2}\circ\capO^{\circ s}\circ X\ar[d]\ar[r]^-{\subset} &
  W(r)\circ\capO\circ\capO^{\circ s}\circ X\ar[d]\\
  & {*}\ar[r] &
  W(r)\circ\tau_1\capO\circ\capO^{\circ s}\circ X
}
\end{align*}
where
\begin{align*}
  W(r)&:=\capO^{>1}\circ\capO^{\circ r}\\
\widetilde{W(r)} &:= \coprod\limits_{t>1}W(r)[t]
  \Smash_{\Sigma_t}(\capO^{\tensorcheck t})^{>t}
\end{align*}
The indicated square is a pushout diagram in $\ModR$ and a pullback diagram in $\AlgO$; here, we used Propositions \ref{prop:calculations_for_iterated_fibers_part2} and \ref{prop:calculations_for_iterated_fibers_part3}. Since $X$ is $0$-connected, it follows that $(\iter\hofib)_{r,s}$ is $2$-connected for each $r,s\geq 0$. Hence, applying realization in the vertical direction gives a (horizontal) simplicial $\capO$-algebra that is $2$-connected in every simplicial degree $r$, and therefore realization in the horizontal direction gives an $\capO$-algebra that is $2$-connected; hence the total homotopy fiber of $\capX'''_2$ is $2$-connected.
\end{rem}

The following two technical propositions are the reason for requiring $\capO$ to satisfy Cofibrancy Condition \ref{CofibrancyCondition}; together they imply that the maps of cubes in \eqref{eq:zigzag_of_codegeneracy_cubes} are objectwise weak equivalences (\cite[4.10]{Harper_Hess}).

\begin{prop}
\label{prop:homotopical_analysis_of_forgetful_functors}
Let $\capO$ be an operad in $\capR$-modules such that $\capO[r]$ is flat stable cofibrant in $\ModR$ for each $r\geq 0$.
\begin{itemize}
\item[(a)] If $\function{j}{A}{B}$ is a cofibration between cofibrant objects in $\AlgO$, then $j$ is a positive flat stable cofibration in $\ModR$.
\item[(b)] If $A$ is a cofibrant $\capO$-algebra and $\capO[0]=*$, then $A$ is positive flat stable cofibrant in $\ModR$.
\end{itemize}
\end{prop}

\begin{proof}
This is proved in Harper-Hess \cite[4.11]{Harper_Hess} using an important technical result verified in Pereira \cite{Pereira}. This is also proved independently in \cite{Pereira}, and in Pavlov-Scholbach \cite{Pavlov_Scholbach} under very general conditions.
\end{proof}

\begin{prop}
\label{prop:reedy_cofibrant_for_bar_constructions}
Let $\function{f}{\capO}{\capO'}$ be a morphism of operads in $\capR$-modules such that $\capO[\mathbf{0}]=*$. Assume that $\capO$ satisfies Cofibrancy Condition \ref{CofibrancyCondition}. Let $Y$ be an $\capO$-algebra (resp. left $\capO$-module) and consider the simplicial bar construction $\BAR(\capO',\capO,Y)$.
\begin{itemize}
\item[(a)] If $Y$ is positive flat stable cofibrant in $\ModR$, then $\BAR(\capO',\capO,Y)$ is Reedy cofibrant in $(\Alg_{\capO'})^{\Delta^\op}$.
\item[(b)] If $Y$ is positive flat stable cofibrant in $\ModR$, then $|\BAR(\capO',\capO,Y)|$ is cofibrant in $\Alg_{\capO'}$.
\end{itemize}
\end{prop}

\begin{proof}
This is proved in Harper-Hess \cite[4.19]{Harper_Hess} using an important technical result verified in Pereira \cite{Pereira}. This is also proved independently in \cite{Pereira}, and in Pavlov-Scholbach \cite{Pavlov_Scholbach} under very general conditions.
\end{proof}

The following proposition gives the connectivity estimates that we need.

\begin{prop}[Proposition \ref{prop:connectivities_for_total_homotopy_fiber_homotopy_completion_cube_no_zigzags} restated]
\label{prop:connectivities_for_total_homotopy_fiber_homotopy_completion_cube_no_zigzags_other_section}
Let $X$ be a cofibrant $\capO$-algebra and $n\geq 1$. Denote by $\capX_n$ the coface $n$-cube associated to the cosimpicial $\TQ$-homology resolution $X\rarrow C(QX)$ of $X$. If $X$ is $0$-connected, then the total homotopy fiber of $\capX_n$ is $n$-connected.
\end{prop}

\begin{proof}
This is a special case of Proposition \ref{prop:connectivities_for_total_homotopy_fiber_homotopy_completion_cube}.
\end{proof}

\begin{prop}[Proposition \ref{prop:iterated_homotopy_fibers_calculation} restated]
\label{prop:iterated_homotopy_fibers_calculation_other_section}
Let $Z$ be a cosimplicial $\capO$-algebra and $n\geq 0$. There are natural zigzags of weak equivalences
\begin{align*}
  \hofib(\holim_{\Delta^{\leq n}}Z\rarrow\holim_{\Delta^{\leq n-1}}Z)
  \wequiv\Omega^n(\iter\hofib)\capY_n
\end{align*}
where $\capY_n$ denotes the canonical $n$-cube built from the codegeneracy maps of
\begin{align*}
\xymatrix{
  Z^0 &
  Z^1
  \ar[l]_-{s^0} &
  Z^2\ar@<-0.5ex>[l]_-{s^0}\ar@<0.5ex>[l]^-{s^1}
  \ \cdots\ Z^n
}
\end{align*}
 the $n$-truncation of $Z$; in particular, $\capY_0$ is the object (or $0$-cube) $Z^0$. Here, $\Omega^n$ is weakly equivalent, in the underlying category $\ModR$, to the $n$-fold desuspension $\Sigma^{-n}$ functor. We often refer to $\capY_n$ as the \emph{codegeneracy} $n$-cube associated to $Z$.
\end{prop}

\begin{proof}
This is proved in Bousfield-Kan \cite[X.6.3]{Bousfield_Kan} for the $\Tot$ tower of a Reedy fibrant cosimplicial pointed space, and the same argument verifies it in our context. It follows easily from the cosimplicial identities, together with repeated application of Proposition \ref{prop:retraction_two_cube_argument} by using the fact that codegeneracy maps provide retractions for the appropriate coface maps.
\end{proof}

For each $n\geq 1$, denote by $\capY_n$ the codegeneracy $n$-cube associated to $C(Y)$. The reason for introducing $\capY'_n$, $\capY''_n$, $\capY'''_n$, $\capY''''_n$ in Definition \ref{defn:sequence_of_weakly_equivalent_codegeneracy_cubes} is that we can homotopically analyze the iterated homotopy fibers of the $n$-cube $\capY''''_n$ using the combinatorics developed above; since these intermediate $n$-cubes fit into a zigzag of weak equivalences $\capY_n\wequiv\capY_n''''$ of $n$-cubes, the combinatorial analysis results in a homotopical analysis of $\capY_n$ (Proposition \ref{prop:multisimplicial_calculation_of_iterated_hofiber_codegeneracy_cube}).

\begin{defn}
\label{defn:sequence_of_weakly_equivalent_codegeneracy_cubes}
If $Y$ is a cofibrant $\K$-coalgebra and $n\geq 1$, denote by $\capY_n$ the codegeneracy $n$-cube associated to $C(Y)$, and by $\capY'_n$, $\capY''_n$, $\capY'''_n$, $\capY''''_n$ the following weakly equivalent $n$-cubes. The functor $\function{\capY'_n}{\Cube_{n}}{\AlgJ}$ is defined by
\begin{align*}
  (v_1,v_2,\dotsc,v_n)&\longmapsto J\circ_{A_1}J\circ_{A_2}\cdots J\circ_{A_n}(Y)\\
  \text{where}
  \quad\quad
  A_i &:=
  \left\{
    \begin{array}{ll}
    \capO,&\text{for $v_i=0$,}\\
    J,&\text{for $v_i=1$,}
  \end{array}
  \right.
\end{align*}
with maps induced by $\capO\rarrow J$; the functor $\function{\capY''_n}{\Cube_{n}}{\AlgJ}$ is defined by
\begin{align*}
  (v_1,v_2,\dotsc,v_n)&\longmapsto|\BAR(J,A_1,|\BAR(J,A_2,\dots,|\BAR(J,A_n,Y)|\dotsb)|)|\\
  \text{where}
  \quad\quad
  A_i &:=
  \left\{
    \begin{array}{ll}
    \capO,&\text{for $v_i=0$,}\\
    J,&\text{for $v_i=1$,}
  \end{array}
  \right.
\end{align*}
with maps induced by $\capO\rarrow J$; the functor $\function{\capY'''_n}{\Cube_{n}}{\AlgJ}$ is defined by
\begin{align*}
  (v_1,v_2,\dotsc,v_n)&\longmapsto|\BAR(J,A_1,|\BAR(J,A_2,\dots,|\BAR(J,A_n,\tilde{Y})|\dotsb)|)|\\
  \text{where}
  \quad\quad
  A_i &:=
  \left\{
    \begin{array}{ll}
    \capO,&\text{for $v_i=0$,}\\
    J,&\text{for $v_i=1$,}
  \end{array}
  \right.
\end{align*}
with maps induced by $\capO\rarrow J$; the functor $\function{\capY''''_n}{\Cube_{n}}{\Alg_{\tau_1\capO}}$ is defined by
\begin{align*}
  (v_1,v_2,\dotsc,v_n)&\longmapsto|\BAR(\tau_1\capO,A_1,|\BAR(\tau_1\capO,A_2,\dots,|\BAR(\tau_1\capO,A_n,\tilde{Y})|\dotsb)|)|\\
  \text{where}
  \quad\quad
  A_i &:=
  \left\{
    \begin{array}{ll}
    \capO,&\text{for $v_i=0$,}\\
    \tau_1\capO,&\text{for $v_i=1$,}
  \end{array}
  \right.
\end{align*}
with maps induced by $\capO\rarrow\tau_1\capO$; here $\tilde{Y}:=\tau_1\capO\circ_J(Y)$. It is important to note that the natural map $Y\rightarrow\tilde{Y}$ in $\AlgJ$ is a weak equivalence since $Y\in\AlgJ$ is cofibrant.
\end{defn}

\begin{rem}
For instance, $\capY_2$ is the $2$-cube associated to $C(Y)$ of the form
\begin{align*}
\xymatrix@1{
  J\circ_\capO J\circ_\capO(Y)\ar[d]^-{s^1}\ar[r]^-{s^0} &
  J\circ_\capO(Y)\ar[d]^-{s^0}\\
  J\circ_\capO(Y)\ar[r]^-{s^0} & Y
}
\end{align*}
$\capY'_2$ is the associated $2$-cube of the form
\begin{align*}
\xymatrix@1{
  J\circ_\capO J\circ_\capO(Y)\ar[d]^-{s^1}\ar[r]^-{s^0} &
  J\circ_J J\circ_\capO(Y)\ar[d]^-{s^0}\\
  J\circ_\capO J\circ_J(Y)\ar[r]^-{s^0} &
  J\circ_J J\circ_J(Y)
}
\end{align*}
$\capY''_2$ is the associated $2$-cube of the form
\begin{align*}
\xymatrix@1{
  |\BAR(J,\capO,|\BAR(J,\capO,Y)|)|\ar[d]^-{s^1}\ar[r]^-{s^0} &
  |\BAR(J,J,|\BAR(J,\capO,Y)|)|\ar[d]^-{s^0}\\
  |\BAR(J,\capO,|\BAR(J,J,Y)|)|\ar[r]^-{s^0} &
  |\BAR(J,J,|\BAR(J,J,Y)|)|
}
\end{align*}
$\capY'''_2$ is the associated $2$-cube of the form
\begin{align*}
\xymatrix@1{
  |\BAR(J,\capO,|\BAR(J,\capO,\tilde{Y})|)|\ar[d]^-{s^1}\ar[r]^-{s^0} &
  |\BAR(J,J,|\BAR(J,\capO,\tilde{Y})|)|\ar[d]^-{s^0}\\
  |\BAR(J,\capO,|\BAR(J,J,\tilde{Y})|)|\ar[r]^-{s^0} &
  |\BAR(J,J,|\BAR(J,J,\tilde{Y})|)|
}
\end{align*}
and $\capY''''_2$ the associated $2$-cube of the form
\begin{align}
\label{eq:diagram_n_equals_2_case_capX_2_cube_codegeneracy}
\xymatrix@1{
  |\BAR(\tau_1\capO,\capO,|\BAR(\tau_1\capO,\capO,\tilde{Y})|)|
  \ar[d]^-{s^1}\ar[r]^-{s^0} &
  |\BAR(\tau_1\capO,\tau_1\capO,|\BAR(\tau_1\capO,\capO,\tilde{Y})|)|\ar[d]^-{s^0}\\
  |\BAR(\tau_1\capO,\capO,|\BAR(\tau_1\capO,\tau_1\capO,\tilde{Y})|)|\ar[r]^-{s^0} &
  |\BAR(\tau_1\capO,\tau_1\capO,|\BAR(\tau_1\capO,\tau_1\capO,\tilde{Y})|)|
}
\end{align}
\end{rem}

\begin{prop}
\label{prop:multisimplicial_calculation_of_iterated_hofiber_codegeneracy_cube}
If $Y$ is a cofibrant $\K$-coalgebra and $n\geq 1$, then the maps
\begin{align*}
  \capY_n\Iso\capY'_n
  \xleftarrow{\wequiv}\capY''_n
  \xrightarrow{\wequiv}\capY'''_n
  \xrightarrow{\wequiv}\capY''''_n
\end{align*}
of $n$-cubes are objectwise weak equivalences. If furthermore, $Y$ is $0$-connected, then the iterated homotopy fiber of $\capY''''_n$ is $2n$-connected, and hence the total homotopy fiber of $\capY_n$ is $2n$-connected.
\end{prop}

\begin{proof}
This is argued exactly as in the proof of Proposition \ref{prop:connectivities_for_total_homotopy_fiber_homotopy_completion_cube}.
By the objectwise calculation of the iterated fibers of the corresponding $n$-multisimplicial objects, it follows easily but tediously from the combinatorics in Propositions \ref{prop:calculations_for_iterated_fibers_part1}, \ref{prop:calculations_for_iterated_fibers_part2}, and  \ref{prop:calculations_for_iterated_fibers_part3} that the iterated homotopy fiber is weakly equivalent to the realization of a simplicial object which is $(2n-1)$-connected in each simplicial degree, and which is a point in simplicial degree $0$; by Proposition \ref{prop:homotopy_spectral_sequence} the realization of such an object is $2n$-connected. To finish off the proof, it suffices to verify that the maps of cubes are objectwise weak equivalences; this follows from Propositions \ref{prop:homotopical_analysis_of_forgetful_functors} and \ref{prop:reedy_cofibrant_for_bar_constructions}.
\end{proof}

\begin{rem}
For instance, consider the case $n=2$. Removing the outer realization in diagram \eqref{eq:diagram_n_equals_2_case_capX_2_cube_codegeneracy}, evaluating at (horizontal) simplicial degree $r$, and calculating the fibers, together with the induced map $(*)$, we obtain the commutative diagram
\begin{align}
\label{eq:peeling_off_outer_realization_for_iterated_hofiber_codeg_cube}
\xymatrix{
  \tau_1\capO\circ(\capO^{\circ r})^{>1}\circ Z'\ar@{.>}[d]^-{(*)}\ar[r]^-{\subset} &
  \tau_1\capO\circ(\capO^{\circ r})\circ Z'\ar[d]\ar[r] &
  \tau_1\capO\circ\tau_1(\capO^{\circ r})\circ Z'\ar[d]\\
  \tau_1\capO\circ(\capO^{\circ r})^{>1}\circ Z\ar[r]^-{\subset} &
  \tau_1\capO\circ(\capO^{\circ r})\circ Z\ar[r] &
  \tau_1\capO\circ\tau_1(\capO^{\circ r})\circ Z
}
\end{align}
where $Z':=|\BAR(\tau_1\capO,\capO,\tilde{Y})|$, $Z:=|\BAR(\tau_1\capO,\tau_1\capO,\tilde{Y})|$, $\tau_1(\capO^{\circ r})\Iso(\tau_1\capO)^{\circ r}$, and the rows are cofiber sequences in $\ModR$. We want to calculate the fiber of the map $(*)$. Removing the realization from $(*)$ in diagram \eqref{eq:peeling_off_outer_realization_for_iterated_hofiber_codeg_cube} and evaluating at (vertical) simplicial degree $s$, we obtain the commutative diagram
\begin{align*}
\xymatrix{
  (\iter\hofib)_{r,s}\ar@{=}[r] &
  \tau_1\capO\circ\bigl(\widetilde{W(r,s)}\bigr)^{\,>2}\circ\tilde{Y}\ar[d]\ar[r]^-{\subset} &
  \tau_1\capO\circ W(r,s)\circ\capO^{\circ s}\circ\tilde{Y}\ar[d]\\
  & {*}\ar[r] &
  \tau_1\capO\circ W(r,s)\circ\tau_1(\capO^{\circ s})\circ\tilde{Y}
}
\end{align*}
where
\begin{align*}
W(r,s)&:=(\capO^{\circ r})^{>1}\circ\tau_1\capO\\
\widetilde{W(r,s)} &:= \coprod\limits_{t>1}W(r)[t]
  \Smash_{\Sigma_t}((\capO^{\circ s})^{\tensorcheck t})^{>t}
\end{align*}
The indicated square is a pushout diagram in $\ModR$ and a pullback diagram in $\AlgO$; here, we used Propositions \ref{prop:calculations_for_iterated_fibers_part2} and \ref{prop:calculations_for_iterated_fibers_part3}. Since $\tilde{Y}$ is $0$-connected, it follows that $(\iter\hofib)_{r,s}$ is $2$-connected, $(\iter\hofib)_{0,s}=*$, and $(\iter\hofib)_{r,0}=*$, for each $r,s\geq 0$. In other words, the iterated homotopy fiber, after removing the horizontal and vertical realizations, is a bisimplicial $\capO$-algebra of the form (not showing the face or degeneracy maps)
\begin{align*}
\xymatrix@!0{
  \vdots \\
  {*} & \square & \square \\
  {*} & \square & \square \\
  {*} & {*} & {*} & \cdots\\
}
\end{align*}
where each box $\square$ indicates a $2$-connected $\capO$-algebra. Applying realization in the vertical direction gives a (horizontal) simplicial $\capO$-algebra that is $3$-connected in each simplicial degree $r$, and is the null object $*$ in simplicial degree $0$, and therefore realization in the horizontal direction gives an $\capO$-algebra that is $4$-connected; hence the total homotopy fiber of $\capY''''_2$ is $4$-connected.
\end{rem}

\begin{prop}[Proposition \ref{prop:multisimplicial_calculation_of_iterated_hofiber_codegeneracy_cube_no_zigzags} restated]
\label{prop:multisimplicial_calculation_of_iterated_hofiber_codegeneracy_cube_no_zigzags_other_section}
Let $Y$ be a cofibrant $\K$-coalgebra and $n\geq 1$. Denote by $\capY_n$ the the codegeneracy $n$-cube associated to the cosimplicial cobar construction $C(Y)$ of $Y$. If $Y$ is $0$-connected, then the total homotopy fiber of $\capY_n$ is $2n$-connected.
\end{prop}

\begin{proof}
This is a special case of Proposition \ref{prop:multisimplicial_calculation_of_iterated_hofiber_codegeneracy_cube}.
\end{proof}

\subsection{Commuting $\LL Q$ past homotopy limits over $\Delta^{\leq n}$}

The purpose of this section is to prove Theorem \ref{thm:connectivities_for_map_that_commutes_Q_into_inside_of_holim}, which provides connectivity estimates for the comparison map
$\LL Q\holim\nolimits_{\Delta^{\leq n}} C(Y)\rightarrow\holim\nolimits_{\Delta^{\leq n}} \LL Q\,C(Y)$.

The following definitions and constructions appear in Goodwillie \cite{Goodwillie_calculus_2} in the context of spaces, and will also be useful in our context when working with the spectral algebra higher Blakers-Massey theorems proved in Ching-Harper \cite{Ching_Harper}.

\begin{defn}[Indexing categories for cubical diagrams]
Let $W$ be a finite set and $\M$ a category.
\begin{itemize}
\item Denote by $\powerset(W)$ the poset of all subsets of $W$, ordered by inclusion $\subset$ of sets. We will often regard $\powerset(W)$ as the category associated to this partial order in the usual way; the objects are the elements of $\powerset(W)$, and there is a morphism $U\rarrow V$ if and only if $U\subset V$.
\item Denote by $\powerset_0(W)\subset\powerset(W)$ the poset of all nonempty subsets of $W$; it is the full subcategory of $\powerset(W)$ containing all objects except the initial object $\emptyset$.
\item A \emph{$W$-cube} $\capX$ in $\M$ is a $\powerset(W)$-shaped diagram $\capX$ in $\M$; in other words, a functor $\function{\capX}{\powerset(W)}{\M}$.
\end{itemize}
\end{defn}

\begin{rem}
If $n=|W|$ and $\capX$ is a $W$-cube in $\M$, we will sometimes refer to $\capX$ simply as an \emph{$n$-cube} in $\M$. In particular, a $0$-cube is an object in $\M$ and a $1$-cube is a morphism in $\M$.
\end{rem}

\begin{defn}[Faces of cubical diagrams]
Let $W$ be a finite set and $\M$ a category. Let $\capX$ be a $W$-cube in $\M$ and consider any subsets $U\subset V\subset W$. Denote by $\partial_U^V\capX$ the $(V-U)$-cube defined objectwise by
\begin{align*}
  T\mapsto(\partial_U^V\capX)_T:=\capX_{T\cup U},\quad\quad T\subset V-U.
\end{align*}
In other words, $\partial_U^V\capX$ is the $(V-U)$-cube formed by all maps in $\capX$ between $\capX_U$ and $\capX_V$. We say that $\partial_U^V\capX$ is a \emph{face} of $\capX$ of \emph{dimension} $|V-U|$.
\end{defn}

\begin{defn}
\label{defn:the_wide_tilde_construction}
Let $Z\in(\AlgO)^\Delta$ and $n\geq 0$. Assume that $Z$ is objectwise fibrant and denote by $\function{Z}{\capP_0([n])}{\ModR}$ the composite
\begin{align*}
  \capP_0([n])\rightarrow\Delta^{\leq n}
  \rightarrow\Delta
  \rightarrow\AlgO,
  \quad\quad
  T\longmapsto Z_T,
\end{align*}
or ``restriction'' to $\capP_0([n])$. The \emph{associated $\infty$-cartesian $(n+1)$-cube built from $Z$}, denoted $\function{\widetilde{Z}}{\capP([n])}{\AlgO}$, is defined objectwise by
\begin{align*}
  \widetilde{Z}_V :=
  \left\{
    \begin{array}{rl}
    \holim^\BK_{T\neq\emptyset}Z_T,&\text{for $V=\emptyset$,}\\
    Z_V,&\text{for $V\neq\emptyset$}.
    \end{array}
  \right.
\end{align*}
It is important to note (Proposition \ref{prop:punctured_cube_calculation_of_holim_truncated_delta}) that there are natural weak equivalences
\begin{align*}
  \holim_{\Delta^{\leq n}}Z\wequiv
  \holim\nolimits^\BK_{T\neq\emptyset}Z_T=\widetilde{Z}_\emptyset
\end{align*}
in $\AlgO$.
\end{defn}

The following proposition is motivated by Munson-Volic \cite[10.6.10]{Munson_Volic_book_project}, and is proved in essentially the same way; it is closely related to Sinha \cite[7.2]{Sinha_cosimplicial_models}.

\begin{prop}
\label{prop:comparing_faces_of_coface_cube_with_codegeneracy_cube}
Let $Z\in(\AlgO)^\Delta$ and $n\geq 0$. Assume that $Z$ is objectwise fibrant. Let $\emptyset\neq T\subset[n]$ and $t\in T$. Then there is a weak equivalence
\begin{align*}
  (\iter\hofib)\partial_{\{t\}}^T\widetilde{Z}\wequiv
  \Omega^{|T|-1}(\iter\hofib)\capY_{|T|-1}
\end{align*}
in $\AlgO$, where $\capY_{|T|-1}$ denotes the codegeneracy $(|T|-1)$-cube associated to $Z$. Here, $\Omega^k$ is weakly equivalent, in the underlying category $\ModR$, to the $k$-fold desuspension $\Sigma^{-k}$ functor for each $k\geq 0$.
\end{prop}

\begin{proof}
It follows easily from the cosimplicial identities that $\partial_{\{t\}}^T\widetilde{Z}$ is connected to $\capY_{|T|-1}$ by a sequence of retractions, built from codegeneracy maps, that have the following special property: the retractions fit into a commutative diagram made up of a sequence of concatenated $(|T|-1)$-cubes, starting with $\partial_{\{t\}}^T\widetilde{Z}$, such that each added cube is in a distinct direction (the concatenation direction of the added cube), and such that each arrow in the added cube, whose direction is in the concatenation direction, is a retraction of the preceding arrow in that direction (in the previous cube in the sequence), with the last $(|T|-1)$-cube in the sequence composed entirely of codegeneracy maps; this is elaborated more precisely in Remark \ref{rem:formalizing_the_coordinate_free_codegeneracy_maps}. The resulting commutative diagram is a sequence of $|T|$ concatenated $(|T|-1)$-cubes that starts with $\partial_{\{t\}}^T\widetilde{Z}$ and ends with $\capY_{|T|-1}$. Repeated application of Proposition \ref{prop:retraction_two_cube_argument} finishes the proof.
\end{proof}

Consider the following special case.

\begin{rem}
\label{rem:example_of_concatenation_sequences}
For instance, suppose $n=2$ and $T=\{0,1,2\}$. In the case $t=1$, consider the following left-hand commutative diagram
\begin{align}
\label{eq:path_of_2_cubes_cases_one_and_two}
\xymatrix@1{
  Z_{\{0\}}\ar[d]^-{d^1}\ar[r]^-{d^1}\ar@{}[dr]|(0.43){(*)'} &
  Z_{\{0,1\}}\ar[d]^-{d^2}\ar[r]^-{s^0} &
  Z_{\{0\}}\ar[d]^-{d^1}\\
  Z_{\{0,2\}}\ar[r]^-{d^1} &
  Z_{\{0,1,2\}}\ar[d]^-{s^1}\ar[r]^-{s^0} &
  Z_{\{0,2\}}\ar[d]^-{s^0}\\
& Z_{\{0,1\}}\ar[r]^-{s^0} &
  Z_{\{0\}}
}\quad\quad
\xymatrix@1{
  Z_{\{1\}}\ar[d]^-{d^1}\ar[r]^-{d^0}\ar@{}[dr]|(0.43){(*)''} &
  Z_{\{0,1\}}\ar[d]^-{d^2}\ar[r]^-{s^0} &
  Z_{\{1\}}\ar[d]^-{d^1}\\
  Z_{\{1,2\}}\ar[d]^-{s^0}\ar[r]^-{d^0} &
  Z_{\{0,1,2\}}\ar[d]^-{s^1}\ar[r]^-{s^0} &
  Z_{\{1,2\}}\ar[d]^-{s^0}\\
  Z_{\{1\}}\ar[r]^-{d^0} &
  Z_{\{0,1\}}\ar[r]^-{s^0} &
  Z_{\{1\}}
}
\end{align}
in the case $t=2$, consider the right-hand commutative diagram in \eqref{eq:path_of_2_cubes_cases_one_and_two},
and in the case $t=3$, consider the commutative diagram
\begin{align*}
\xymatrix@1{
  Z_{\{2\}}\ar[d]^-{d^0}\ar[r]^-{d^0}\ar@{}[dr]|(0.43){(*)'''} &
  Z_{\{0,2\}}\ar[d]^-{d^1}\\
  Z_{\{1,2\}}\ar[d]^-{s^0}\ar[r]^-{d^0} &
  Z_{\{0,1,2\}}\ar[d]^-{s^1}\ar[r]^-{s^0} &
  Z_{\{1,2\}}\ar[d]^-{s^0}\\
  Z_{\{2\}}\ar[r]^-{d^0} &
  Z_{\{0,2\}}\ar[r]^-{s^0} &
  Z_{\{2\}}
}
\end{align*}
The upper left-hand squares $(*)'$, $(*)''$, $(*)'''$ are the $2$-cubes $\partial_{\{0\}}^{T}\widetilde{Z}$, $\partial_{\{1\}}^{T}\widetilde{Z}$, $\partial_{\{2\}}^{T}\widetilde{Z}$, respectively, and the lower right-hand squares are each a copy of the codegeneracy $2$-cube $\capY_2$. By repeated application of Proposition \ref{prop:retraction_two_cube_argument}, it follows easily that there are weak equivalences
\begin{align*}
  (\iter\hofib)\partial_{\{t\}}^T\widetilde{Z}\wequiv
  \Omega^2(\iter\hofib)\capY_2
\end{align*}
in $\AlgO$ for each $t\in T$. Note that there are two distinct paths connecting $\partial_{\{1\}}^{T}\widetilde{Z}$ with $\capY_2$, but there is only one path connecting $\partial_{\{0\}}^{T}\widetilde{Z}$ with $\capY_2$; similarly, there is only one path connecting $\partial_{\{2\}}^{T}\widetilde{Z}$ with $\capY_2$. The proof of Proposition \ref{prop:comparing_faces_of_coface_cube_with_codegeneracy_cube} is simply the observation, which follows easily from the cosimplicial identities, that one such path always exists.
\end{rem}

\begin{rem}
\label{rem:formalizing_the_coordinate_free_codegeneracy_maps}
The construction in the proof of Proposition \ref{prop:comparing_faces_of_coface_cube_with_codegeneracy_cube} can be elaborated as follows. Let $n\geq 1$. Suppose that $\emptyset\neq T\subset[n]$ and $t\in T$. For each $\emptyset\neq S\subset T$, denote by $\capP_S^T$ the poset of all subsets of $T$ that contain $S$, ordered by inclusion $\subset$ of sets; it is the full subcategory of $\capP_0(T)\subset\capP_0([n])$ containing all objects that contain the set $S$. Define $\widetilde{Z}_S^T$ to be the restriction of $\widetilde{Z}$ to $\capP_S^T$; it is important to note that $\widetilde{Z}_S^T$ is a copy of the cube $\partial_S^T\widetilde{Z}$.

Denote by $\Delta_\mathrm{large}$ the category of nonempty totally-ordered finite sets and order-preserving maps; note that $\Delta\subset\Delta_\mathrm{large}$ is a skeletal subcategory and there is an equivalence of categories $\Delta_\mathrm{large}\rarrow\Delta$. Consider the functor $\function{\sigma}{\bigl(\capP_{\{t\}}^T\bigr)^\op}{\Delta_\mathrm{large}}$ which on objects is defined by $\sigma(V):=V$ and on arrows is defined as follows. It maps each $V\subset W$ in $\capP_{\{t\}}^T$ to the order-preserving function $\function{\sigma_{W,V}}{W}{V}$ defined by
\begin{align*}
  \sigma_{W,V}(w) :=
  \left\{
    \begin{array}{rl}
    \max\{v\in V: v\leq w\},&\text{for $t\leq w$,}\\
    \min\{v\in V: v\geq w\},&\text{for $t\geq w$}.
    \end{array}
  \right.
\end{align*}
Note that $\sigma_{W,V}(v)=v$ for each $v\in V$. Define $\capY_{\{t\}}^T$ to be the composite
\begin{align*}
  \bigl(\capP_{\{t\}}^T\bigr)^\op\xrightarrow{\ \sigma\ }
  \Delta_\mathrm{large}\rarrow\Delta\xrightarrow{Z}\AlgO
\end{align*}
Then $\capY_{\{t\}}^T$ is a $(|T|-1)$-cube with the same vertices as $\widetilde{Z}_{\{t\}}^T$, but with edges going in the opposite direction; it is important to note that $\capY_{\{t\}}^T$ is a copy of the codegeneracy cube $\capY_{|T|-1}$.

By construction, for each $V\subset W$ in $\capP_{\{t\}}^T$ the composite
\begin{align*}
  \bigl(\widetilde{Z}_{\{t\}}^T\bigr)_{V}\longrightarrow
  \bigl(\widetilde{Z}_{\{t\}}^T\bigr)_{W}=
  \bigl(\capY_{\{t\}}^T\bigr)_{W}\longrightarrow
  \bigl(\capY_{\{t\}}^T\bigr)_{V}\Equal
  \bigl(\widetilde{Z}_{\{t\}}^T\bigr)_{V}
\end{align*}
is the identity map; in other words, we have built a collection of retracts from the  coordinate free description of the codegeneracy maps. It follows from this construction that $\widetilde{Z}_{\{t\}}^T$ is connected to $\capY_{\{t\}}^T$ by a sequence of retractions, built from codegeneracy maps.  The resulting commutative diagram is a sequence of $|T|$ concatenated $(|T|-1)$-cubes that starts with $\widetilde{Z}_{\{t\}}^T$ and ends with $\capY_{\{t\}}^T$.
\end{rem}

\begin{prop}
\label{prop:base_case_of_induction_argument_connectivity}
Let $n\geq 0$. Suppose that $\emptyset\neq T\subset[n]$ and $t\in T$. If $Y\in\coAlgK$ is cofibrant and $0$-connected, then the cube
\begin{align*}
  \partial_{\{t\}}^T\widetilde{\mathfrak{C}(Y)}\quad
  \text{is $|T|$-cartesian.}
\end{align*}
\end{prop}

\begin{proof}
We know by Proposition \ref{prop:multisimplicial_calculation_of_iterated_hofiber_codegeneracy_cube} that the iterated homotopy fiber of $\capY_{|T|-1}$ is $2(|T|-1)$-connected. Hence by Proposition \ref{prop:comparing_faces_of_coface_cube_with_codegeneracy_cube} we know that the iterated homotopy fiber of $\partial_{\{t\}}^T\widetilde{\mathfrak{C}(Y)}$ is $(|T|-1)$-connected; it follows that $\partial_{\{t\}}^T\widetilde{\mathfrak{C}(Y)}$ is $|T|$-cartesian.
\end{proof}

\begin{prop}
\label{prop:estimates_for_the_face_from_S_to_T_of_associated_cube}
Let $n\geq 0$. Suppose that $\emptyset\neq T\subset[n]$ and $\emptyset\neq S\subset T$. If $Y\in\coAlgK$ is cofibrant and $0$-connected, then the cube
\begin{align}
\label{eq:connectivity_of_desired_face}
  \partial_S^T\widetilde{\mathfrak{C}(Y)}\quad
  \text{is $(|T|-|S|+1)$-cartesian}.
\end{align}
\end{prop}

\begin{proof}
We want to verify \eqref{eq:connectivity_of_desired_face} for each $\emptyset\neq S\subset T$. We know that \eqref{eq:connectivity_of_desired_face} is true for $|S|=1$ by Proposition \ref{prop:base_case_of_induction_argument_connectivity}. We will argue by upward induction on $|S|$. Let $t\in S$ and note that the cube $\partial_{S-\{t\}}^T\widetilde{\mathfrak{C}(Y)}$ can be written as the composition of cubes
\begin{align*}
  \partial_{S-\{t\}}^{T-\{t\}}\widetilde{\mathfrak{C}(Y)}\longrightarrow
  \partial_S^T\widetilde{\mathfrak{C}(Y)}
\end{align*}
We know by the induction hypothesis that the composition is $(|T|-|S|+2)$-cartesian and that the left-hand cube is $(|T|-|S|+1)$-cartesian, hence it follows (Ching-Harper \cite[3.8, 3.10]{Ching_Harper}) that the right-hand cube is $(|T|-|S|+1)$-cartesian, which finishes the argument that \eqref{eq:connectivity_of_desired_face} is true for each $\emptyset\neq S\subset T$.
\end{proof}

\begin{thm}
\label{thm:cocartesian_and_cartesian_estimates}
Let $Y\in\coAlgK$ be cofibrant and $n\geq 1$. Consider the $\infty$-cartesian $(n+1)$-cube $\widetilde{\mathfrak{C}(Y)}$ in $\AlgO$ built from $\mathfrak{C}(Y)$. If $Y$ is $0$-connected, then
\begin{itemize}
\item[(a)] the cube $\widetilde{\mathfrak{C}(Y)}$ is $(2n+4)$-cocartesian in $\AlgO$,
\item[(b)] the cube $FQc\widetilde{\mathfrak{C}(Y)}$ is $(2n+4)$-cocartesian in $\AlgJ$,
\item[(c)] the cube $FQc\widetilde{\mathfrak{C}(Y)}$ is $(n+4)$-cartesian in $\AlgJ$.
\end{itemize}
\end{thm}

\begin{proof}
Consider part (a) and let $W=[n]$. We want to use the higher dual Blakers-Massey theorem for structured ring spectra in Ching-Harper \cite[1.11]{Ching_Harper} to estimate how close the $W$-cube $\widetilde{\mathfrak{C}(Y)}$ in $\AlgO$ is to being cocartesian. We know from Proposition \ref{prop:estimates_for_the_face_from_S_to_T_of_associated_cube} that for each nonempty subset $V\subset W$, the $V$-cube $\partial_{W-V}^W\widetilde{\mathfrak{C}(Y)}$ is $(|V|+1)$-cartesian; it is $\infty$-cartesian by construction when $V=W$. Hence it follows from \cite[1.11]{Ching_Harper} that $\widetilde{\mathfrak{C}(Y)}$ is $(2(n+1)+2)$-cocartesian in $\AlgO$, which finishes the proof of part (a). Part (b) follows from the fact that $\function{Q}{\AlgO}{\AlgJ}$ is a left Quillen functor together with the relative $\TQ$-Hurewicz theorem in Harper-Hess \cite[1.9]{Harper_Hess}. Part (c) follows from the fact that $\AlgJ$ and $\Alg_{\tau_1\capO}\Iso\Mod_{\capO[1]}$ are Quillen equivalent (Section \ref{sec:TQ_homology_completion}) via the change of operads adjunction along $J\rarrow\tau_1\capO$, together with Ching-Harper \cite[3.10]{Ching_Harper}.
\end{proof}

\begin{proof}[Proof of Theorem \ref{thm:connectivities_for_map_that_commutes_Q_into_inside_of_holim}]
We want to verify that the comparison map
\begin{align*}
  \LL Q\holim\nolimits_{\Delta^{\leq n}} C(Y)\longrightarrow
  \holim\nolimits_{\Delta^{\leq n}} \LL Q\,C(Y),
\end{align*}
is $(n+4)$-connected, which is equivalent to verifying that $\LL Q\widetilde{\mathfrak{C}(Y)}$ is $(n+4)$-cartesian.  Since $\LL Q\wequiv FQc$, Theorem \ref{thm:cocartesian_and_cartesian_estimates}(c) completes the proof.
\end{proof}

\section{Proofs}
\label{sec:proofs}

The purpose of this section is to prove Propositions \ref{prop:useful_properties_of_the_adjunction}, \ref{prop:unit_and_counit_are_simplicial}, \ref{prop:basic_properties_of_Tot_towers_etc}, \ref{prop:tot_commutes_with_realization}, and \ref{prop:holim_commutes_with_realization} that were used in this paper.

\begin{proof}[Proof of Proposition \ref{prop:useful_properties_of_the_adjunction}]
Part (c) follows from the observation that the underlying $\capR$-modules are identical and the $\capO$-action maps are the same. Consider part (a). This follows easily from part (c), together with the Yoneda lemma, by verifying there are natural isomorphisms
\begin{align*}
  \hom(f_*(X\tensordot K),Y)\xrightarrow[\Iso]{\varphi}
  \hom(f_*(X)\tensordot K,Y).
\end{align*}
In particular, $\sigma$ is the image under $\varphi$ of the identity map on $f_*(X\tensordot K)$ and $\sigma^{-1}$ is the image under $\varphi^{-1}$ of the identity map on $f_*(X)\tensordot K$. Consider part (b). The indicated isomorphism of simplicial sets is defined objectwise by the composition of natural isomorphisms
\begin{align*}
  \hom(f_*(X)\tensordot\Delta[n],Y)\Iso
  \hom(f_*(X\tensordot\Delta[n]),Y)\Iso
  \hom(X\tensordot\Delta[n],f^*(Y)),
\end{align*}
where the left-hand isomorphism is the map $(\sigma^{-1},\id)$. Consider part (d). The map $\sigma$ is the map induced by $\function{f\circ\id}{\capO\circ(Y\Smash K_+)}{\capO'\circ(Y\Smash K_+)}$ via the colimit description of tensor product in \eqref{eq:tensordot_definition_as_coequalizer}; it will be useful to note that $\sigma$ can also be described as the image of the identity map on $Y\tensordot K$ under the composition of maps
\begin{align*}
  \hom(Y\tensordot K,Y')&\Iso
  \hom(Y,\hombold(K,Y'))\xrightarrow{f^*}
  \hom\bigl(f^*(Y),\hombold(K,f^*(Y'))\bigr)\\
  &\Iso
  \hom(f^*(Y)\tensordot K,f^*(Y'))
\end{align*}
Here, we used the identification in part (c). Consider part (e). It suffices to verify that corresponding diagrams (1) and (2) in \cite[9.8.5]{Hirschhorn} commute. In the case of $f_*$, this follows from the fact that the structure maps are the canonical isomorphisms. Consider the case of $f^*$. Verifying that diagram (1) in \cite[9.8.5]{Hirschhorn} commutes follows most easily from using the colimit description in \eqref{eq:tensordot_definition_as_coequalizer}, and verifying that the following diagram
\begin{align}
\label{eq:sigma_respects_composition}
\xymatrix{
  f^*\bigl(Y\tensordot(K\times L)\bigr) &
  f^*(Y)\tensordot(K\times L)\ar[l]_-{\sigma} &
  \bigl(f^*(Y)\tensordot K\bigr)\tensordot L
  \ar_-{\Iso}[l]\ar[d]^-{\sigma\tensordot\id}\\
  f^*\bigl((Y\tensordot K)\tensordot L\bigr)\ar[u]^-{\Iso} &&
  f^*(Y\tensordot K)\tensordot L\ar[ll]_-{\sigma}
}
\end{align}
commutes---this corresponds to diagram (2) in \cite[9.8.5]{Hirschhorn}---follows similarly from \eqref{eq:tensordot_definition_as_coequalizer}, together with the fact that $f$ respects the operad multiplication maps $\capO\circ\capO\rarrow\capO$ and $\capO'\circ\capO'\rarrow\capO'$.
\end{proof}

\begin{proof}[Proof of Proposition \ref{prop:unit_and_counit_are_simplicial}]
Consider the case of the unit map. By Proposition \ref{prop:characterization_of_simplicial_natural_transformations}, it suffices to verify that the diagram
\begin{align*}
\xymatrix{
  X\tensordot K\ar[d]^-{\eta_X\tensordot\id}\ar@{=}[rr] &&
  X\tensordot K\ar[d]^-{\eta_{X\tensordot K}}\\
  f^*f_*(X)\tensordot K\ar[r]_-{\sigma} &
  f^*(f_*(X)\tensordot K) &
  f^*f_*(X\tensordot K)\ar[l]_-{\Iso}^-{f^*(\sigma^{-1})}
}
\end{align*}
commutes; hence it suffices to verify that the two composite maps of the form
$X\tensordot K\rightrightarrows f^*(f_*(X)\tensordot K)$
are identical. This follows by working with the hom-set description of the indicated $\sigma^{-1}$ and $\sigma$ maps (see the proof of Proposition \ref{prop:useful_properties_of_the_adjunction}) and noting that each arrow is the image of the identity map on $f_*(X)\tensordot K$ under the composition of natural isomorphisms
\begin{align*}
  \hom(f_*(X)\tensordot K,Y)&\Iso
  \hom(f_*(X),\hombold(K,Y))\Iso
  \hom(X,f^*\hombold(K,Y)\\
  &\Equal
  \hom(X,\hombold(K,f^*Y)\Iso
  \hom(X\tensordot K,f^*Y).
\end{align*}
Consider the case of the counit map. It suffices to verify that the diagram
\begin{align*}
\xymatrix{
  f_*f^*(Y)\tensordot K\ar[d]^-{\varepsilon_Y\tensordot\id} &
  f_*(f^*(Y)\tensordot K)\ar[r]^-{f_*(\sigma)}\ar[l]_-{\sigma^{-1}}^-{\Iso} &
  f_*f^*(Y\tensordot K)\ar[d]^-{\varepsilon_{Y\tensordot K}}\\
  Y\tensordot K\ar@{=}[rr] &&
  Y\tensordot K
}
\end{align*}
commutes; this follows, similar to above, by noting that the two composite maps of the form $f_*(f^*(Y)\tensordot K)\rightrightarrows Y\tensordot K$ are each equal to the image of the identity map on $Y\tensordot K$ under the composition of natural maps
\begin{align*}
  \hom(Y\tensordot K, Y')&\Iso
  \hom(Y,\hombold(K,Y'))\xrightarrow{f^*}
  \hom\bigr(f^*(Y),\hombold(K,f^*(Y'))\bigl)\\
  &\Iso
  \hom(f^*(Y)\tensordot K,f^*(Y'))\Iso
  \hom(f_*(f^*(Y)\tensordot K),Y')
\end{align*}
Here we have used the identification in Proposition \ref{prop:useful_properties_of_the_adjunction}(c). The remaining two cases follow from the fact that the multiplication map is $f^*\varepsilon f_*$, the comultiplication map is $f_*\eta f^*$, and composing a simplicial natural transformation with a simplicial functor, on the left or right, gives a simplicial natural transformation.
\end{proof}

\begin{proof}[Proof of Proposition \ref{prop:basic_properties_of_Tot_towers_etc}]
Consider part (a). By definition, $X$ is Reedy fibrant if the natural map $X^s\rarrow M^{s-1}X$ is a fibration for each $s\geq 0$. It follows from the pullback diagrams in Proposition \ref{prop:key_pullback_diagram_for_Tot_tower} that each map $\Tot_s(X)\rarrow\Tot_{s-1}(X)$ is a fibration, and hence each natural map $\Tot(X)\rarrow\Tot_s(X)$ is a fibration. Part (b) follows similarly from the pullback diagrams in Proposition \ref{prop:key_pullback_diagram_for_restricted_Tot_tower}. Consider part (c). Since $X\rarrow X'$ is a weak equivalence between Reedy fibrant objects, it follows from Goerss-Jardine \cite[VIII.1]{Goerss_Jardine} that $M^{s-1}X\rarrow M^{s-1}X'$ is a weak equivalence between fibrant objects. Hence by  Proposition \ref{prop:key_pullback_diagram_for_Tot_tower} the induced map $\{\Tot_s X\}\rarrow\{\Tot_s X'\}$ is an objectwise weak equivalence between towers of fibrations, and applying the limit functor $\lim_s$ finishes the proof that $\Tot(X)\rarrow\Tot(X')$ is a weak equivalence. Part (d) follows similarly by using Proposition \ref{prop:key_pullback_diagram_for_restricted_Tot_tower} instead of Proposition \ref{prop:key_pullback_diagram_for_Tot_tower}.
\end{proof}

\subsection{Proofs: totalization and $\holim_\Delta$ commute with realization}

The purpose of this section is to prove Propositions \ref{prop:tot_commutes_with_realization} and \ref{prop:holim_commutes_with_realization}.

Recall that the realization functor fits into an adjunction
\begin{align}
\label{eq:realization_adjunction_sSet_CGHaus}
&\xymatrix{
  \sSet
  \ar@<0.5ex>[r]^-{|-|} &
  \CGHaus,\ar@<0.5ex>[l]^-{\Sing}
}
\end{align}
with left adjoint on top and right adjoint the singular simplicial set functor defined objectwise by $\Sing Y:=\hom_\CGHaus(\Delta^{(-)},Y)$; see, for instance, \cite[I.1, II.3]{Goerss_Jardine}. In addition to the fact that  \eqref{eq:realization_adjunction_sSet_CGHaus} is a Quillen equivalence (see, for instance, \cite{Goerss_Jardine}), the following properties of the realization functor will also be important.

\begin{prop}
\label{prop:nice_properties_of_realization_into_CGHaus}
The realization functor $\function{|-|}{\sSet}{\CGHaus}$
\begin{itemize}
\item[(a)] commutes with finite limits,
\item[(b)] preserves fibrations,
\item[(c)] preserves weak equivalences.
\end{itemize}
\end{prop}

\begin{prop}
\label{prop:realization_preserves_reedy_fibrant_objects}
If $Y\in(\sSet)^\Delta$ is Reedy fibrant, then $|Y|\in(\CGHaus)^\Delta$ is Reedy fibrant.
\end{prop}

\begin{proof}
By assumption the canonical map $Y^{s+1}\rarrow M^s Y$ into the indicated matching object is a fibration for each $s\geq -1$. Since realization preserves finite limits and fibrations (Proposition \ref{prop:nice_properties_of_realization_into_CGHaus}), it follows that $|M^s Y|\Iso M^s|Y|$ and the natural map $|Y|^{s+1}\rarrow M^s|Y|$ is a fibration for each $s\geq -1$ which finishes the proof.
\end{proof}

\begin{prop}
\label{prop:sing_commutes_with_tot}
Let $Z\in(\CGHaus)^\Delta$. There is a natural isomorphism
\begin{align*}
  \Tot\Sing Z\Iso\Sing\Tot Z
\end{align*}
\end{prop}

\begin{proof}
Consider the following diagram of adjunctions
\begin{align*}
\xymatrix{
  (\sSet)^\Delta\ar@<0.5ex>[r]^-{|-|}\ar@<0.5ex>[d]^-{\Tot} &
  (\CGHaus)^\Delta\ar@<0.5ex>[l]^-{\Sing}\ar@<0.5ex>[d]^-{\Tot}\\
  \sSet\ar@<0.5ex>[r]^-{|-|}\ar@<0.5ex>[u]^-{-\times\Delta[-]} &
  \CGHaus\ar@<0.5ex>[l]^-{\Sing}\ar@<0.5ex>[u]^-{-\times\Delta^{(-)}}
}
\end{align*}
with left adjoints on top and on the left. Since $|X\times\Delta[-]|\Iso|X|\times\Delta^{(-)}$, this diagram commutes up to natural isomorphism, and hence uniqueness of right adjoints (up to isomorphism) finishes the proof.
\end{proof}

\begin{prop}
\label{prop:sing_preserves_reedy_fibrant_objects}
If $Z\in(\CGHaus)^\Delta$ is Reedy fibrant, then $\Sing Z\in(\sSet)^\Delta$ is Reedy fibrant.
\end{prop}

\begin{proof}
Arguing exactly as in the proof of Proposition \ref{prop:realization_preserves_reedy_fibrant_objects}, this is because $\Sing$ preserves limits and fibrations.
\end{proof}

\begin{prop}
\label{prop:sing_of_realization_of_reedy_fibrant_is_reedy_fibrant}
If $Y\in(\sSet)^\Delta$ is Reedy fibrant, then $\Sing|Y|\in(\sSet)^\Delta$ is Reedy fibrant.
\end{prop}

\begin{proof}
This follows from Propositions \ref{prop:realization_preserves_reedy_fibrant_objects} and \ref{prop:sing_preserves_reedy_fibrant_objects}.
\end{proof}

\begin{prop}
\label{prop:working_out_the_weak_equivalences}
If $Y\in(\sSet)^\Delta$ is Reedy fibrant, then the natural maps
\begin{align}
\label{eq:into_tot_sing_realization}
  \Tot Y\xrightarrow{\wequiv}&\Tot\Sing|Y|\\
  \label{eq:into_realization_tot_sing_realization}
  |\Tot Y|\xrightarrow{\wequiv}&|\Tot\Sing(|Y|)|\\
  \label{eq:into_tot_realization}
  |\Sing\Tot(|Y|)|\xrightarrow{\wequiv}&\Tot|Y|
\end{align}
are weak equivalences.
\end{prop}

\begin{proof}
We know that the natural map $Y\rarrow\Sing|Y|$ is a weak equivalence, since \eqref{eq:realization_adjunction_sSet_CGHaus} is a Quillen equivalence. Since $\Tot$ sends weak equivalences between Reedy fibrant objects to weak equivalences,  Proposition \ref{prop:sing_of_realization_of_reedy_fibrant_is_reedy_fibrant} verifies that \eqref{eq:into_tot_sing_realization} is a weak equivalence. It follows from \eqref{eq:into_tot_sing_realization} that the map \eqref{eq:into_realization_tot_sing_realization} is a weak equivalence since realization preserves weak equivalences. Finally, since $|Y|$ is Reedy fibrant (Proposition \ref{prop:realization_preserves_reedy_fibrant_objects}), we know that $\Tot|Y|$ is fibrant, and hence it follows that \eqref{eq:into_tot_realization} is a weak equivalence since \eqref{eq:realization_adjunction_sSet_CGHaus} is a Quillen equivalence.
\end{proof}

\begin{proof}[Proof of Proposition \ref{prop:tot_commutes_with_realization}]
The following concise line of argument is suggested in Dwyer \cite{Dwyer}.
Consider the commutative diagram
\begin{align*}
\xymatrix{
  |\Tot Y|\ar[d]^-{\wequiv}\ar[r] & \Tot|Y|\\
  |\Tot\Sing(|Y|)|\ar[r]^-{\Iso} & |\Sing\Tot (|Y|)|\ar[u]_-{\wequiv}
}
\end{align*}
By Propositions \ref{prop:working_out_the_weak_equivalences} and \ref{prop:sing_commutes_with_tot} the vertical maps are weak equivalences and the bottom map is an isomorphism, hence the top map is weak equivalence. A similar argument verifies the $\Tot^\res$ case; by replacing Reedy fibrant with objectwise fibrant everywhere.
\end{proof}

\begin{prop}
\label{prop:objectwise_fibrant_to_reedy_fibrant}
If $Y\in(\sSet)^\Delta$ is objectwise fibrant, then the cosimplicial replacements $\prod^*Y\in(\sSet)^\Delta$ and $\prod^*|Y|\in(\CGHaus)^\Delta$ are Reedy fibrant.
\end{prop}

\begin{proof}
This is an exercise left to the reader.
\end{proof}

\begin{prop}
\label{prop:realization_commutes_with_cosimplicial_replacement_weak_equivalence}
If $Y\in(\sSet)^\Delta$ is objectwise fibrant, then the natural map
\begin{align*}
  \xymatrix{|\prod^* Y|\xrightarrow{\wequiv}\prod^*|Y|}
\end{align*}
is a weak equivalence.
\end{prop}

\begin{proof}
This follows by arguing exactly as in the proof of Proposition \ref{prop:tot_commutes_with_realization} above.
\end{proof}

\begin{proof}[Proof of Proposition \ref{prop:holim_commutes_with_realization}]
The map \eqref{eq:natural_map_realization_into_holim_nice} factors as
\begin{align*}
  \xymatrix{|\Tot\prod^* Y|\rarrow
  \Tot|\prod^* Y|\rarrow
  \Tot\prod^* |Y|
  }
\end{align*}
The left-hand map is a weak equivalence by Propositions \ref{prop:objectwise_fibrant_to_reedy_fibrant} and \ref{prop:tot_commutes_with_realization}, and the right-hand map is a weak equivalence by Propositions \ref{prop:objectwise_fibrant_to_reedy_fibrant} and \ref{prop:realization_commutes_with_cosimplicial_replacement_weak_equivalence}, since $\Tot$ sends weak equivalences between Reedy fibrant objects to weak equivalences; hence the composition is a weak equivalence.
\end{proof}

\bibliographystyle{plain}
\bibliography{DerivedKoszulDuality.bib}

\end{document}